\documentclass{article}
\usepackage[utf8]{inputenc}
\usepackage{lineno,epsfig,graphics,amssymb,amsmath,graphicx,pifont,soul,esint,float,multirow,array,color,url,booktabs,mathrsfs,mathdots,amsthm}
\usepackage[dvipsnames]{xcolor}
\usepackage{siunitx}
\usepackage{enumitem}
\usepackage[margin=1in]{geometry}
\usepackage{url}
\usepackage{caption}
\usepackage{subcaption}
\usepackage{cleveref}
\usepackage{mathtools}
\usepackage{authblk}

\newtheorem{thm}{Theorem}
\newtheorem{lem}[thm]{Lemma}

\newcommand{\bl}[1]{\boldsymbol{#1}}
\newcommand{\dnorm}[1]{{\left\vert\kern-0.25ex\left\vert\kern-0.25ex\left\vert #1 
    \right\vert\kern-0.25ex\right\vert\kern-0.25ex\right\vert}}
\newcommand{\norm}[1]{\left\|{#1}\right\|}
\newcommand{\snorm}[1]{\left|{#1}\right|}
\newcommand{\inner}[1]{\bigg({#1}\bigg)}
\newcommand{\Gh}[0]{{\Gamma_{\rm h}}}
\newcommand{\Gg}[0]{{\Gamma_{\rm g}}}
\newcommand{\Ot}[0]{{\tilde \Omega}}
\newcommand{\An}[0]{{\bl A_{\rm n}}}
\newcommand{\res}[0]{{\pmb{\mathscr{L}}}}
\newcommand{\rres}[0]{{\mathscr{L}}}
    
\title{A new stabilized time-spectral finite element solver for fast simulation of blood flow} 
\author[1,*]{Mahdi Esmaily}
\author[1]{Dongjie Jia}
\affil[1]{Sibley School of Mechanical and Aerospace Engineering, Cornell University, Ithaca, NY, USA}
\affil[*]{Correspondence: me399@cornell.edu}
\date{}

\begin{document}
\maketitle

\begin{abstract}
    The increasing application of cardiorespiratory simulations for diagnosis and surgical planning necessitates the development of computational methods significantly faster than the current technology. To achieve this objective, we leverage the time-periodic nature of these flows by discretizing equations in the frequency domain instead of the time domain. This approach markedly reduces the size of the discrete problem and, consequently, the simulation cost. With this motivation, we introduce a finite element method for simulating time-periodic flows that are physically stable. The proposed time-spectral method is formulated by augmenting the baseline Galerkin's method with a least-squares penalty term that is weighed by a positive-definite stabilization matrix. An error estimate is established for the convective-diffusive system, showing that the proposed method emulates the behavior of existing standard time methods including optimal convergence rate in diffusive regimes and stability in strong convection. This method is tested on a patient-specific Fontan model at nominal Reynolds and Womersley numbers of 500 and 10, respectively, demonstrating its ability to replicate conventional time simulation results using as few as 7 modes at 11\% of the computational cost. Owing to its higher local-to-processor computation density, the proposed method also exhibits improved parallel scalability for fast simulation of time-critical applications.
\end{abstract}

\section{Introduction}

The widespread adoption of cardiorespiratory simulations for patient-specific surgical planning relies on the development of methods that are both fast and affordable~\cite{taylor2009patient}. 
The effectiveness of this predictive technology is contingent upon its ability to generate optimal surgical designs within a time frame shorter than the interval between diagnosis and operation~\cite{Yang20102135, Marsden20081890, soerensen2007introduction}. 
Taking stage-one operation performed on single ventricle children as an example, one may have a day at most to perform simulations and utilize its predictions for surgical planning purposes~\cite{blalock1945surgical, Norwood1980hypoplastic, Esmaily2015ABG}. 
With the existing technology, such computations take longer than a week given that an optimization requires many simulations with each taking hours to complete even if they are parallelized~\cite{Esmaily2012optimization, verma2018optimization}. 

Even in cases where simulation turnover time is not a limiting factor, the cost of these calculations can impede widespread industrial adoption.
A notable example is the simulation-based prediction of fractional-flow-reserve (FFR$_{\rm CT}$) for diagnosing coronary artery disease in adults~\cite{driessen2019comparison, modi2019predicting}. 
Despite its deployment at an industrial scale across many hospitals, this non-intrusive diagnostic method relies on the assumption of steady-state flow. 
While this assumption may be physically inaccurate, it is utilized in practice for cost-saving reasons to make such computations more economically viable for large-scale deployment.

Addressing the constraints of simulation turnover time in single ventricles or reducing the cost of unsteady flow analysis for improved FFR$_{\rm CT}$ prediction necessitates the introduction of simulation technology significantly faster than the existing methods. 
To achieve this, we propose the idea of simulating these flows in a time-spectral (or frequency) domain. 
Physically \emph{stable} cardiorespiratory flows are periodic and vary smoothly in time; thus, their temporal behavior can be well-approximated using a few Fourier modes. 
That, in contrast to the thousands of time steps used in conventional methods, presents a unique opportunity to reduce the number of computed unknowns and, consequently, the overall cost of these calculations.

To capitalize on this opportunity, the current study introduces a stabilized time-spectral finite element method for simulating incompressible flows that are inherently periodic in time. 
The proposed Galerkin/least-squares (GLS) method is grounded in a collection of classical approaches from the 1980s and 1990s initially devised for incompressible and compressible flow simulations in the time domain~\cite{hughes1987recent, franca1992stabilized, hughes1995multiscale, hauke1994unified, codina2000stabilization, hughes2010stabilized}.
Specifically, we draw inspiration from the streamline-upwind/Petrov-Galerkin (SUPG) method to ensure numerical stability in the presence of strong convection~\cite{hughes1979multidimentional, brooks1982streamline}. 
The convenient use of equal-order interpolation functions for velocity and pressure~\cite{ladyzhenskaya1969mathematical, babuvska1971error, brezzi1974existence} is facilitated by extending the pressure-stabilized/Petrov-Galerkin (PSPG) method to the present scenario involving multi-modal solutions~\cite{hughes1986circumventing, tezduyar1991stabilized}.
The diagonalization technique, originally developed for compressible flows, is applied in this work to address mode coupling~\cite{hughes1986generalized}. 
The extension of the current method to multidimensional domains discretized by arbitrary elements builds upon fundamental concepts introduced in the past for compressible flows~\cite{shakib1991new}.
The Galerkin/least-squares framework, selected here as the baseline for constructing a stabilized method, boasts a rich history of application across various contexts, including fluid dynamics~\cite{hauke1998comparative, hughes1989new, shakib1989finite}.

The specific choice of methods employed in this study is informed by a series of techniques that we have introduced earlier for time-spectral simulations of the Stokes and convection-diffusion equations~\cite{meng2020time, esmaily2023stabilized, esmaily2023stabilizedA}. 
The first proof-of-concept study showcased the feasibility of simulating blood flow in the time-spectral domain at zero Reynolds number using a Bubnov-Galerkin method~\cite{meng2020time}. 
Subsequently, we demonstrated that the time-spectral Stokes equations, akin to their temporal counterparts, can be solved using the same interpolation functions, provided the formulation is adapted to include the Laplacian of pressure in the continuity equation. 
This modification, plus a change of variable that enabled its implementation using real arithmetic, demonstrated the convenient use and clear cost advantage of the time-spectral method over its temporal counterpart for the simulation of cardiorespiratory flows at zero Reynolds number~\cite{esmaily2023stabilized}. 

Most recently, we focused on the time-spectral form of the convection-diffusion equation to investigate various stabilization strategies for convection-dominant flows as well as reducing dispersion and dissipation errors at high Womersley numbers~\cite{esmaily2023stabilizedA}. 
Five methods were compared in this study: the baseline Petrov-Galerkin (GAL), SUPG, variational multiscale (VMS), GLS, and a novel method tailored for the time-spectral convection-diffusion equation, termed the augmented SUPG method (ASU). 
The key observation was that, although ASU stands out as the most accurate method, it may become unstable at very high Womersley numbers. 
In contrast, the GLS (or equivalently VMS) method was reasonably accurate and always stable, making it an appealing framework for constructing stabilized methods in more general cases. 
Consequently, we rely on this framework to develop a time-spectral solver for the Navier-Stokes equations in the present study.

To the best of our knowledge, the present study is the first to introduce a stabilized time-spectral finite element method for the solution of the Navier-Stokes equations. 
Furthermore, it pioneers the application of time-spectral methods for fast simulation of cardiovascular flows involving complex geometries at physiologic Reynolds numbers. 
However, this is not the first numerical implementation of a frequency-based method, nor is it the initial application of this technology to real-world problems.
Finite volume and finite difference methods have been developed in the past to capitalize on the periodic behavior of the underlying flow \cite{jameson2002application, mcmullen2003application, gopinath2005time}.
These methods have found particular use in modeling turbomachinery flows, which are characterized by time-periodicity and a known rotor rotational frequency \cite{mcmullen2001acceleration, gopinath2007three, he2002analysis, sicot2012time}. 
The preference for frequency-based methods in simulating these flows, beyond their cost-effectiveness, is justified by their superior parallel scalability compared to standard time methods \cite{arbenz2017comparison, hupp2016parallel}. 
In this context, multiple variants of frequency-based techniques have been proposed to boost their overall performance \cite{hall2013harmonic, woiwode2020comparison}.
A noteworthy variant, the harmonic balance method, has recently been applied to investigate modes triggering instabilities in boundary layer transition \cite{hall2002computation, rigas2021nonlinear}, demonstrating wider adoption of the frequency-based methods for fluid dynamics. 

The current study extends the concept of solving problems in the frequency domain to cardiovascular flows. 
To do so, we opt for the finite element method for spatial discretization due to its versatility, widespread use in modeling biological flows~\cite{bermejo2015clinical, mittal2001application, updegrove2017simvascular, jia2022characterization, jia2021efficient}, flexibility for future expansion to address complex boundary conditions~\cite{Lagana2002359, Formaggia2002, Vignon20063776, Vignon2010625, Esmaily2013coupling}, fluid-structure interaction~\cite{Bazilevs20083, Bazilevs2009computational, long2013fluid, kamensky2015immersogeometric, Esmaily2012multipleSPS}, thrombosis~\cite{arzani2014longitudinal, Esmaily2013RT, shadden2013potential, rydquist2022cell}, and growth-and-remodeling~\cite{wu2015coupled, bangalore2017towards, coogan2013computational}.

This study demonstrates the potential benefits of Fourier-based frequency discretization in simulating cardiorespiratory flows. 
That said, several challenges remain to be addressed by the future studies. 
Firstly, the proposed method's cost scales quadratically with the number of simulated modes. 
Developing a method with linear cost scaling is crucial for establishing the time-spectral method as a no-brainer simulation strategy for a broader range of applications. 
Secondly, the current formulation cannot be applied to physically unstable flows, such as chaotic flows or those involving vortex shedding, due to the necessity of specifying the fundamental frequency at which these events occur. 
Thirdly, hemodynamic simulations are often performed in tandem with other physics to model, for instance, fluid-structure interaction, thrombosis, growth and remodeling, etc. 
While that ecosystem already exists for the time formulation, it must be developed for the time-spectral formulation.

The article is organized as follows. 
In the next section, the GLS method is constructed for the Navier-Stokes equation through a series of model problems.
In Section \ref{sec:result}, the proposed method is tested using a realistic patient-specific case involving the Fontan operation. 
The conclusions are drawn in Section \ref{sec:conclusion}. 


\section{Method}
To systematically arrive at a time-spectral formulation for the Navier-Stokes, we first consider the convection-diffusion equation in one dimension in Section \ref{sec:1D}, then generalize it to multiple dimensions in Section \ref{sec:AD}. 
The properties of this method will be studied in Section \ref{sec:analysis}.
A formulation for the Navier-Stokes will be introduced and optimized in Sections \ref{sec:NS} and \ref{sec:opt}, respectively.
Issues related to solution instability in the presence of backflow at the Neumann boundaries will be addressed in Section \ref{sec:bfs}. 
 
As a general rule, variables are selected based on the following convention: italic $f$ for scalar quantities, bold-italic $\bl f$ for vectors, and capital-bold-italic $\bl F$ for matrices. There are exceptions to this rule as, for instance, $\bl \tau$ denotes a matrix. As for subscripts, we employ $n$ and $m$  to denote the mode number and in general $i$ and $j$ to denote direction. Roman subscripts or superscripts $f^{\rm A}_{\rm B}$ are used to construct new variables. 
Variables expressed in time and frequency domains are distinguished using a hat for the time version as, e.g., $f_n$ denotes Fourier discretization of $\hat f(t)$. 
This unconventional notation is adopted here to avoid clutter in equations and emphasize the primary role of the spectral variables that appear in the final implementation of the solver.

\subsection{A model problem: 1D convection-diffusion} \label{sec:1D}
Earlier in \cite{esmaily2023stabilizedA}, we discussed how finite element discretization of the time-spectral form of the convection-diffusion equation can be stabilized for a given steady flow. Our goal in the present and next section is to generalize that stabilization technique to cases in which the flow is unsteady. For this purpose, consider the unsteady convection of an unreactive neutral tracer $\hat \phi(x,t)$ in a one-dimensional domain that is governed by
\begin{equation}
\begin{split}
    \frac{\partial \hat \phi}{\partial t} + \hat u \frac{\partial \hat \phi}{\partial x} &= \kappa \frac{\partial^2 \hat \phi}{\partial x^2}, \\
    \hat \phi(0,t) &= 0, \\
    \hat \phi(L,t) &= \hat g(t),
\end{split}
\label{1D-time}
\end{equation}
where $L\in \mathbb R$ is the domain size, $\kappa \in \mathbb R^+$ is the diffusivity, $\hat g(t) \in \mathbb R$ is the imposed time-dependent boundary condition, and $\hat u(t) \in \mathbb R$ is the given unsteady convective velocity that is uniform in the entire domain.

No initial condition was specified in \eqref{1D-time} since we are only interested in the particular solution (i.e., $\hat \phi(x,t)$ as $t \to \infty$) that is independent of the initial transient behavior of $\hat \phi$ when $\kappa > 0$. 

\eqref{1D-time} is discretized in the time-spectral domain as
\begin{equation}
    \hat \phi(x,t) = \sum_{|n|<N} \phi_n(x) e^{\hat in\omega t},
    \label{phiDef}
\end{equation} 
where $\hat i \coloneqq \sqrt{-1}$ and $\omega\in \mathbb R$ denotes the base frequency and is related to the breathing or cardiac cycle duration $T$ through $\omega \coloneqq 2\pi/T$.
Out of $2N-1$ modes included in \eqref{phiDef} for discretization, only $N$ modes (including the steady mode $n=0$) are independent. That is so since $\hat \phi(x,t)$ is a real function and hence $\phi_{-n} = \phi_n^*$ with $f^*$ denoting complex conjugation of $f$. 

Similarly to \eqref{phiDef}, the imposed boundary condition and convective velocity are discretized as
\begin{equation}
\begin{split}
    \hat g(t) &= \sum_{|n|<N} g_n e^{\hat in\omega t}, \\
    \hat u(t) & = \sum_{|n|<N} u_n e^{\hat in\omega t},
\end{split} 
    \label{guDef} 
\end{equation}
where the corresponding Fourier coefficients are computed as
\begin{equation}
\begin{split}
     g_n & \coloneqq \frac{1}{T}\int_0^T \hat g(t) e^{-\hat in\omega t} dt, \\
     u_n & \coloneqq \frac{1}{T}\int_0^T \hat u(t) e^{-\hat in\omega t} dt.
\end{split}
    \label{guInv} 
\end{equation} 

With these definitions, \eqref{1D-time} is discretized in the time-spectral domain to obtain
\begin{equation}
\begin{split}
    \hat i \underline m \omega \phi_{\underline m} + u_{\underline m-n} \frac{d \phi_n}{d x} & = \kappa \frac{d^2 \phi_{\underline m}}{dx^2}. \\
    \phi_n(0) &= 0, \\
    \phi_n(L) &= g_n,
\end{split}
\label{1D-spec}
\end{equation}
where repeated indices imply summation except for those with an underline $\underline m$ that denote a free index.
Note that \eqref{1D-spec} is a boundary value problem that takes the form of a \emph{steady} convection-diffusion equation with a nonzero imaginary source term. 

In deriving \eqref{1D-spec} from \eqref{1D-time}, we introduced two sources of error.  
The first is related to \eqref{guDef}, where the given $\hat g(t)$ and $\hat u(t)$ are approximated using a finite number of modes. 
Secondly, due to the product of $\hat u$ and $\hat \phi_{,x}$, the highest mode appearing in the convective term is $2N-1$. 
Since that number is larger than the highest mode simulated, i.e., $N-1$, we committed a truncation error in writing \eqref{1D-spec}. 
In practice, however, the error committed by these approximations will be small if the solution is sufficiently smooth so that $\|\phi_n\|$ drops at a fast rate with $n$. 
Later in the result section, we will evaluate these truncation errors for a physiologically relevant test case.   

In writing \eqref{phiDef}, we considered a uniform discretization in the frequency domain with all modes being an integer multiple of $\omega$. Due to modal interaction produced by the convective term in \eqref{1D-spec}, which injects energy to mode $n\pm m$ from modes $n$ and $m$, this uniform discretization is generally an appropriate choice. That is particularly the case when energy is injected at lower modes (i.e., a truncated series in \eqref{guInv} will provide a good approximation to $\hat g$) and cascaded to energize higher modes. However, in special scenarios, the solution may contain two distinct frequencies $\omega_1$ and $\omega_2$ with no (or too small of a) common divisor $\omega$. The discussion of such scenarios, which will likely occur when flow becomes physically unstable, is left for future studies.  

Before proceeding any further to build a stabilized method for \eqref{1D-spec}, it will be instrumental to summarize the findings of an earlier study \cite{esmaily2023stabilizedA} regarding the stabilized solution of a simpler form of \eqref{1D-spec} in strongly convective steady flows when $u_n = 0$ for $n\ne 0$.
In this case, the convective term in \eqref{1D-spec} will no longer couple solutions at various modes, thus permitting us to write 
\begin{equation}
    \hat i \omega \phi + u \frac{d\phi}{dx} = \kappa \frac{d^2 \phi}{dx^2}, 
\label{1D-steady}
\end{equation}
with $u$, $\omega$, and $\phi$ denoting the steady fluid velocity (i.e., $u_0$), the single frequency under consideration, and the solution at that frequency, respectively.

In that study, we showed that supplementing the baseline Galerkin method with the least-squares penalty terms produces a reasonably accurate scheme, called the Galerkin/least-squares (GLS), that exhibits excellent stability. 
In a semi-discrete form, this method, after proper treatment of the boundary conditions, can be stated as finding $\phi^h$ so that for any $w^h$ 
\begin{equation}
       \underbrace{\left(w^h, \hat i\omega \phi^h + u\frac{d\phi^h}{dx} \right)_\Omega + \left(  \frac{dw^h}{dx}, \kappa \frac{d\phi^h}{dx} \right)_\Omega}_\text{Galerkin's terms} + \underbrace{\inner{ \rres(w^h), \tau \rres(\phi^h) }_\Ot}_\text{Least-square penalty terms} = 0,
        \label{1D-Sweak}
\end{equation}
where
\begin{equation*}
    \Ot \coloneqq \bigcup_e \Omega_e
\end{equation*}
is the union of all elements' interior $\Omega_e$.
In general, the inner product $(\cdot,\cdot)_S$ for given complex vector functions $\bl f(\bl x)$ and $\bl g(\bl x)$ over $S$ is defined as  
\begin{equation}
    \inner{\bl f, \bl g}_S \coloneqq \int_S \bl f^H \bl g dS.
    \label{inner_def}
\end{equation} 
where $\bl f^H\coloneqq (\bl f^*)^T$ denotes conjugate transpose.
Furthermore, in \eqref{1D-Sweak}
\begin{equation*}
\rres(\phi^h) \coloneqq \hat i\omega\phi^h + u\frac{d\phi^h}{dx} - \kappa \frac{d^2 \phi^h}{dx^2},
\end{equation*}
denotes the differential operator associated with \eqref{1D-steady}.
The stabilization parameter $\tau$ in \eqref{1D-Sweak}, which plays a key role in the performance of this method, can be defined such that in the steady limit when $\omega=0$ and for linear piecewise interpolation functions the resulting method becomes nodally exact. 
After some approximation to simplify future generalization to multiple dimensions~\cite{hughes1986generalized,shakib1991new,bazilevs2007variational}, that results in 
\begin{equation}
    \tau = \left[\left(\frac{2u}{h}\right)^2 + \left( \frac{12\kappa}{h^2} \right)^2\right]^{-\frac{1}{2}},
\label{1D-Stau}
\end{equation}
where $h$ is the element size. 

As a side note, it is possible to derive a super-convergent method for \eqref{1D-steady} with a nodally exact solution \emph{even} when $\omega\ne0$. 
Such a method will require two stabilization parameters to independently modify, e.g., the physical diffusivity and oscillation frequency.
That is because the exact solution is a function of two dimensionless numbers, i.e., Peclet and Womersley. 
As detailed in \cite{esmaily2023stabilizedA}, that produces a method (ASU) with superior accuracy relative to the GLS, but also inferior stability at high Womersley numbers in multiple dimensions.
Therefore, we leverage the GLS framework to construct a stabilized method in what follows. 

In accordance with the findings in \cite{esmaily2023stabilizedA}, the GLS method yields accurate results when the Womersley number defined based on the element size $h$, i.e., $\beta = \frac{wh^2}{6\kappa}$, remains relatively small. 
For $\beta \gtrapprox 1$, the solution becomes susceptible to significant dispersive errors, which has been studied extensively in the context of Helmholtz equations when modeling short wave propagation \cite{babuska1997pollution, harari1990design, thompson1995galerkin}.

The relevance of this pollution effect for the target applications of the GLS method can be gauged by estimating the value of $\beta$ for a typical cardiorespiratory flow.  
Maintaining $\beta < 1$ for blood flow simulation where $\omega \approx 2\pi$ 1/s (equivalent to a heart rate of 60 beats per minute) and $\kappa \approx 4$ mm$^2$/s (blood kinematic viscosity) implies a maximum mesh size of approximately $h \approx$ 2 mm. 
Imposing a more stringent condition by considering the tenth mode (which will have little effect on the overall solution accuracy for its small amplitude) will necessitate $h \lessapprox 0.6$ mm.

For context, a coronary artery typically has a diameter of 3 to 4 mm. 
Thus, the stipulated upper limit on the mesh size holds true unless an excessively coarse mesh is employed. 
In other words, a mesh that is adequately refined to resolve the flow via a standard time-domain simulation should also be adequate if it were to be used for a simulation performed via the proposed time-spectral GLS method.

Although the GLS method provided in \eqref{1D-Sweak} is successful at preventing non-physical oscillations in the solution in convection-dominant regimes, it can not be readily applied to the coupled system in \eqref{1D-spec}. 
Thus, our strategy in stabilizing this coupled system is to first transform it to a form that is similar to \eqref{1D-steady} and then subject it to penalty terms similar to those appearing in \eqref{1D-Sweak} to stabilize the solution in the presence of strong convection. 
To do so, we begin by defining
\begin{equation}
\begin{split}
    \bl \phi(x) & \coloneqq \left[ \phi_{-N+1}, \cdots, \phi_0, \cdots, \phi_{N-1}\right]^T, \\
    \Omega_{\underline mn} & \coloneqq \hat i \underline m \omega \delta_{\underline m n}, \\
    A_{mn} & \coloneqq u_{m-n}, \hspace{1.3cm} |m-n|<N
\end{split}
\label{phi-def}
\end{equation} 
where $\delta_{mn}$ is the Kronecker delta function. 
Owing to the restriction $|m-n|<N$, the top left and bottom right entries of the convolution matrix $\bl A$ are zero. 
Thus, the calculation of the nonlinear convective term through the product of $\bl A$ with $\frac{d\bl \phi}{dx}$ below will no longer produce aliasing errors by generating frequencies outside of the captured spectrum to spill over and contaminate the lower frequencies. 
Thus, given the definition of $\bl A$, no anti-aliasing strategy is necessary in our implementation. 

Note that utilizing the convolution matrix to compute the convective term produces an $O(N^2)$ operation. 
This operation can be performed in $O(N\log(N))$ if one were to use the fast Fourier transformation instead to compute $\hat u$ and $\frac{d\hat \phi}{dx}$ and bring their product back to the frequency domain to compute the convective term. 
Following such a process will require the adoption of an anti-aliasing strategy -- which again is not necessary in our case owing to the definition of $\bl A$.
Note that the more expensive $O(N^2)$ procedure involving the convolution matrix is adopted here since the computation of $\bl A$ is necessary for the role it plays in the design of the GLS method below.

With the definitions in \eqref{phi-def}, \eqref{1D-spec} can be written as 
\begin{equation}
    \bl \Omega \bl \phi + \bl A \frac{d\bl \phi}{dx} = \kappa \frac{d\bl \phi}{dx^2}. 
    \label{1D-spec1}
\end{equation}

Note that the matrix $\bl A$ is Hermitian (i.e., $\bl A^H = \bl A$) and Toeplitz with only $N$ independent entries. Thus, through eigenvalue decomposition, it can be expressed as
\begin{equation}
    \bl A = \bl V \bl \Lambda \bl V^{H},
    \label{eigA}
\end{equation}
where $\bl \Lambda$ is a diagonal matrix containing eigenvalues of $\bl A$ and $\bl V$ is a matrix containing its eigenvectors (note $\bl V \bl V^H=\bl I$). 
From \eqref{eigA} and \eqref{1D-spec1}, we have
\begin{equation}
        \tilde {\bl \Omega} \tilde{\bl \phi} + \bl \Lambda \frac{d\tilde {\bl \phi}}{dx} = \kappa \frac{d^2\tilde {\bl \phi}}{dx^2},
        \label{1D-spec2}
\end{equation}
where 
\begin{equation}
\begin{split}
    \tilde {\bl \phi} & \coloneqq \bl V^{H} \bl \phi, \\
    \tilde {\bl \Omega} & \coloneqq \bl V^{H} \bl \Omega \bl V. 
    \end{split}
    \label{potDef}
\end{equation}

Owing to $\tilde {\bl \Omega}$, \eqref{1D-spec2} is not fully diagonalized. Nevertheless, we use this equation as the starting point to construct a stabilized scheme. This choice is justified by the fact that the stabilization parameter (c.f., $\tau$ in \eqref{1D-Stau}) is independent of the Eulerian acceleration term (i.e., $\hat i\omega \phi$). 
From a physical perspective, this independence arises since the stabilization terms become active in convection-dominant regimes when the convective acceleration (which is already diagonalized in \eqref{1D-spec2}) becomes more important than the Eulerian acceleration term. 

Given the transformed system in \eqref{1D-spec2} and its corresponding differential operator
\begin{equation*}
       \tilde \res (\tilde {\bl \phi}^h) \coloneqq \tilde {\bl \Omega} \tilde{\bl \phi}^h + \bl \Lambda \frac{d\tilde {\bl \phi}^h}{dx} - \kappa \frac{d^2\tilde {\bl \phi}^h}{dx^2},
\end{equation*}
the semi-discrete form of the GLS can be constructed in a fashion similar to \eqref{1D-Sweak}, yielding  
\begin{equation}
       \inner{\tilde {\bl w}^h, \tilde {\bl \Omega} \tilde{\bl \phi}^h + \bl \Lambda \frac{d\tilde {\bl \phi}^h}{dx} }_\Omega + \inner{\frac{d\tilde {\bl w}^h}{dx}, \kappa \frac{d\tilde {\bl \phi}^h}{dx}}_\Omega + \inner{\tilde \res(\tilde {\bl w}^h), \tilde {\bl \tau} \tilde \res(\tilde{\bl \phi}^h) }_\Ot = 0.
        \label{1D-weakT}
\end{equation}

$\tilde {\bl \tau}$ for the transformed problem in \eqref{1D-weakT} will be a rank $2N-1$ diagonal matrix, the entries of which are calculated based on the eigenvalues of $\bl A$ or entries of $\bl \Lambda$.
Those diagonal entries are computed using the same relationship as that of the steady flow problem in \eqref{1D-Stau}, resulting in
\begin{equation}
    \tilde \tau_{\underline n \underline n} = \left(\left(\frac{2\Lambda_{\underline n \underline n} }{h}\right)^2 + \left( \frac{12\kappa}{h^2} \right)^2\right)^{-\frac{1}{2}}.
\label{1D-tauT}
\end{equation}

Having a method in the transformed system, we can derive the final stabilized method by defining 
\begin{equation*}
   \tilde {\bl w}^h \coloneqq \bl V^H  \bl w^h,
\end{equation*}
which in combination with \eqref{eigA} and \eqref{potDef} permits us to write \eqref{1D-weakT} as
\begin{equation}
       \underbrace{\inner{\bl w^h, \bl \Omega \bl \phi^h + \bl A \frac{d\bl \phi^h}{dx} }_\Omega + \inner{\frac{d\bl w^h}{dx}, \kappa \frac{d\bl \phi^h}{dx}}_\Omega}_\text{Galerkin's terms} + \underbrace{\inner{\res(\bl w^h), \bl \tau \res(\bl \phi^h)}_\Ot}_\text{Least-square penalty terms} = 0,
        \label{1D-weak}
\end{equation}
where 
\begin{equation}
\begin{split}
    \res(\bl \phi^h) & \coloneqq \bl \Omega \bl \phi^h + \bl A \frac{d\bl \phi^h}{dx} - \kappa \frac{d^2\bl \phi^h}{dx^2}, \\ 
    \bl \tau & \coloneqq \bl V \tilde {\bl \tau} \bl V^H. 
\end{split}
\label{1D-tau1}
\end{equation}

Provided that \eqref{1D-tauT} is a well-behaved function with no singular points, it has a Taylor series expansion in terms of eigenvalues of $\bl A$. Hence, it can be combined with \eqref{1D-tau1} to arrive at 
\begin{equation}
\bl  \tau = \left(\left(\frac{2}{h}\right)^2\bl A^2 + \left( \frac{12\kappa}{h^2} \right)^2 \bl I\right)^{-\frac{1}{2}}.
\label{1D-tau}
\end{equation}

\paragraph{Remarks:} 
\begin{enumerate}
    \item The proposed stabilized method, defined by \eqref{1D-weak} and \eqref{1D-tau}, is a generalization of the earlier GLS method that is defined by \eqref{1D-Sweak} and \eqref{1D-Stau}. While the earlier method is only valid for steady flows, the present method applies to both steady and unsteady flows. It is fairly straightforward to arrive at the earlier method if one takes $\bl A = u\bl I$ and decouples the resulting system. 
    \item Computing $\bl \tau$ via \eqref{1D-tau} requires the solution of an eigenvalue problem. In practice, one will perform the eigenvalue decomposition in \eqref{eigA}, then compute eigenvalues of $\bl \tau$ via \eqref{1D-tauT} to compute $\bl \tau$ via \eqref{1D-tau1}. Since $\bl A$ is a Hermitian Toeplitz matrix, one can employ closed-form expressions for the solution of that eigenvalue problem \cite{noschese2013tridiagonal}. However, for larger values of $N$, one has to resort to numerical methods for this purpose. In our experience, the computation of eigenpairs will not increase the cost significantly as $N$ is relatively small. 
    \item As written, \eqref{1D-weak} results in $2N-1$ equations at each node. Since $\phi_{-n}^h = (\phi^h_n)^*$, one has to only solve for $N$ of those equations. This reduction, however, can not be applied to the eigenvalue problem in \eqref{eigA} as $\bl \Lambda$ may contain $2N-1$ independent eigenvalues.
    \item Since $\bl A$ is Hermitian, all of its eigenvalues are real. Thus, all the eigenvalues of $\bl \tau$ are positive in view of \eqref{1D-tauT}. That ensures $\bl \tau$ is a positive-definite matrix. As we will show in Section \ref{sec:analysis}, this property ensures the stability of the GLS method in strongly convective flows. 
    \item $\bl A$ is a Toeplitz Hermitian matrix and, thus, it can be stored using $N$ complex-valued entries. $\bl \tau$, on the other hand, is a centrosymmetric Hermitian matrix so that $\tau_{mn} = \tau_{nm}^*$ and $\tau_{mn} = \tau_{(2N-n)(2N-m)}$. Thus, $\bl \tau$ can be stored using $N^2$ independent complex-valued entries.  
\end{enumerate}

\subsection{Convection-diffusion in multiple dimensions} \label{sec:AD}
The stabilized method derived for the 1D case in the previous section may seem readily applicable to higher dimensions. 
There are, however, two major challenges associated with this generalization that one must overcome to obtain a stabilization method in higher dimensions. 

To demonstrate those challenges, let us first consider the multidimensional form of \eqref{1D-spec1}, which is 
\begin{equation}
\begin{alignedat}{3}
    \bl \Omega \bl \phi + \bl A_i \frac{\partial \bl \phi}{\partial x_i} & = \frac{\partial}{\partial x_i} \left(\kappa \frac{\partial \bl \phi}{\partial x_i}\right),\hspace{0.5in} && \text{in}\;\;\; && \Omega \\ 
    \bl \phi &= \bl g,  && \text{on} && \Gg  \\
    \kappa \frac{\partial \bl \phi}{\partial x_i}n_i  &= \bl h, && \text{on} && \Gh 
\end{alignedat}
\label{3D-spec}
\end{equation}
where $\bl n$ is the outward normal vector to $\partial \Omega \coloneqq \Gamma = \Gh \bigcup \Gg$, $\Gh$ and $\Gg$ are the portion of the boundary where a Neumann and Dirichlet boundary condition is imposed, respectively, $\bl A_i\in \mathbb C^{(2N-1)\times(2N-1)}$ is the convolution matrix, representing the fluid velocity in direction $i$ with $i=$1, 2, and 3 in 3D. For the problem to be well-posed, we assume a Dirichlet boundary condition is imposed on some portion of the boundary, i.e., $\Gamma\ne \Gh$. We additionally assume constant $\kappa$ and a divergent-free flow such that 
\begin{equation*}
    \frac{\partial \bl A_i}{\partial x_i}=\bl 0.
\end{equation*}

To utilize the method described in Section \ref{sec:1D}, we should solve \eqref{eigA} to compute $2N-1$ eigenvalues, which later can be used for $\bl \tau$ calculations. 
This brings up the first challenge: there are three convective velocity matrices and it is not clear how they should be utilized to generate those eigenvalues. 
From a physical perspective, this challenge arises since we need to assign a streamwise direction for $\bl \tau$ calculation in multiple dimensions when velocity vectors associated with various modes are pointing in different directions. 

The second challenge, which is tangled with the first, is to define a rigorous and universal approach for computing element size $h$ in multiple dimensions. 
That is especially the case for arbitrary unstructured grids, including tetrahedral elements that are widely adopted for complex geometries encountered in cardiorespiratory simulations. 

The insight to overcome these challenges has come from the earlier design of stabilized methods for a different application area, namely compressible flows \cite{hughes1986generalized,shakib1991new}.
To explain the idea, consider the form of $\bl \tau$ in the 1D case in \eqref{1D-tau}, which explicitly depends on the element size $h$. 
The $h$-dependent value of $\bl \tau$ is obtained when the derivatives in the convective and diffusive terms in \eqref{1D-spec1} are expressed with regard to the physical coordinate system. 
To bake in the dependence of $\bl \tau$ on the element size (and hence its configuration), one can instead express these terms in the element parent coordinate system $\xi$ rather than the physical coordinate system $x$. 
That results in 
\begin{equation}
    \bl \Omega \bl \phi + \bl A \frac{\partial \xi}{\partial x}\frac{\partial \bl \phi}{\partial \xi} = \kappa \left(\frac{\partial \xi}{\partial x}\right)^2 \frac{\partial^2 \bl \phi}{\partial \xi^2}. 
    \label{1D-spec3}
\end{equation}

Designing a stabilization parameter $\bl \tau$ for \eqref{1D-spec3} is straightforward once we recognize its convective velocity and diffusivity are $\bl A \frac{\partial \xi}{\partial x}$ and $\kappa \left(\frac{\partial \xi}{\partial x}\right)^2$, respectively. 
Since the piecewise linear shape functions in the parent coordinate system have an element size of 2 (assuming $-1 \le \xi \le 1$), \eqref{1D-tau} for \eqref{1D-spec3} can be written as 
\begin{equation}
    \bl \tau = \left(\left(\frac{\partial \xi}{\partial x}\right)^2\bl A^2 + 9\kappa^2 \left(\frac{\partial \xi}{\partial x}\right)^4\bl I\right)^{-\frac{1}{2}}.
\label{1D-tau1p}
\end{equation}
Since \eqref{1D-spec3} is an alternative form of \eqref{1D-spec1} and $\bl \tau$ remains unchanged when deriving a semi-discrete form for these two equations, one can compute $\bl \tau$ using \eqref{1D-tau1p} rather than \eqref{1D-tau}.

Using \eqref{1D-tau1p} as the starting point significantly simplifies the construction of $\bl \tau$ in multiple dimensions for \eqref{3D-spec}. 
Rewriting the derivatives in \eqref{3D-spec} in the parent coordinate system results in
\begin{equation*}
    \bl \Omega \bl \phi + \bl A_i \frac{\partial \xi_j}{\partial x_i} \frac{\partial \bl \phi}{\partial \xi_j} = \kappa \frac{\partial \xi_j}{\partial x_i} \frac{\partial \xi_k}{\partial x_i} \frac{\partial^2 \bl \phi}{\partial \xi_j\partial \xi_k}.
\end{equation*}
Therefore, the generalization of \eqref{1D-tau1p} to multiple dimensions is 
\begin{equation}
    \bl \tau = \left( \bl A_i \frac{\partial \xi_j}{\partial x_i}  \frac{\partial \xi_j}{\partial x_k} \bl A_k + C_{\rm I} \kappa^2 \frac{\partial \xi_j}{\partial x_i} \frac{\partial \xi_k}{\partial x_i} \frac{\partial \xi_j}{\partial x_l} \frac{\partial \xi_k}{\partial x_l} \bl I \right)^{-1/2}, 
    \label{3D-tau1}
\end{equation}
where the coefficient of 9 in \eqref{1D-tau1p} is replaced by the element-type-dependent $C_{\rm I}$ since the element size in the parent coordinate for a given direction may not be 2. 
This coefficient is set to $C_{\rm I} =3$ for simulations performed using tetrahedral linear elements in the following sections. 
As discussed in detail in Section \ref{sec:analysis}, a bound on error can be established if there is no backflow on Neumann boundaries and   
\begin{equation}
    \frac{4C_{\rm inv}^4}{C_\square} \le C_{\rm I}, 
    \label{ci_cinv}
\end{equation}
where $C_{\rm inv}$ is a constant from the inverse inequality
\begin{equation}
   \left\|\frac{\partial^2\bl w^h}{\partial x_i\partial x_i}\right\|_{\Omega_e} \le \frac{C_{\rm inv}}{h} \left\|\frac{\partial \bl w^h}{\partial x_i}\right\|_{\Omega_e}. 
    \label{inv_ineq}
\end{equation}
In this equation, $h$ is the size of the largest element, i.e.,
\begin{equation*}
    h = \max_e(h_e),
\end{equation*}
where $h_e$ is the diameter of the sphere containing $\Omega_e$ in 3D. 
$C_\square$ in \eqref{ci_cinv} is a constant that depends on the definition of shape functions in the parent coordinate only (e.g., $C_\square = 128$ for quadrilateral element if $-1\le \xi_i \le 1$). It is selected to satisfy
\begin{equation}
C_\square \le h^4 G_{ij}G_{ij} \le C_\gamma,
\label{GH_rel}
\end{equation}
where
\begin{equation}
    G_{ij} \coloneqq \frac{\partial \xi_k}{\partial x_i} \frac{\partial \xi_k}{\partial x_j},
    \label{metric}
\end{equation}
is the metric tensor. 
Note that in contrast to $C_\square$, $C_\gamma$ depends on the element shape and will be large if elements have a large aspect ratio. 
For $C_\gamma$ to be independent of $h$, we assume if the mesh is refined, it is done in a way that its quality is preserved (or the element aspect ratio remains bounded).

The lower bound imposed on $C_{\rm I}$ via \eqref{ci_cinv} limits the amount of artificial (or numerical) diffusion introduced via the stabilization term.
This limit, however, is only applied to nonlinear interpolation functions given that for linear shape functions \eqref{inv_ineq} holds true with $C_{\rm inv} = 0$, thus producing the trivial condition $C_{\rm I} \ge 0$. 
This observation is compatible with the general trend that $\bl \tau$ should become smaller with increasing polynomial degree \cite{galeao2004finite}.

Before proceeding further, it is crucial to address the universality of \eqref{3D-tau1} concerning the element type. Since this equation originates from \eqref{1D-Stau} for linear elements in 1D, it can be demonstrated that this design is directly applicable to quadrilateral elements in 2D and hexahedral elements in 3D. The extension to other element types is also facilitated by the aforementioned transformation.

To illustrate this, one can examine the conventional time simulations and convection-dominant flows. The conventional $\tau$ has proven successful for flow simulations across various element types. The same holds true for the proposed $\bl \tau$, given that it simplifies to the conventional $\tau$ in steady flows. Additionally, in convection-dominant flows, the term containing $C_{\rm I}$ in \eqref{3D-tau1} becomes negligible at the asymptotic limit. This results in a $\bl \tau$ design that is entirely independent of the element type. This behavior is particularly significant since the solution strongly relies on the design of $\bl \tau$ in convection-dominant flows. This shows that the proposed form becomes universal when its design matters the most. For a more formal analysis of the behavior of this formulation using arbitrary elements, the readers may refer to Section \ref{sec:analysis}.

To arrive at the final form of the GLS method, we can simplify \eqref{3D-tau1} using \eqref{metric} to obtain 
\begin{equation}
    \bl \tau = \left[ \bl A_i G_{ij} \bl A_j + C_{\rm I} \kappa^2 G_{ij}G_{ij} \bl I \right]^{-1/2}. 
    \label{3D-tau}
\end{equation}
The numerical calculation of $\bl \tau$ from \eqref{3D-tau}, which requires computing the root square of a matrix, is identical to the process mentioned earlier for \eqref{1D-tau}. 
It is done by computing the matrix under the root square in \eqref{3D-tau}, then calculating its eigenpairs to apply the operation of inverse root square to its eigenvalues before multiplying the result by the eigenvectors to get $\bl \tau$.
Having $\bl \tau$, the GLS stabilized formulation of the 3D convection-diffusion in the time-spectral domain becomes very similar to its 1D counterpart. 
It can be formally stated as finding $\bl \phi^h \in \bl V^h$ such that for any $\bl w^h \in \bl W^h$
\begin{equation}
       b(\bl w^h,\bl \phi^h) = \inner{\bl w^h,\bl h}_\Gh,
        \label{3D-weak}
\end{equation}
where
\begin{equation}
\begin{split}
       b(\bl w^h,\bl \phi^h) & \coloneqq \inner{\bl w^h, \bl \Omega \bl \phi^h + \bl A_i \frac{\partial \bl \phi^h}{\partial x_i}}_\Omega + \inner{ \frac{\partial \bl w^h}{\partial x_i}, \kappa \frac{\partial \bl \phi^h}{\partial x_i}}_\Omega + \inner{\res(\bl w^h), \bl \tau \res(\bl \phi^h)}_\Ot, \\
    \res(\bl \phi^h) & \coloneqq \bl \Omega \bl \phi^h + \bl A_i \frac{\partial \bl \phi^h}{\partial x_i} - \kappa \frac{\partial^2 \bl \phi^h}{\partial x_i \partial x_i}.
    \end{split}
        \label{3D-weak-def}
\end{equation}
Here $\bl V^h$ and $\bl W^h$ are the finite element spaces containing all possible $\bl \phi^h$ and $\bl w^h$ subjected to the constraint $\bl \phi^h=\bl g$ and $\bl w^h=\bl 0$ on $\Gg$, respectively. 
That concludes the construction of the GLS stabilized method for the convection-diffusion equation for the general case in which flow may be unsteady, the number of spatial dimensions may be larger than one, and the elements may be arbitrary. 

\subsection{Properties of the GLS method} \label{sec:analysis}
Our ultimate goal in this section is to establish the accuracy of the GLS method described in the previous section by providing an upper bound on the error in terms of the mesh size $h$. 
This condition, which will be established later via Theorem \ref{thm:error}, relies on a set of preliminary results that are presented here via Lemma \ref{lem:stability}-\ref{lem:interp}. 
The practical interpretation of these results will be provided in the form of a set of remarks at the end of this section. 

The properties of stabilized methods, including the GLS have been studied in detail in the past. 
For an overview of the concepts, the readers may refer to~\cite{brezzi1990discourse,franca1991error,franca1993convergence}. 
An analysis of stabilized methods for advection-diffusion-reaction equations has been performed in \cite{hughes1989new,knobloch2011stability} for a single equation and in~\cite{shakib1991new} for a system of equations.
Provided that the GLS method discussed here involves an imaginary source term, the arguments found in those studies are slightly modified to prove an error estimate for the method under consideration.

Throughout the following derivation, we use the following definitions:
\begin{equation*}
    \norm{\bl f}_{\bl M,S}^2 \coloneqq \inner{\bl f, \bl M \bl f}_S = \int_S \bl f^H \bl M \bl f dS
\end{equation*}
is the $L_2$ norm with respect to metric $\bl M$ in domain $S$. 
We drop $\bl M$ from this notation if $\bl M=\bl I$.
Furthermore, the smallest and largest eigenvalue of symmetric matrix $\bl M$ are denoted by $\lambda_{\min}(\bl M)$ and $\lambda_{\max}(\bl M)$, respectively. Namely
\begin{alignat*}{2}
    \lambda_{\max}(\bl M) & \coloneqq \sup_{\bl v} \frac{\bl v^H \bl M \bl v}{\bl v^H\bl v}, \\
    \lambda_{\min}(\bl M) & \coloneqq \inf_{\bl v} \frac{\bl v^H \bl M \bl v}{\bl v^H\bl v}. 
\end{alignat*}
The local and global element Peclet numbers are defined as  
\begin{equation}
\begin{split}
    \alpha_e & \coloneqq \left[\frac{\lambda_{\max}(\bl A_i G_{ij} \bl A_j)}{C_{\rm I} \kappa^2 G_{ij} G_{ij}}\right]^{1/2}, \\
    \alpha & \coloneqq  \max_{\bl x\in \Omega} (\alpha_e), 
\end{split}
    \label{peclet_def}
\end{equation}
respectively. 
Similarly, the element Womersley number is defined as 
\begin{equation}
    \beta \coloneqq h\sqrt\frac{(N-1)\omega}{\kappa}.
    \label{womersley_def}
\end{equation}
Lastly, $c$, $c_1$, $c_2$, and $c_3$ constants appearing in the following derivation are all taken to be positive and all vector variables (e.g., $\bl w$) are symmetric with regard to the imaginary half-plane as shown in \eqref{phi-def} and sufficiently regular to be in $H^1(\Omega)^{2N-1} \cap H^2(\Ot)^{2N-1}$.

\begin{lem} \label{lem:stability}
(Stability). If  
\begin{equation*}
    \An \coloneqq \bl A_i n_i
\end{equation*}
is positive semi-definite on $\Gh$ (no backflow), then for a given $\bl w$
\begin{equation*}
       b(\bl w,\bl w) \coloneqq \dnorm{\bl w}^2 = \frac{1}{2}\snorm{\bl w}_{\An,\Gh}^2 +  \kappa \norm{\frac{\partial \bl w}{\partial x_i}}^2_\Omega + \norm{\res(\bl w)}^2_{\bl \tau,\Ot}.
\end{equation*}
\end{lem}

\begin{proof}
By definition \eqref{3D-weak-def}
\begin{equation}
       b(\bl w,\bl w) = \inner{\bl w, \bl \Omega \bl w}_\Omega + \inner{\bl w, \bl A_i \frac{\partial \bl w}{\partial x_i}}_\Omega + \kappa \norm{\frac{\partial \bl w}{\partial x_i}}_\Omega^2 + \norm{\res(\bl w)}_{\bl \tau,\Ot}^2.
       \label{lem1_1}
\end{equation}
Note that by symmetry, all terms produced by the sesquilinear $b(\bl w,\bl w)$,  which is conjugate transposed with regard to the first argument, are real. Hence
\begin{alignat}{2}
        \inner{\bl w,\bl \Omega \bl w}_\Omega &= \int_\Omega \bl w^H \bl \Omega \bl w d\Omega && (\text{by definition \eqref{inner_def}}) \nonumber \\ 
        & = \int_\Omega \left( \bl w^H \bl \Omega \bl w \right)^H d\Omega  \hspace{1.1in} && (\text{by symmetry}) \nonumber \\
        & = -\int_\Omega \bl w^H \bl \Omega \bl w d\Omega && \nonumber (\text{since } \bl \Omega^H=-\bl \Omega^H)\\ 
        & = 0.
    \label{lem1_2}
\end{alignat}
Also
\begin{alignat}{2}
        \inner{\bl w,\bl A_i \frac{\partial \bl w}{\partial x_i}}_\Omega & = \inner{\bl w,\bl A_in_i \bl w }_\Gh - \inner{\frac{\partial \bl w}{\partial x_i},\bl A_i \bl w }_\Omega \hspace{0.5in} && (\text{by integration-by-parts}) \nonumber \\
        & = \frac{1}{2}\inner{\bl w,\An \bl w }_\Gh &&  \nonumber \\ 
        & = \frac{1}{2}\snorm{\bl w}^2_{\An,\Gh}, && (\An \ge 0) 
    \label{lem1_3}
\end{alignat}
which is written in terms of a seminorm since $\An$ may be zero on $\Gh$ or $\Gamma = \Gg$.  
The intended result directly follows from \eqref{lem1_1}, \eqref{lem1_2}, and \eqref{lem1_3}.  
From the diffusive term contribution, one can readily conclude that $\dnorm{\bl w^h}$ is positive-definite for all $\bl w^h \in \bl W^h$ (note $\bl w^h$ can not be a nonzero constant since $\Gamma \ne \Gh$).
Thus, in view of this lemma, $b$ is obviously coercive with respect to $\dnorm{\cdot}$.
\end{proof}

\begin{lem}\label{lem:consistency}
(Consistency). Let $\bl \phi$ be the solution to \eqref{3D-spec}. Then for all $\bl w^h \in \bl W^h$
\begin{equation*}
       b(\bl w^h,\bl \phi) = \inner{\bl w^h,\bl h}_\Gh. 
\end{equation*}
\end{lem}

\begin{proof}
    This is immediate from \eqref{3D-spec} and \eqref{3D-weak-def}. 
\end{proof}

\begin{lem}\label{lem:tau_lb}
($\bl \tau$ bounded from below). If \eqref{GH_rel} holds, then for a given $\bl w$
\begin{equation*}
       \norm{\bl w}_{\bl \tau^{-1},\Omega}^2 \le c \kappa h^{-2}(1+\alpha) \norm{\bl w}^2_{\Omega},
\end{equation*}
where $c^2 = C_{\rm I}C_\gamma$.
\end{lem}

\begin{proof}
\begin{alignat*}{2}
       \norm{\bl w}^2_{\bl \tau^{-1},\Omega} & = \inner{\bl w, (\bl A_iG_{ij}\bl A_j + C_{\rm I} \kappa^2 G_{ij}G_{ij}\bl I)^{1/2}\bl w}_\Omega && (\text{by definition \eqref{3D-tau}})\\ 
       & \le \inner{\bl w, ((\alpha_e^2+1) C_{\rm I} \kappa^2 G_{ij}G_{ij})^{1/2}\bl w}_\Omega && (\text{by definition \eqref{peclet_def}})  \\ 
       &\le \inner{\bl w,(C_{\rm I} C_\gamma)^{1/2} \kappa h^{-2}(1+\alpha_e^2)^{1/2} \bl w}_\Omega \hspace{1in} && ({\rm by\; \eqref{GH_rel}}) \\ 
       &\le c \kappa h^{-2}(1+\alpha^2)^{1/2} \norm{\bl w}_\Omega^2 && (\text{by definition \eqref{peclet_def}}) \\ 
       &\le c \kappa h^{-2}(1+\alpha) \norm{\bl w}_\Omega^2  && (\text{since } \alpha \ge 0).
\end{alignat*}
\end{proof}

\begin{lem} \label{lem:tau-uba}
($\bl \tau$ bounded from above). If \eqref{GH_rel} holds, then for a given $\bl w$
\begin{equation*}
       \norm{\bl A_i \frac{\partial \bl w}{\partial x_i}}_{\bl \tau,\Omega}^2 \le c \kappa \alpha \norm{\frac{\partial \bl w}{\partial x_i}}^2_{\Omega},
\end{equation*}
where $c^2 = C_{\rm I}C_\gamma d/C_{\square}$ and $d$ is the number of spatial dimensions.
\end{lem}

\begin{proof}
First, we need to establish a preliminary result. Note  
\begin{equation*}
    G_{ij}G_{ij} = \norm{\bl G}_{\rm F}^2 = \sum_{i=1}^d \lambda_i^2(\bl G),
\end{equation*}
where $\norm{\cdot}_{\rm F}$ is the Frobenius norm of a matrix. Thus
\begin{equation*}
    d\lambda^2_{\min}(\bl G) \le G_{ij}G_{ij}. 
\end{equation*}
Thus, \eqref{GH_rel} will be satisfied if we select $C_{\square}$ such that 
\begin{equation}
    C_{\square} \le h^4d \lambda^2_{\min}(\bl G). 
    \label{frob_norm}
\end{equation}
With this result, we have
\begin{alignat*}{2}
       \norm{\bl A_i \frac{\partial \bl w}{\partial x_i}}^2_{\bl \tau,\Omega} & = \norm{\bl A_i\frac{\partial \xi_j}{\partial x_i} \frac{\partial \bl w}{\partial \xi_j}}^2_{\bl \tau,\Omega} \\ 
       & = \inner{\frac{\partial \bl w}{\partial \xi_k},\frac{\partial \xi_k}{\partial x_i} \bl A_i \bl \tau \bl A_j\frac{\partial \xi_l}{\partial x_j} \frac{\partial \bl w}{\partial \xi_l}}_\Omega  && (\text{since } \bl A_i^H = \bl A_i)  \\ 
       &\le \inner{\frac{\partial \bl w}{\partial \xi_k},\frac{\partial \xi_l}{\partial x_i} \bl A_i \bl \tau \bl A_j\frac{\partial \xi_l}{\partial x_j} \frac{\partial \bl w}{\partial \xi_k}}_\Omega \hspace{1.5in} && (\text{by Cauchy-Schwarz}) \\ 
       & = \inner{\frac{\partial \bl w}{\partial \xi_k},\bl A_i G_{ij}\bl A_j \bl \tau \frac{\partial \bl w}{\partial \xi_k}}_\Omega && (\text{by definition \eqref{metric}}) \\ 
       &\le \inner{\frac{\partial \bl w}{\partial \xi_k},(\bl A_i G_{ij}\bl A_j)^{1/2} \frac{\partial \bl w}{\partial \xi_k}}_\Omega  &&  (\text{by definition \eqref{3D-tau}})  \\
       &\le \inner{\frac{\partial \bl w}{\partial \xi_k},C_{\rm I}^{1/2}\alpha_e \kappa (G_{ij}G_{ij})^{1/2} \frac{\partial \bl w}{\partial \xi_k}}_\Omega  &&  (\text{by definition \eqref{peclet_def}}) \\
       &\le (C_{\rm I}C_\gamma)^{1/2}\kappa \alpha h^{-2} \inner{\frac{\partial \bl w}{\partial \xi_k}, \frac{\partial \bl w}{\partial \xi_k}}_\Omega  &&  ({\rm by\; \eqref{GH_rel}})  \\
       &= (C_{\rm I}C_\gamma)^{1/2}\kappa \alpha h^{-2} \inner{\frac{\partial \bl w}{\partial x_i}\frac{\partial x_i}{\partial \xi_k}, \frac{\partial x_j}{\partial \xi_k} \frac{\partial \bl w}{\partial x_j}}_\Omega  \\
       &= (C_{\rm I}C_\gamma)^{1/2}\kappa \alpha h^{-2} \inner{\frac{\partial \bl w}{\partial x_i}, G_{ij}^{-1} \frac{\partial \bl w}{\partial x_j}}_\Omega  &&  (\text{by definition \eqref{metric}})  \\
       &\le (C_{\rm I}C_\gamma)^{1/2}\kappa \alpha h^{-2} \inner{\frac{\partial \bl w}{\partial x_i}, \lambda_{\max}(\bl G^{-1}) \frac{\partial \bl w}{\partial x_i}}_\Omega   \\
       &= (C_{\rm I}C_\gamma)^{1/2}\kappa \alpha h^{-2} \inner{\frac{\partial \bl w}{\partial x_i}, \lambda^{-1}_{\min}(\bl G) \frac{\partial \bl w}{\partial x_i}}_\Omega  \\
       &\le c \kappa \alpha \norm{\frac{\partial \bl w}{\partial x_i}}^2_\Omega. &&  ({\rm by\; \eqref{frob_norm}})  
       \label{adv_ub_proof}
\end{alignat*}
\end{proof}

\begin{lem} \label{lem:tau-ubd}
($\bl \tau$ bounded from above). If \eqref{ci_cinv}, \eqref{inv_ineq}, and \eqref{GH_rel} hold, then for a given $\bl w^h \in \bl W^h$ 
\begin{equation*}
       \norm{\kappa \frac{\partial^2 \bl w^h}{\partial x_i \partial x_i}}^2_{\bl \tau,\Ot} \le \frac{\kappa}{2} \norm{\frac{\partial \bl w^h}{\partial x_i}}^2_\Omega.
\end{equation*}
\end{lem}

\begin{proof}
\begin{alignat*}{2}
        \lambda_{\min}( \bl \tau^{-2} ) & = \lambda_{\min}\left( \bl A_i G_{ij} \bl A_j +  C_{\rm I}\kappa^2G_{ij}G_{ij} \bl I \right) \hspace{1.4in} && (\text{by definition \eqref{3D-tau}})  \\ 
       &\ge  C_{\rm I}\kappa^2 G_{ij}G_{ij} \\ 
       &\ge C_{\rm I}C_\square\frac{\kappa^2}{h^4},  && ({\rm by\; \eqref{GH_rel}}) 
\end{alignat*}
or 
\begin{equation}
    \lambda_{\max}(\bl \tau) \le \frac{h^2}{\kappa (C_{\rm I}C_\square)^{1/2}}.
     \label{lambda_max}
\end{equation}
Thus
\begin{alignat*}{2}
       \norm{\kappa \frac{\partial^2 \bl w^h}{\partial x_i \partial x_i}}^2_{\bl \tau,\Ot} & \le \kappa^2 \lambda_{\max}(\bl \tau) \norm{\frac{\partial^2 \bl w^h}{\partial x_i \partial x_i}}^2_\Ot  \\ 
       &\le \frac{ \kappa h^2}{(C_{\rm I}C_\square)^{1/2}} \norm{\frac{\partial^2 \bl w^h}{\partial x_i \partial x_i}}^2_\Ot \hspace{1.5in} && ({\rm by\; \eqref{lambda_max}}) \\ 
       &\le \frac{\kappa C_{\rm inv}^2}{(C_{\rm I}C_\square)^{1/2}} \norm{\frac{\partial \bl w^h}{\partial x_i}}^2_\Omega && ({\rm by\; \eqref{inv_ineq}}) \\ 
       & \le \frac{\kappa}{2} \norm{\frac{\partial \bl w^h}{\partial x_i}}^2_\Omega. && ({\rm by\; \eqref{ci_cinv}})
\end{alignat*}
\end{proof}

\begin{lem} \label{lem:interp}
 (Interpolation error). Let $\bl \phi \in H^{p+1}(\Omega)^{2N-1}$ be the solution to \eqref{3D-spec} and $\Pi^h\bl \phi$ its best interpolation using order $p$ polynomials. Assuming Lemma \ref{lem:tau_lb} and \ref{lem:tau-uba} hold, then the interpolation error $\bl \eta = \Pi^h\bl \phi - \bl \phi$ satisfies
\begin{equation*}
       2\norm{\bl \eta}_{\bl \tau^{-1},\Omega}^2 + \kappa \norm{\frac{\partial \bl \eta}{\partial x_i}}_\Omega^2 + \snorm{\bl \eta}_{\An,\Gh}^2 + \norm{\res(\bl \eta)}_{\bl \tau,\Ot}^2 \le  (c_1 +c_2 \alpha+c_3 \beta^4)\kappa h^{2p} \snorm{\bl \phi}_{H^{p+1}}^2.
\end{equation*}
\end{lem}

\begin{proof}
    From Lemma \ref{lem:tau_lb} and standard interpolation theory
\begin{equation}
     2\norm{\bl \eta}_{\bl \tau^{-1},\Omega}^2 \le c\kappa h^{-2}(1+\alpha)\norm{\bl \eta}_\Omega^2 \le c \kappa (1+\alpha) h^{2p}  \snorm{\bl \phi}_{H^{p+1}}^2. 
    \label{int_e_1}
\end{equation}
Similarly
\begin{equation}
     \kappa \norm{\frac{\partial \bl \eta}{\partial x_i}}_\Omega^2 = \kappa \snorm{\bl \eta}_{H^1,\Omega}^2 \le c \kappa h^{2p} \snorm{\bl \phi}_{H^{p+1}}^2. 
    \label{int_e_2}
\end{equation}
Also
\begin{alignat}{2}
       \snorm{\bl \eta}_{\An,\Gh}^2 & = \inner{\bl \eta, \An\bl \eta}_\Gh  \nonumber \\ 
       &= 2\inner{\bl \eta, \bl A_i \frac{\partial \bl \eta}{\partial x_i}}_\Omega \hspace{1.5in} && (\text{by integration-by-parts}) \nonumber \\ 
       &\le 2\norm{\bl \eta}_{\bl \tau^{-1},\Omega} \norm{\bl A_i\frac{\partial \bl \eta}{\partial x_i}}_{\bl \tau,\Omega} && (\text{by Cauchy-Schwarz}) \nonumber \\ 
       &\le c\kappa h^{-1}\sqrt{(1+\alpha)\alpha}\norm{\bl \eta}_\Omega \norm{\frac{\partial \bl \eta}{\partial x_i}}_\Omega  &&  (\text{by Lemma \ref{lem:tau_lb} and \ref{lem:tau-uba}})  \nonumber \\
       &\le c\kappa (1+\alpha) h^{2p} \snorm{\bl \phi}_{H^{p+1}}^2  &&  (\text{since }\alpha \ge 0). 
       \label{int_e_3}
\end{alignat}
Using triangular inequality on the least-squares term, we have
\begin{equation}
    c\norm{\res(\bl \eta)}_{\bl \tau,\Ot}^2 \le \norm{\bl \Omega \bl \eta}_{\bl\tau,\Ot}^2 + \norm{\bl A_i \frac{\partial \bl \eta}{\partial x_i}}_{\bl\tau,\Ot}^2 + \norm{\kappa \frac{\partial^2 \bl \eta}{\partial x_i\partial x_i}}_{\bl\tau,\Ot}^2. 
     \label{int_e_4}
\end{equation}
The first term can be bounded as 
\begin{alignat}{2}
       \norm{\bl \Omega \bl \eta}_{\bl\tau,\Ot}^2 & \le \norm{\lambda_{\max}(\bl \tau)^{1/2}\bl \Omega \bl \eta}_{\Ot}^2 \nonumber \\ 
      & \le c\kappa^{-1} h^2 \norm{\bl \Omega \bl \eta}_{\Ot}^2 && (\text{by \eqref{lambda_max}}) \nonumber \\ 
      & \le c\kappa^{-1} h^2 (N-1)^2\omega^2 \norm{\bl \eta}_\Ot^2 \hspace{0.8in} && (\text{by definition \eqref{phi-def}}) \nonumber \\ 
      & = c\kappa \beta^4 h^{-2} \norm{\bl \eta}_\Ot^2 && (\text{by definition \eqref{womersley_def}}) \nonumber \\ 
      & \le c\kappa \beta^4 h^{2p} \snorm{\bl \phi}_{H^{p+1}}^2,
       \label{int_e_5}
\end{alignat}
and by Lemma \ref{lem:tau-uba}
\begin{equation}
    \norm{\bl A_i \frac{\partial \bl \eta}{\partial x_i}}_{\bl\tau,\Ot}^2 \le c \kappa \alpha \norm{\frac{\partial \bl \eta}{\partial x_i}}_\Omega^2 \le c \kappa \alpha  h^{2p} \snorm{\bl \phi}_{H^{p+1}}^2,
    \label{int_e_6}
\end{equation}
and by \eqref{lambda_max}
\begin{equation}
    \norm{\kappa \frac{\partial^2 \bl \eta}{\partial x_i\partial x_i}}_{\bl\tau,\Ot}^2 \le  c\kappa h^2 \snorm{\bl \eta}_{H^2,\Ot}^2 \le c \kappa h^{2p} \snorm{\bl \phi}_{H^{p+1}}^2. 
    \label{int_e_7}
\end{equation}
The proof of this Lemma directly follows from \eqref{int_e_1}-\eqref{int_e_7}. 
\end{proof}

\begin{thm} \label{thm:error}
(Error estimate). Assuming Lemma \ref{lem:stability}, \ref{lem:consistency}, \ref{lem:tau-ubd}, and \ref{lem:interp} hold, then the error $\bl e = \bl \phi - \bl \phi^h$ in the solution of \eqref{3D-weak} will be bounded as follows
\begin{equation*}
       \dnorm{\bl e}^2 \le (c_1 +c_2 \alpha+c_3 \beta^4) \kappa h^{2p}\snorm{\bl \phi}_{H^{p+1}}^2.
\end{equation*}
\end{thm}

\begin{proof}
With the essential ingredients established through the earlier lemmas, the proof of this theorem can follow the same steps as those described in \cite{hughes1989new}, which are reproduced here for completeness.
Defining the method's error $\bl e^h =\bl \phi^h - \Pi^h \bl \phi$, we have
\begin{alignat}{2}
      \dnorm{\bl e^h}^2 &= b(\bl e^h,\bl e^h) && (\text{by Lemma \ref{lem:stability}}) \nonumber \\
      &  = -b(\bl e^h,\bl e + \bl \eta) \nonumber \\ 
      &  = -b(\bl e^h,\bl \eta)&& (\text{by Lemma \ref{lem:consistency}}) \nonumber \\ 
      & \le \snorm{b(\bl e^h,\bl \eta)} \nonumber \\ 
      & \le \snorm{ \inner{\bl e^h, \bl \Omega \bl \eta + \bl A_i \frac{\partial \bl \eta}{\partial x_i}}_\Omega + \inner{ \frac{\partial \bl e^h}{\partial x_i}, \kappa \frac{\partial \bl \eta}{\partial x_i}}_\Omega + \inner{\res(\bl e^h), \bl \tau \res(\bl \eta)}_\Ot}  && (\text{by definition \eqref{3D-weak-def}}) \nonumber \\ 
      & \le \Bigg| -\inner{\bl \eta,\res(\bl e^h)}_\Ot + \inner{\bl e^h,\An\bl \eta}_\Gh - \inner{\bl \eta , \kappa \frac{\partial^2 \bl e^h}{\partial x_i \partial x_i}}_\Ot  \nonumber \\ 
      &  \hspace{0.25in} + \inner{ \frac{\partial \bl e^h}{\partial x_i}, \kappa \frac{\partial \bl \eta}{\partial x_i}}_\Omega +  \inner{\res(\bl e^h), \bl \tau \res(\bl \eta)}_\Ot \Bigg|  && (\text{by integration-by-parts}) \nonumber \\ 
      & \le  \frac{1}{4}\norm{\res(\bl e^h)}_{\bl \tau,\Ot}^2 + \norm{\bl \eta}_{\bl \tau^{-1}, \Ot}^2 + \frac{1}{4}\snorm{\bl e^h}_{\An,\Gh}^2 + \snorm{\bl \eta}_{\An,\Gh}^2 +  \frac{1}{4}\norm{\kappa \frac{\partial^2 \bl e^h}{\partial x_i \partial x_i}}_{\bl \tau,\Ot}^2  \nonumber \\ 
      &  \hspace{0.15in} +  \norm{\bl \eta}_{\bl \tau^{-1},\Ot}^2 +  \frac{1}{4}\kappa \norm{ \frac{\partial \bl e^h}{\partial x_i}}_\Omega^2 + \kappa \norm{\frac{\partial \bl \eta}{\partial x_i}}_\Omega^2 + \frac{1}{4} \norm{\res(\bl e^h)}_{\bl \tau,\Ot}^2 + \norm{ \res(\bl \eta)}_{\bl \tau,\Ot}^2 \nonumber \\ 
      & \le 2\norm{\bl \eta}_{\bl \tau^{-1}, \Ot}^2  + \snorm{\bl \eta}_{\An,\Gh}^2 +\kappa \norm{\frac{\partial \bl \eta}{\partial x_i}}_\Omega^2  + \norm{ \res(\bl \eta)}_{\bl \tau,\Ot}^2 +\frac{1}{2}\dnorm{\bl e^h}^2 && (\text{by Lemma \ref{lem:stability} and \ref{lem:tau-ubd}}) \nonumber \\ 
      & \le 2\left[2\norm{\bl \eta}_{\bl \tau^{-1}, \Ot}^2  + \snorm{\bl \eta}_{\An,\Gh}^2 + \kappa \norm{\frac{\partial \bl \eta}{\partial x_i}}_\Omega^2  + \norm{ \res(\bl \eta)}_{\bl \tau,\Ot}^2 \right] \nonumber \\ 
     & \le (c_1 +c_2 \alpha+c_3 \beta^4)\kappa h^{2p} \snorm{\bl \phi}_{H^{p+1}}^2. && (\text{by Lemma \ref{lem:interp}})
       \label{error_e_1}
\end{alignat}
Thus
\begin{alignat*}{2}
      \dnorm{\bl e}^2 &= \dnorm{-\bl e^h -\bl \eta}^2 &&  \\
      &  \le c (\dnorm{\bl e^h}^2 + \dnorm{\bl \eta}^2) && (\text{by triangular inequality}) \nonumber \\ 
      &  \le (c_1 +c_2 \alpha+c_3 \beta^4)\kappa h^{2p} \snorm{\bl \phi}_{H^{p+1}}^2. \hspace{2in} && (\text{by \eqref{error_e_1} and Lemma \ref{lem:interp}})
\end{alignat*}
\end{proof}

Before discussing the implications of the above theorem, we prove an additional lemma to show that $\dnorm{\cdot}$ provides an adequate measure of the method's error if we are interested in an estimate in terms of the norm of the acceleration, convection, and diffusion terms.

\begin{lem} \label{lem:adq}
(Adequacy of $\dnorm{\cdot}$). Assuming Lemma \ref{lem:stability} and \ref{lem:tau-ubd} hold, then for a given $\bl w^h \in \bl W^h$ 
\begin{equation*}
       \dnorm{\bl w^h}^2 \ge \frac{1}{2}\snorm{\bl w^h}_{\An,\Gh}^2 + \frac{1}{2}\kappa \norm{\frac{\partial \bl w^h}{\partial x_i}}^2_\Omega + \frac{1}{2} \norm{\bl \Omega \bl w^h}^2_{\bl \tau,\Ot} + \frac{1}{2} \norm{ \bl A_i \frac{\partial \bl w^h}{\partial x_i}}^2_{\bl \tau,\Ot}.
\end{equation*}
\end{lem}

\begin{proof}
\begin{alignat}{2}
       \dnorm{\bl w^h}^2 &= \frac{1}{2}\snorm{\bl w^h}_{\An,\Gh}^2 + \kappa \norm{\frac{\partial \bl w^h}{\partial x_i}}^2_\Omega + \norm{\bl \Omega \bl w^h+ \bl A_i \frac{\partial \bl w^h}{\partial x_i}}^2_{\bl \tau,\Ot} + \norm{\kappa \frac{\partial^2 \bl w^h}{\partial x_i \partial x_i}}^2_{\bl \tau,\Ot} \nonumber \\ 
        &- 2 \inner{ \bl \Omega \bl w^h + \bl A_i \frac{\partial \bl w^h}{\partial x_i},\bl \tau \kappa \frac{\partial^2 \bl w^h}{\partial x_i \partial x_i} }_\Ot  && (\text{by Lemma \ref{lem:stability}}) \nonumber \\ 
        &\ge \frac{1}{2}\snorm{\bl w^h}_{\An,\Gh}^2 + \kappa \norm{\frac{\partial \bl w^h}{\partial x_i}}^2_\Omega + \frac{1}{2} \norm{\bl \Omega \bl w^h+ \bl A_i \frac{\partial \bl w^h}{\partial x_i}}^2_{\bl \tau,\Ot} - \norm{\kappa \frac{\partial^2 \bl w^h}{\partial x_i \partial x_i}}^2_{\bl \tau,\Ot} \nonumber \\
        &= \frac{1}{2}\snorm{\bl w^h}_{\An,\Gh}^2 + \kappa \norm{\frac{\partial \bl w^h}{\partial x_i}}^2_\Omega + \frac{1}{2} \norm{\bl \Omega \bl w^h}^2_{\bl \tau,\Ot} + \frac{1}{2} \norm{ \bl A_i \frac{\partial \bl w^h}{\partial x_i}}^2_{\bl \tau,\Ot} \nonumber \\
        & + \inner{\bl \Omega \bl w^h, \bl \tau \bl A_i \frac{\partial \bl w^h}{\partial x_i} }_\Ot - \norm{\kappa \frac{\partial^2 \bl w^h}{\partial x_i \partial x_i}}^2_{\bl \tau,\Ot}.
       \label{norm_simpl}
\end{alignat}
Next, by following a process similar to that of \eqref{lem1_2} and \eqref{lem1_3}, we can show that the last inner product in \eqref{norm_simpl} is zero. 
Defining $\bl M_i \coloneqq \hat i \bl \Omega \bl \tau \bl A_i$ and $\bl M_{\rm n} \coloneqq \bl M_i n_i = \hat i \bl \Omega \bl \tau \An$, we have $\bl M_i \le 0$ and $\bl M_i^H = \bl M_i$, and also $\bl M_{\rm n} \le 0$ and $\bl M_{\rm n}^H = \bl M_{\rm n}$.
Thus
\begin{alignat}{2}
       \inner{\bl \Omega \bl w^h, \bl \tau \bl A_i \frac{\partial \bl w^h}{\partial x_i} }_\Ot & = \hat i \inner{\bl w^h, \bl M_i \frac{\partial \bl w^h}{\partial x_i} }_\Ot \nonumber \\ 
       & = \hat i \inner{\bl w^h ,\bl M_{\rm n}  \bl w^h }_{\Gamma_e} - \hat i \inner{\frac{\partial \bl w^h}{\partial x_i}, \bl M_i \bl w^h }_\Ot  \hspace{0.5in} && (\text{by integration-by-parts}) \nonumber\\ 
       & = \frac{\hat i}{2} \inner{\bl w^h ,\bl M_{\rm n}  \bl w^h }_{\Gamma_e} \nonumber  \\
       & = \frac{\hat i}{2} \int_{\Gamma_e} (\bl w^h)^H \bl M_{\rm n}  \bl w^h d\Gamma   && (\text{by definition \eqref{inner_def}}) \nonumber \\
       & = -\frac{\hat i}{2} \int_{\Gamma_e} (\bl w^h)^H \bl M_{\rm n}^H  \bl w^h d\Gamma && (\text{by symmetry}) \nonumber \\
       & = 0.
       \label{norm_simpl1}
\end{alignat}
The intended result follows from \eqref{norm_simpl} and \eqref{norm_simpl1} in combination with Lemma \ref{lem:tau-ubd}.
\end{proof}

Before ending this section, let us summarize the practical implications of what was discussed above in the form of a few remarks. 
\paragraph{Remarks:}
\begin{enumerate}
\item The above result establishes the stability, consistency (thereby convergence), and accuracy of the GLS method for convective diffusive problems as formulated by \eqref{3D-weak}. The requirements for these properties to hold are relatively benign: no backflow through Neumann boundaries ($\An \ge 0$), $C_{\rm I}$ should be large enough due to \eqref{ci_cinv} so that the method is not overly diffusive, and the mesh should not be too distorted so that the eigenvalues of the metric tensor can be bounded from above via \eqref{GH_rel}. 
\item It is not too difficult to meet these three requirements. As we discuss later in Section \ref{sec:bfs}, the first requirement on backflow can be relaxed if the weak form is modified appropriately. The second requirement is satisfied if $C_{\rm I}$ is selected properly for a given element type at the implementation stage. The third requirement is imposed on the user of the GLS method, which if not met, will lower the order of accuracy of the GLS method by making $C\gamma$ dependent on $h$. This dependence, however, is not unique to the GLS method as \emph{any} method will experience a lower order of convergence if the element aspect ratio were to blow up during mesh refinement. 
\item Given Theorem \ref{thm:error}, we can analyze the behavior of the GLS in three regimes. First is the diffusive regime where $\alpha \ll 1$ and $\beta \ll 1$. Second is the convective regime where $\alpha \gg 1$. And third is the oscillatory regime where $\beta \gg 1$. We don't consider $\alpha \gg 1$ and $\beta \gg 1$ as a separate case since the error will be dominated either by $c_2 \alpha$ or $c_3 \beta^4$ term, placing us in one of the enumerated regimes. 
\item The diffusive regime is produced, e.g., when the mesh is sufficiently refined. In this regime, the order of convergence will be optimal. In the natural $L_2$-norm, the GLS method will be $\norm{\bl e} \propto h^{p+1}$. That is compatible with results reported in \cite{esmaily2023stabilizedA}, showing the method is second order when using linear shape functions. 
\item In the convective regime, one can neglect the terms associated with the Eulerian acceleration, producing a behavior that is similar to the conventional upwinding stabilized methods that are formulated in time~\cite{hughes1986generalized,shakib1991new}. Further analysis of the error in such regimes produces an $L_2$ error estimate that scales with $h^{p+1/2}$ \cite{Hughes1986beyond,johnson1984finite}. That translates to a sub-optimal convergence rate of 1.5 when linear shape functions are used. 
\item In the oscillatory regime, we see $\dnorm{\bl e}\propto \beta^2 h^p$, namely the error in that norm should grow proportional to the frequency $\omega$. This behavior is similar to what has been reported for the GLS method when it is applied to the Helmholtz equation \cite{harari1990design}. This similarity is no coincidence given that the solution to these two problems exhibits similar characteristics in this oscillatory regime. As we discussed earlier, the element Womersley number for the target application of the present method, namely cardiovascular flows, is not expected to be very large. Thus, the present time-spectral GLS method is not expected to operate in the oscillatory regime, viz., it should behave similarly to its conventional temporal counterpart for the convective-diffusive problems. 
\end{enumerate}

\subsection{Time-spectral Navier-Stokes equations} \label{sec:NS}
The numerical solution of the Navier-Stokes equations for modeling the flow of incompressible Newtonian fluid in $d$-dimensions can be considered as the convection-diffusion of $d$ scalar quantities (i.e., velocities) that are subjected to a Lagrangian constraint (the continuity equation). 
Therefore, the GLS method constructed in Section \ref{sec:AD} provides a scaffold for the Navier-Stokes equations. 
In this section, we discuss the generalization of that method to the Navier-Stokes equations while taking into account the effect of the incompressibility constraint. 

The three-dimensional incompressible Navier-Stokes equation in the Cartesian coordinate system and the time domain is formulated as  
\begin{equation}
    \begin{alignedat}{3}
    \rho \frac{\partial \hat u_i}{\partial t} + \rho \hat u_j \frac{\partial \hat u_i}{\partial x_j} &= - \frac{\partial \hat p}{\partial x_i} + \frac{\partial }{\partial x_j}\left(\mu \frac{\partial \hat u_i}{\partial x_j}\right), \hspace{0.5in} && \text{in}\;\;\; && \Omega\times(0,N_cT] \\
    \frac{\partial \hat u_i}{\partial x_i}  &= 0,\;\;\; && \text{in} && \Omega\times(0,N_cT]  \\
    \hat u_i &= \hat g_i,  && \text{on} && \Gg\times(0,N_cT]  \\
    -\hat pn_i + \mu \frac{\partial \hat u_i}{\partial x_j}n_j  &= \hat h n_i, && \text{on} && \Gh\times(0,N_cT] 
    \end{alignedat}
 \label{NS-time}
\end{equation}
where $\rho \in \mathbb R^+$ and $\mu \in \mathbb R^+$ are the given fluid density and dynamic viscosity, respectively, $\hat u_i(\bl x, t) \in \mathbb R$ for $i=1, \cdots, d$ and $\hat p(\bl x, t) \in \mathbb R$ are the unknown fluid velocity and pressure at point $\bl x$ and time $t$, respectively, $\hat g_i(\bl x, t) \in \mathbb R$ and $\hat h(\bl x, t) \in \mathbb R$ are the imposed Dirichlet and Neumann boundary conditions, respectively, and $N_c$ and $T$ are the number and duration of cardiorespiratory cycles to be simulated, respectively. The definition of the Neumann boundary condition in \eqref{NS-time}, which simplifies the following derivation by decoupling velocity components, is adopted from the literature \cite{Vignon2010625, Bazilevs20093534}. This choice may be justified for cardiovascular applications by the fact that the deviation of pressure from the prescribed value due to the viscous term in \eqref{NS-time} is often negligible.

To solve \eqref{NS-time} in the time-spectral domain, we first discretize the solution and boundary conditions using 
\begin{equation}
\begin{split}
    \hat u_i(\bl x,t) &= \sum_{|n|<N} u_{ni}(\bl x) e^{\hat in\omega t}, \\ 
    \hat p(\bl x,t) &= \sum_{|n|<N} p_n(\bl x) e^{\hat in\omega t},      \\ 
    \hat g_i(\bl x,t) &= \sum_{|n|<N} g_{ni}(\bl x) e^{\hat in\omega t},  \\ 
    \hat h(\bl x,t)& = \sum_{|n|<N} h_n(\bl x) e^{\hat in\omega t},    
\end{split}
    \label{upgh_def}
\end{equation} 
where
\begin{align*}
     g_{ni}(\bl x) & \coloneqq \frac{1}{T}\int_0^T \hat g_i(\bl x,t) e^{-\hat in\omega t} dt,  \\ 
     h_n(\bl x) & \coloneqq \frac{1}{T}\int_0^T \hat h(\bl x, t) e^{-\hat in\omega t} dt,
\end{align*} 
are the Fourier coefficients computed such that the error in the approximations made in discretizing the boundary conditions in \eqref{upgh_def} are minimized.  

To arrive at an expression similar to \eqref{3D-spec}, we organize unknowns and boundary conditions in vectors as 
\begin{equation}
    \bl u_i(\bl x) \coloneqq \left[ 
    \begin{matrix} u_{(-N+1)i} \\ \vdots\\ u_{0i} \\ \vdots \\ u_{(N-1)i}\end{matrix} \right], \hspace{0.15in}
    \bl p(\bl x) \coloneqq \left[ \begin{matrix} p_{-N+1} \\ \vdots \\ p_{0} \\ \vdots \\ p_{N-1} \end{matrix} \right],  \hspace{0.15in}
    \bl g_i(\bl x) \coloneqq \left[ \begin{matrix} g_{(-N+1)i}\\ \vdots\\ g_{0i}\\ \vdots\\ g_{(N-1)i} \end{matrix} \right], \hspace{0.15in}
    \bl h(\bl x) \coloneqq \left[ \begin{matrix} h_{-N+1}\\ \vdots \\ h_{0} \\ \vdots \\ h_{N-1} \end{matrix} \right]. 
\label{upghDef}
\end{equation} 

With these definitions, \eqref{NS-time} is expressed in the time-spectral domain as
\begin{equation}
\begin{alignedat}{3}
    \rho \bl \Omega \bl u_i + \rho \bl A_j \frac{\partial \bl u_i}{\partial x_j} &= - \frac{\partial \bl p}{\partial x_i} + \frac{\partial }{\partial x_j}\left(\mu \frac{\partial \bl u_i}{\partial x_j}\right),\hspace{0.5in} && \text{in}\;\;\; && \Omega  \\
    \frac{\partial \bl u_i}{\partial x_i}  &= 0, && \text{in} && \Omega \\
    \bl u_i &= \bl g_i,  && \text{on} && \Gg  \\
    -\bl pn_i + \mu \frac{\partial \bl u_i}{\partial x_j}n_j  &= \bl h n_i, && \text{on} && \Gh   
\end{alignedat}
\label{NS-spec}
\end{equation}
 where the convolution and time-derivative matrices, $\bl A_i$ and $\bl \Omega$, respectively, are the same as those appearing in \eqref{3D-spec}. 

An important note must be made regarding the existence and uniqueness of the solution to \eqref{NS-spec}.
In this form, where $\bl \Omega$ is prescribed based on $\omega$, these equations may or may not have a unique solution. 
A unique solution exists, e.g., when modeling a laminar pipe flow \cite{womersley1955method}. 
A solution does not exist if, for instance, the flow base frequency corresponding to an inherent flow instability is not divisible by the prescribed $\omega$. 
In this case, letting $\omega$ to be a free finite parameter may or may not yield a unique solution (we are not considering $\omega\to0$ here as that recovers the original Navier-Stokes equations). 
In this case, a unique solution will exist if the instability is characterized by a single frequency. 
No unique solution can be found for the same problem, however, if instabilities occur at two frequencies with no common divisor (e.g., 1 and $\pi$). 
While allowing for multiple frequencies appears to allow for a solution in such cases, a formal proof is required before that statement can be made confidently. 
Furthermore, these statements may or may not translate to the discrete setting as a numerical method may always produce ``a'' solution even if the underlying differential equation has no solution.

The time-spectral form of the Navier-Stokes equations as written is intended to resolve the particular solution after it has reached cycle-to-cycle convergence. 
That is in contrast to the time formulation, where the solution is calculated for $t\in (0, N_c]T$ with $N_c$ being sufficiently large to ensure the homogeneous solution associated with the arbitrary initial conditions is ``washed-out''. 
Assuming that is the case, the times-spectral formulation will produce a solution that corresponds to $t\in(N_c-1, N_c]T$. 
By avoiding the wasteful simulation of the first $N_c-1$ cardiorespiratory cycles, the time-spectral method gains a significant cost advantage over the time formulation. 
This gap in cost advantage widens for problems with $N_c \gg 1$ that are slow to converge cyclically. 

The solution strategy employed for \eqref{NS-spec} is identical to that of \eqref{3D-spec} except for the fact that the computed velocity unknowns are subjected to an incompressibility constraint. 
This constraint must be handled with special care as the finite element solution of these equations using Galerkin's method with equal order shape functions for velocity and pressure is bound to fail as it does not satisfy Babuska-Brezzi stability condition~\cite{ladyzhenskaya1969mathematical, babuvska1971error, brezzi1974existence}. 
As numerous studies have shown in the past \cite{barbosa1991finite, hughes1986circumventing, tezduyar1991stabilized}, this issue can be circumvented if the continuity equation is modified appropriately at the discrete level to make it directly dependent on the pressure. 
In the conventional variational multiscale (VMS) method this dependence is introduced by modeling the unresolved velocity in the continuity equation via the pressure-dependent residual of the momentum equation \cite{hughes2007variational, bazilevs2007variational}. 
As a result, the discrete Laplacian of pressure appears in the continuity equation, thereby producing a pressure-stable method that permits convenient use of the same spatial discretization for velocity and pressure.
The GLS framework also produces a similar pressure-stabilization term, thereby endowing it with benefits similar to those of the standard stabilized finite element methods for fluids. 

With these considerations, the proposed GLS method for the time-spectral form of the Navier-Stokes equations is stated as finding $\bl u_i^h$ and $\bl p^h$ such that for any test functions $\bl w_i^h$ and $\bl q^h$ we have
\begin{equation}
\begin{split}
       & \overbrace{\inner{\bl w^h_i,\rho \bl \Omega \bl u^h_i + \rho \bl A_j \frac{\partial \bl u^h_i}{\partial x_j}}_\Omega + \inner{\frac{\partial\bl w^h_i}{\partial x_j}, -\bl p^h \delta_{ij} + \mu \frac{\partial \bl u^h_i}{\partial x_j}}_\Omega  + \inner{\bl q^h, \frac{\partial \bl u^h_i}{\partial x_i}}_\Omega}^\text{Galerkin's terms} \\
       + &\underbrace{\inner{\res_i(\bl w^h_i, \bl q^h),\frac{\bl \tau}{\rho} \res_i(\bl u^h_i, \bl p^h)}_\Ot}_\text{Least-square penalty terms} = \underbrace{ \inner{\bl w_i^h,  \bl h n_i }_\Gh,}_\text{Neumann BCs. terms}
\end{split}
\label{NS-weak}
\end{equation}
where  
\begin{equation*}
    \res_i(\bl u^h_i, \bl p^h) \coloneqq \rho \bl \Omega \bl u_i^h + \rho \bl A_j \frac{\partial \bl u^h_i}{\partial x_j} + \frac{\partial \bl p^h}{\partial x_i} - \frac{\partial}{\partial x_j} \left(\mu \frac{\partial \bl u^h_i}{\partial x_j}\right),
\end{equation*}
is the momentum equation differential operator. 
Note that the terms corresponding to the Dirichlet boundary conditions are not included in \eqref{NS-weak} as $\bl u_i^h=\bl g_i$ and $\bl w_i^h=\bl 0$ on $\Gg$ are directly built into the solution and trial functions spaces, respectively. 

The design of $\bl \tau$ in \eqref{NS-weak}, which plays a crucial role in the stability and accuracy of the GLS method above, is identical to that of the convection-diffusion equation in \eqref{3D-tau}.  
In fact, the properties of the GLS method in \eqref{NS-weak} can be studied by following a process similar to that of Section \ref{sec:analysis}. 
As shown in Appendix \ref{sec:prop-ns}, this extension is facilitated by the fact that the incompressible Navier-Stokes equations in \eqref{NS-spec} can be expressed in the form of a system of convection-diffusion equations through a change of variable.

The least-squares terms in \eqref{NS-weak} encapsulate two important terms. 
The first is the streamline upwind Petrov/Galerkin (SUPG) term $\rho \bl A_j\frac{\partial \bl w^h_i}{\partial x_j} \bl \tau \bl A_k\frac{\partial \bl u^h_i}{\partial x_k}$, which stabilizes this method in the presence of strong convection.  
The second is the pressure stabilized Petrov/Galerkin (PSPG) term $\frac{1}{\rho} \frac{\partial \bl q^h}{\partial x_i} \bl \tau \frac{\partial \bl p^h}{\partial x_i}$, which circumvents the inf-sup condition to permit the use of equal order interpolation functions for velocity and pressure. 
These terms are very similar to the corresponding terms that appear in the conventional time formulation of the Navier-Stokes equation (Appendix \ref{sec:app}). 
The key difference, however, is that for a single unknown in the conventional method, we have $2N-1$ (or $N$ independent) unknown in \eqref{NS-weak}. 
This difference in the number of unknowns can quickly increase the cost of the present approach, thus requiring some form of optimization to ensure this method remains cost-competitive relative to the conventional time formulation.

Before describing any optimization strategy, we first provide an overview of a standard solution procedure for \eqref{NS-weak}. In what follows, we assume all interpolation functions are linear, hence terms involving their second derivative are neglected. 

The test functions $\bl w^h_i$ and $\bl q^h$ in \eqref{NS-weak} are discretized in space using
\begin{align*}
    \bl w^h_i(\bl x) & = \sum_A \bl c^{\rm w}_{Ai} N_A(\bl x), \\ 
    \bl q^h(\bl x) & = \sum_A \bl c^{\rm q}_{A} N_A(\bl x), 
\end{align*}
where $N_A(\bl x)$ is the interpolation function associated with node $A$. 
Since $\bl w^h_i$ and $\bl q^h$ are arbitrary functions, \eqref{NS-weak} must hold for any $\bl c^{\rm w}_{Ai}$ and $\bl c^{\rm q}_A$ arbitrary vectors. 
That permits us to obtain a system of equations from \eqref{NS-weak}, which in 3D are
\begin{equation}
\begin{split}
      \bl r^{\rm m}_{Ai} = & \int_\Omega \left(\rho N_A \bl \Omega \bl u^h_i + \rho N_A \bl A_j \frac{\partial \bl u^h_i}{\partial x_j} - \frac{\partial N_A}{\partial x_i}\bl p^h + \mu \frac{\partial N_A}{\partial x_j} \frac{\partial \bl u^h_i}{\partial x_j} \right)d\Omega  - \int_\Gh N_A\bl h n_i d\Gamma  \\
       + & \int_\Ot\left(- \bl \Omega N_A + \bl A_j \frac{\partial N_A}{\partial x_j}\right)  \bl \tau \res_i(\bl u^h_i, \bl p^h) d\Omega = \bl 0, \hspace{1.3cm} A\in\eta-\eta_g,\; i=\text{1, 2, and 3,}  \\
       \bl r^{\rm c}_A = & \int_\Omega N_A \frac{\partial \bl u^h_i}{\partial x_i}d\Omega + \int_\Ot \frac{1}{\rho}\frac{\partial N_A}{\partial x_i}\bl \tau \res_i(\bl u^h_i, \bl p^h) d\Omega = \bl 0, \hspace{1cm} A\in\eta. 
\end{split}
       \label{Rmc}
\end{equation}
In \eqref{Rmc}, $\eta$ and $\eta_g$ denote the set of nodes in the entire domain $\Omega$ and those located on the Dirichlet boundaries $\Gg$, respectively. 

In total, \eqref{Rmc} represents $(2N-1)(3|\eta-\eta_g| + |\eta|)$ complex-valued equations, where $|\eta|$ denotes the set $\eta$ cardinality. As we see later, this system can be reduced to $(2N-1)(3|\eta-\eta_g| + |\eta|)$ real-valued equations given the dependency of those complex-valued equations. 

We also discretize the velocity and pressure using the same interpolation functions as those of the test functions. Namely
\begin{align*}
    \bl u_i^h(\bl x) & = \sum_A \bl u^{\rm d}_{Ai} N_A(\bl x), \\ 
    \bl p^h(\bl x) & = \sum_A \bl p^{\rm d}_{A} N_A(\bl x), 
\end{align*}
where $\bl u^{\rm d}_{Ai}$ and $\bl p^{\rm d}_A$ contain velocity and pressure, respectively, at all modes at node $A$.
In our implementation, we build the Dirichlet boundary condition into the unknown vector so that $\bl u^{\rm d}_{Ai} = \bl g_i(\bl x_A)$, where $\bl x_A$ is the position of node $A$.  

At the discrete level, our goal is to find $\bl u^{\rm d}_{Ai}$ and $\bl p^{\rm d}_A$ such that all equations in \eqref{Rmc} are satisfied. 
Given that these equations are nonlinear, this feat is accomplished in an iterative process using the Newton-Raphson iterations. 
More specifically, organizing all unknowns and equations into vectors 
\begin{equation*}
    \bl y \coloneqq \left[\begin{matrix} \bl u^{\rm d}_1 \\ \bl u^{\rm d}_2 \\ \bl u^{\rm d}_3 \\ \bl p^{\rm d} \end{matrix}\right],  \hspace{0.2in}
    \bl r \coloneqq \left[\begin{matrix} \bl r^{\rm m}_1 \\ \bl r^{\rm m}_2 \\ \bl r^{\rm m}_3 \\ \bl r^{\rm c} \end{matrix}\right],  
\end{equation*}
we solve 
\begin{equation}
    \bl y^{(n+1)} = \bl y^{(n)} - \left(\bl H^{(n)}\right)^{-1}\bl r^{(n)}, 
   \label{NR}
\end{equation}
at each Newton-Raphson iteration $n$ to update the solution from the last iteration $\bl y^{(n)}$ and compute it at the next iteration $\bl y^{(n+1)}$. 
In \eqref{NR}, $\bl r^{(n)}$ denotes the residual vector calculated based on the unknowns at the last iteration $n$ using \eqref{Rmc}. 

The tangent matrix $\bl H^{(n)}$ in \eqref{NR}, which is also dependent on $\bl y^{(n)}$ due to the problem nonlinearity, is computed at each iteration as 
\begin{equation}
\bl H = \frac{\partial \bl r}{\partial \bl y} =  \left[\begin{matrix}
       \bl K & \bl 0 & \bl 0 & \bl G_1 \\     
       \bl 0 & \bl K & \bl 0 & \bl G_2 \\     
       \bl 0 & \bl 0 & \bl K & \bl G_3 \\     
       \bl D_1 & \bl D_2 & \bl D_3 & \bl L \\     
       \end{matrix} \right],
       \label{tangent}
\end{equation}
where the superscript $(n)$ is dropped to simplify the notation.  

The tangent matrix blocks in \eqref{tangent} are explicitly calculated from \eqref{Rmc} and are 
\begin{equation}
\begin{split}
       \bl K_{AB} & = \int_\Omega\left[ \rho N_A \bl \Omega N_B + \rho N_A \bl A_k \frac{\partial N_B}{\partial x_k} + \mu \frac{\partial N_A}{\partial x_k} \frac{\partial N_B}{\partial x_k} \bl I \right]d\Omega \\ 
       & + \int_\Ot \left[\left(-\bl \Omega N_A +  \bl A_k \frac{\partial N_A}{\partial x_k} \right)\rho \bl \tau \left(\bl \Omega N_B  +\bl A_k  \frac{\partial N_B}{\partial x_k} \right)\right] d\Omega,  \\ 
       (\bl G_{AB})_i & = -\int_\Omega \frac{\partial N_A}{\partial x_i} N_B\bl I d\Omega + \int_\Ot \left[ \left(-\bl \Omega N_A + \bl A_k \frac{\partial N_A}{\partial x_k}\right) \bl \tau \frac{\partial N_B}{\partial x_i} \right]d\Omega,  \\ 
       (\bl D_{AB})_j & = \int_\Omega N_A \frac{\partial N_B}{\partial x_j} \bl I d\Omega + \int_\Ot \left[ \frac{\partial N_A}{\partial x_j}\bl \tau \left(\bl \Omega N_B + \bl A_k \frac{\partial N_B}{\partial x_k} \right) \right]d\Omega,  \\
       \bl L_{AB} & = \int_\Ot  \frac{1}{\rho} \frac{\partial N_A}{\partial x_k}\bl \tau\frac{\partial N_B}{\partial x_k} d\Omega.  
\end{split}
\label{K_mat}
\end{equation}

\paragraph{Remarks:} 
\begin{enumerate}
    \item The Newton-Raphson iterations must be initialized at $n=0$ with a guess for all unknowns. In our implementation we take $\bl u^{\rm d(0)}_{Ai} = \bl 0$ for $A\in \eta-\eta_g$, $\bl u^{\rm d(0)}_{Ai} = \bl g_i(\bl x_A)$ for $A\in \eta_g$, and $\bl p^{\rm d(0)}_A = \bl 0$ for $A\in \eta$. 
    \item Dirichlet boundary conditions are imposed at the linear solver level in our implementation so that $\bl u_{Ai}^{{\rm d}(n+1)} = \bl u^{{\rm d}(n)}_{Ai}$ for $A\in \eta_g$. This condition ensures that those boundary conditions remain strongly enforced after proper initialization at $n=0$. 
    \item The Newton-Raphson iterations are terminated when the error falls below a certain tolerance. Namely, iterations in \eqref{NR} are terminated once $\|\bl r^{(n)}\| < \epsilon_{\rm NR}\|\bl r^{(0)}\|$, with $\epsilon_{\rm NR}$ being a user-specified tolerance.
    \item All blocks of the tangent matrix must be updated at each Newton-Raphson iteration for their direct or indirect dependence on $\bl A_i$ through $\bl \tau$. 
    \item In deriving an explicit form for $\bl H$, the dependence of $\bl A_i$ and $\bl \tau$ on $\bl u_i^h$ were neglected. The resulting form of the tangent matrix, although approximate, improves the convergence rate of the Newton-Raphson iterations. This behavior is compatible with the conventional formulation, where the convective velocity and stabilization parameters are frozen in time during the Newton-Raphson iterations for improved convergence within a given time step. 
    \item Solving \eqref{NR} is the most expensive step of the entire algorithm as it requires solving a large linear system of equations. This system is also solved iteratively up to a specified tolerance $\epsilon_{\rm LS}$. We found a relatively large value of $\epsilon_{\rm LS}=0.05$ to be an optimal choice for cases considered here. Although a tighter tolerance improves the convergence rate of the Newton-Raphson method, the cost of those iterations increases at a yet higher rate to produce an overall slower method. Although this optimal value of $\epsilon_{\rm LS}$ has proved relatively universal across many cases considered, we acknowledge that it may differ from what is reported here if one were to consider a case different from those considered in this study.
\end{enumerate}
As described, this algorithm and in particular solving the linear system in \eqref{NR} as is, will be cost-prohibitive.
In the next section, we discuss three optimization strategies that significantly lower the computational cost of the GLS method for fluids.  

\subsection{Optimization of the GLS method for the Navier-Stokes} \label{sec:opt}
As is the case with the conventional time formulation of the Navier-Stokes equation, the majority of the cost of the GLS algorithm described above is associated with solving the underlying linear system. 
Thus, our first effort in optimizing the GLS method is focused on reducing the cost of the linear solver. 

To solve the linear system involving $\bl H$ in \eqref{NR}, we employ the GMRES iterative method \cite{saad1986gmres}, which is a robust Krylov subspace method designed for non-symmetric matrices. 
The bulk of the cost of the GMRES method, similar to other iterative methods, is associated with the matrix-vector product operation and to a lower degree, inner products between vectors, the norm of a vector, and the scaling of a vector. 
By reducing the cost of these operations, one can substantially reduce the overall cost of the iterative solver and thereby the GLS method proposed above. 

By targeting a faster solution of the linear system through the three strategies discussed below, we reduce the cost of the GLS method by about an order of magnitude relative to the baseline method from Section \ref{sec:NS}. 

\subsubsection{Using real-valued unknowns}
The first strategy, which targets faster operations at the linear solver level, involves reorganization of the unknowns so that we only solve for the independent unknowns. 
As formulated, \eqref{NR} contains $4(2N-1)|\eta|$ complex-valued unknowns, out of which only half are independent. 
That is so, since velocity and pressure in \eqref{upghDef} are formulated in terms of a variable and its complex conjugate.
This choice significantly simplifies the formulation of the GLS method and its implementation in the finite element solver. 
But it also produces redundant calculations at the linear solver level. 
To avoid this redundancy, we formulate the linear system (i.e., $\bl H$, $\bl r$, and $\bl y$) in terms of the real and imaginary elements of the first $N$ modes. 
The result is a linear system with $8N|\eta|$ real-valued unknowns (the zero steady imaginary component is retained to simplify implementation), which has a memory footprint that is approximately half of the original system. 
On a platform in which numerical operations on complex data types cost roughly double that of the real data type, this optimization cuts the cost of solving the system to approximately half of that of the original system. 
That has been the case in all platforms we tested. 

\subsubsection{Exploiting the structure of $\bl H$} \label{sec:Hstructure}
The second optimization, which is also implemented at the linear solver level, targets the matrix-vector product operation by exploiting the structure of $\bl H$.  
Out of 16 blocks of $\bl H$, 6 blocks are identically zero in \eqref{tangent}. 
By writing a matrix-vector product kernel that is tailored to this system, we avoid performing operations that are zero. 
Secondly, the diagonal blocks for the momentum equation do not depend on the direction. 
Thus, $\bl K$ is retrieved once from the memory and reused three times in the matrix-vector product operation. 
Thirdly, the contribution of the least-squares terms to $\bl G_i$ and $\bl D_i$ in \eqref{K_mat} can be neglected without harming (even improving) the overall convergence rate of the linear solver and the Newton-Raphson iterations. 
Doing so produces $\bl G_i$ and $\bl D_i$ matrices that are diagonal with regard to the mode number. 
By storing and performing only on those diagonal entries, the cost of the gradient and divergence operations are reduced from $24 N^2 N_{\rm nz}$ to $12N N_{\rm nz}$ real-valued operations, where $N_{\rm nz}$ is the total number of connections between nodes in the mesh (i.e., number of non-zeros in the tangent matrix for a scalar problem). 
Putting all these together, this optimization reduces the cost of a matrix-vector product from $64 N^2 N_{\rm nz}$ to approximately $4N(3+cN) N_{\rm nz}$ with $2\le c \le 4$ depending on how close these calculations are to being memory bound.
That implies a severalfold reduction in the cost of matrix-vector multiplication when $N$ is sufficiently large. 
In practice, the overall cost saving, which is an aggregate average of various operations, is more modest, amounting to approximately a factor of 2 reduction in cost for the problem considered in Section \ref{sec:result}. 

\subsubsection{Pseudo-time stepping}
The last optimization strategy employed in this study is to utilize a pseudo-time stepping method. 
This method is adopted to improve the condition number of $\bl H$ by adding a mass matrix to its diagonal blocks~\cite{fried1972bounds} and thereby reducing the cost of the linear solver. 
Unlike the two other optimization strategies that affected the linear solver only, the GLS formulation at the solver level is modified in this case. 
Nevertheless, the converged solution remains identical to the baseline method as the pseudo-time stepping only alters the path to the final solution. 
The implementation of this optimization strategy is simple. 
It involves adding 
\begin{equation*}
    \inner{\bl w^h_i,\rho\frac{\partial \bl u^h_i}{\partial \tilde t} }_\Omega, 
\end{equation*}
to \eqref{NS-weak}, and integrating $\bl u_i^h$ in pseudo-time $\tilde t$ similar to the conventional time formulation of the Navier-Stokes equation. 
Since we are seeking the ``steady'' solution in this case, there is no point in attempting to obtain an accurate estimate at each pseudo-time step. 
Therefore, it is most cost-effective to perform only one Newton-Raphson iteration per pseudo-time step, thereby effectively replacing the Newton-Raphson iterations with iterations in pseudo-time. 

At the outset, this switch may seem pointless, particularly given the fact that the pseudo-time integration will become identical to the Newton-Raphson iterations at the limit of infinite pseudo-time step size $\Delta \tilde t \to \infty$. 
However, the two methods will depart at finite $\Delta \tilde t$ due to the contribution of the pseudo acceleration term in the tangent matrix, which amounts to supplementing $\bl K_{AB}$ in \eqref{K_mat} with
\begin{equation*}
    \left(N_A,\frac{c_1\rho}{\Delta \tilde t}N_B \bl I \right)_\Omega, 
\end{equation*}
where $c_1$ is an $O(1)$ constant that depends on the adopted time integration scheme (in our case, we use the generalized-$\alpha$ method~\cite{jansen2000generalized} with zero $\rho_\infty$, yielding $c_1=1.5$).  
The condition number of $\bl H$ is significantly improved by the inclusion of this mass matrix in  $\bl H$ if $\Delta \tilde t$ is smaller than other time scales in the problem $\Delta \tilde t < \min\{\omega^{-1}, h\|\bl A\|^{-1},  h^2\rho \mu^{-1} \}$.  
This may reduce the overall computation cost by improving the convergence rate of the linear solver. 

Apart from cost saving, our results show the use of pseudo-time stepping also enables solution convergence for some cases where the Newton-Raphson method fails to converge. 
In those cases, the residual may initially decrease with the Newton-Raphson iterations but will fail to drop below a certain threshold as more iterations are performed. 
With pseudo-time stepping, on the other hand, one can derive the residual to an arbitrary small value, thus ensuring the convergence of the solution. 

Even though using a very small $\Delta \tilde t$ reduces the cost of linear solves, it will also increase the number of pseudo-time steps required for convergence. 
At very large $\Delta \tilde t$, on the other hand, the original Newton-Raphson iterations are recovered, making this optimization ineffective. 
The greatest cost saving is achieved between those two extremes at an optimal $\Delta \tilde t$.
For cases studied in Section \ref{sec:result}, that optimal $\Delta \tilde t$ is roughly an order of magnitude larger than the time step size adopted for physical time integration of the Navier-Stokes equations.
Using the optimal $\Delta \tilde t$, this strategy reduces the overall cost of the GLS method by approximately a factor of 2 for the case considered in Section \ref{sec:result}. 

\subsection{Stabilizing simulations in the presence of backflow} \label{sec:bfs}
The GLS method formulated in \eqref{NS-weak} will become unstable in the presence of reversal flow through $\Gh$ where a Neumann boundary condition is imposed. 
This issue, which equally affects the conventional time formulation of the Navier-Stokes equation, has been the subject of a series of techniques to stabilize simulations involving partial or bulk backflow through outlets \cite{Esmaily2011backflow, bertoglio2018benchmark, Bazilevs20093534, kim20093551}. 
Among those techniques, the most robust and least intrusive method is constructed by adding a stabilization term to the weak form to ensure the resulting method is coercive and hence stable even in the presence of backflow. 

To construct this backflow stabilization technique for the GLS formulation in \eqref{NS-weak}, let us reexamine Lemma \ref{lem:stability}. 
For the convective term produced by $\dnorm{\cdot}$ to be non-negative, we needed to make sure $\An \coloneqq \bl A_i n_i$ is positive semi-definite. 
If that condition is violated, we may get $b(\bl w,\bl w)<0$, causing instability in the solution.
Thus, to stabilize the GLS method in the presence of backflow, we need to modify eigenvalues of $\An$ on $\Gh$ to ensure they are always positive. 
To accomplish this, we add 
\begin{equation*}
    \inner{\bl w^h_i, \frac{\rho}{2}\beta  |\An|_- \bl u_i^h }_\Gh,     
\end{equation*}
to the right-hand side of \eqref{NS-weak}. 
Here, $\beta \in [0, 1]$ is a prescribed user-defined coefficient and 
\begin{equation*}
     |\An|_- = \frac{\An - |\An|}{2}. 
\end{equation*}

\paragraph{Remarks:} 
\begin{enumerate}
    \item Computing $|\An|_-$ involves applying the absolute function to a matrix. Numerically, $|\An|$ is computed by solving a $(2N-1)\times(2N-1)$ eigenvalue problem $\An = \bl V \bl \Lambda \bl V^H$, then taking the absolute value of eigenvalues to compute $|\An| = \bl V |\bl \Lambda| \bl V^H$. 
    \item In the absence of backflow, all eigenvalues of $\An$ will be positive and $|\An|_- = \bl 0$, thus leaving \eqref{NS-weak} unchanged. This is an attractive property of this backflow stabilization method, which makes it less intrusive by only becoming active when necessary.    
    \item This approach is equivalent to its conventional time formulation counterpart (readily verifiable for $N=1$), which has been shown in the past to be a superior method for dealing with instabilities caused by backflow through Neumann boundaries. 
    \item Stability is guaranteed ($\dnorm{\cdot}>0$) if $\beta=1$. That is so since $|\An|_- = \An$ in the presence of backflow, thus eliminating the contribution of the convective term in $\dnorm{\cdot}$. Even though a stable solution is guaranteed at $\beta=1$, one may still generate stable solutions for smaller values of $\beta$ (e.g., $\beta=0.5$). In such cases, using a smaller value of $\beta$ is recommended to reduce the effect of the added term on the results.
    \item Similar to the Navier-Stokes, the convection-diffusion equation can also become unstable in the presence of backflow. To stabilize it, one must subtract the following term from $b(\bl w^h,\bl \phi^h)$ in \eqref{3D-weak-def}:
    \begin{equation*}
        \inner{\bl w^h, \frac{1}{2}\beta  |\An|_- \bl \phi^h}_\Gh.
    \end{equation*}
\end{enumerate}

\section{Results} \label{sec:result}
To investigate the performance of the proposed GLS method, a realistic cardiovascular condition is considered below~ \cite{wilson2013vascular, marsden2010new, esmaily2023stabilizedA}.
We use this case for both the flow and tracer transport modeling, through which we establish the importance of the least-squares stabilization term in the formulation. 
At the end of this section, the error and cost of the time-spectral method are studied and compared against an equivalent time formulation, which is described at length in Appendix \ref{sec:app}.

\subsection{The case study}
The test case considered throughout this section is obtained from a patient who has undergone a Fontan operation. 
As shown in Figure~\ref{fig:model}, in this operation, a connection is established between the systemic venous returns, namely superior and inferior vena cava (or SVC and IVC, respectively), and the right and left pulmonary arteries (or RPA and LPS, respectively). 
By creating this connection, the right heart chamber is bypassed to create an in-series blood flow through the systemic and pulmonary circulations. 

\begin{figure}
    \centering
    \includegraphics[width = 0.8\textwidth]{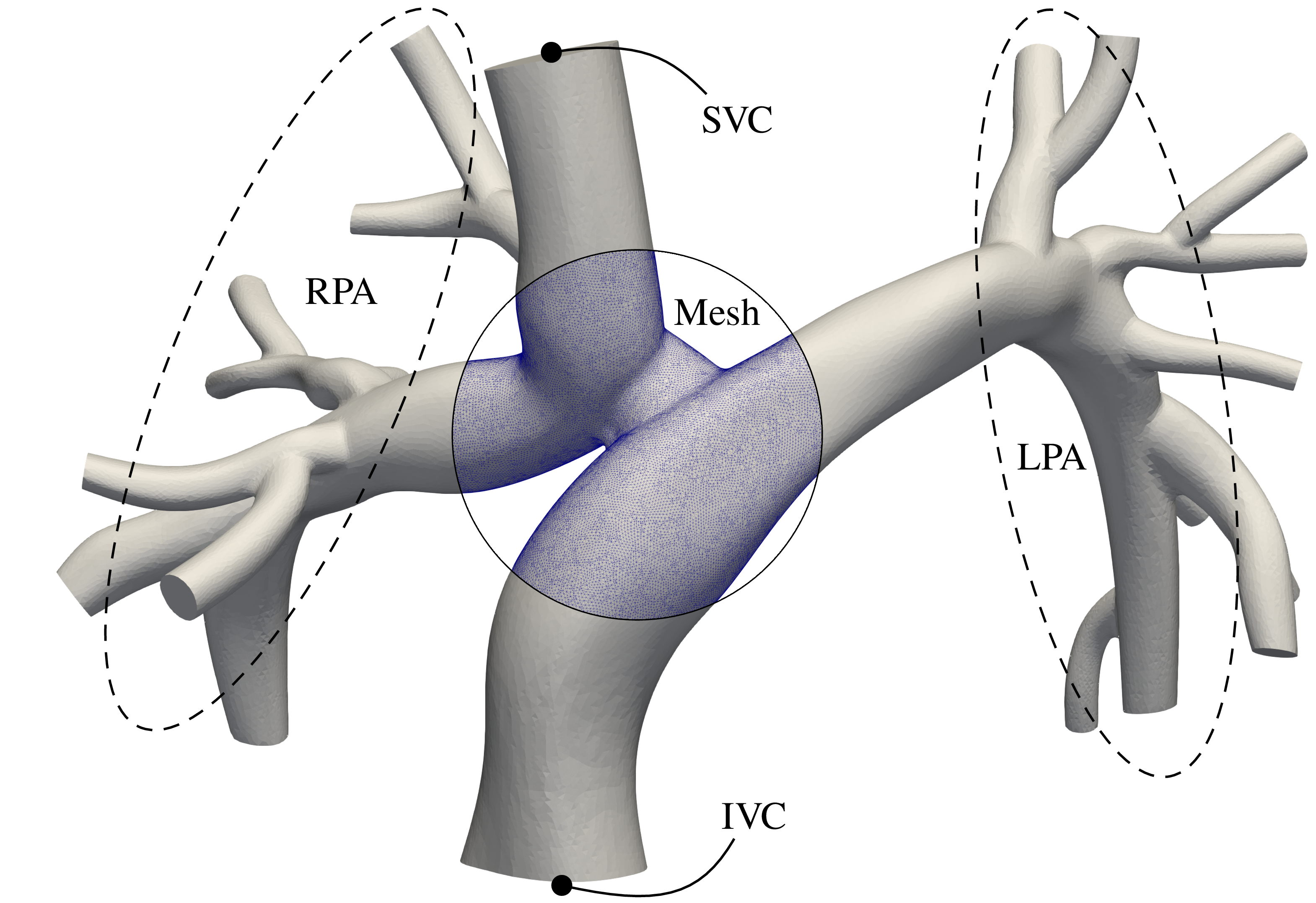}
    \caption{The Fontan geometry considered in this study with boundary faces marked and a sectional view of the mesh. }
    \label{fig:model}
\end{figure}

This life-saving operation is performed after two other open-chest operations on children with a single ventricle as a long-term solution. 
Over time, however, these patients are at risk for renal and liver dysfunctions, protein-losing enteropathy, and failure of the Fontan circulation. 
Part of these complications can be traced back to how hepatic flow from the IVC is distributed between the left and right lungs. 
It has been shown that a hepatic flow that is predominantly directed toward one of the lungs can lead to arteriovenous malformations~\cite{duncan2003pulmonary, yang2013optimization}. 

To ensure a balanced IVC flow distribution to the LPA and RPA, some hospitals have adopted simulations as a surgical planning tool in their practice~\cite{javadi2022predicting}. 
As is the case with any simulation-based design problem, this practice requires simulating flow in many surgical configurations to identify an optimal design. 
This exercise can be cost-prohibitive, thereby making it an excellent candidate for the proposed technology that targets the acceleration of these types of simulations. 

The anatomical model and the flow boundary conditions employed in this study are all obtained from an external repository \cite{wilson2013vascular, marsden2010new}.  
The model is discretized using $1,696,777$ tetrahedral elements (Figure \ref{fig:model}). 
For both the SVC and IVC, an identical unsteady inflow boundary condition is imposed with a parabolic profile (Figure \ref{fig:inflow}). 
Given the smooth and periodic behavior of inflow $Q_{\rm in}(t)$, it is exactly represented using $N=30$ modes (Figure \ref{fig:inflow} inset) or well-approximated using 7 to 10 modes. 
The flow exits the domain through the RPA and LPA via a total of 9 and 11 branches, respectively. 
A zero Neumann boundary condition is imposed on all these outflow branches.
A no-slip boundary condition is also imposed on the vascular walls.

\begin{figure}
    \centering
    \includegraphics[width=0.6\textwidth]{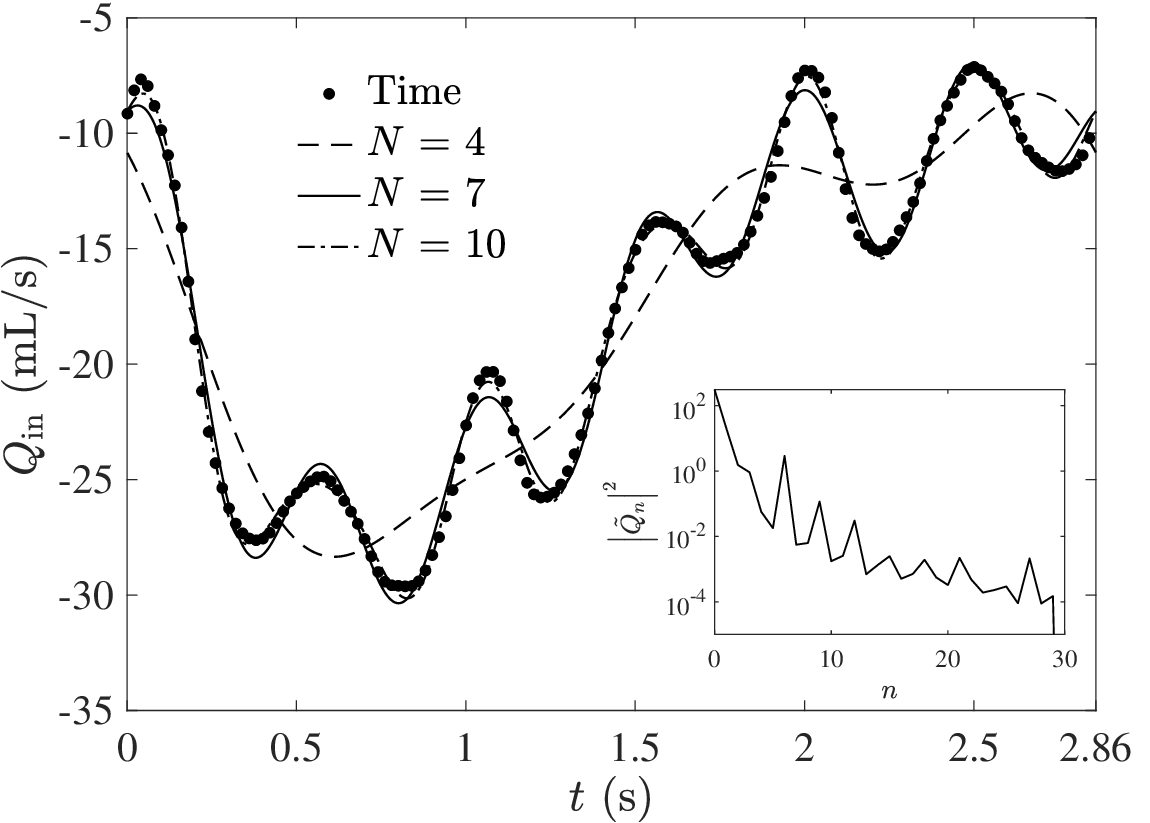}
    \caption{Inlet flow rate for both the SVC and the IVC, $Q_{\mathrm{in}}(t)$. The reference profile used for the time simulation (filled circles) and its best approximation with $N=$ 4, 7, and 10 modes (dashed, solid, and dash-dot lines, respectively) used for time-spectral simulations is shown. The inset shows the flow energy spectrum ($\left| \Tilde{Q}_n\right|^2 = \Tilde{Q}_n \Tilde{Q}_n^*$).}
    \label{fig:inflow}
\end{figure}

In addition to the Navie-Stokes equations, we also solve the convection-diffusion equation to model the IVC flow distribution to the LPA and RPA. 
The boundary conditions for that additional equation are zero Dirichlet at the SVC inlet, unit steady Dirichlet ($\phi_0 = 1$ and $\phi_n=0$ for $n\ne 0$) at the IVC, and zero Neumann (no flux) on the vascular wall. 
To examine the solution stability under a rapid jump, zero Dirichlet is imposed on all outlets.  

The variations in the inflow waveform shown in Figure \ref{fig:inflow} is produced by the respiratory cycle, which is associated with the longer time scale of 2.86 s, and the cardiac cycle, which is 6 times faster with a period of 0.48 s.  
The base frequency in our simulation $\omega$ is thus determined by the respiratory cycle, which is 2.86 s.
To capture the cardiac cycle, $N$ must be at least 7, which explains the peak in the energy spectrum shown in the inset of Figure \ref{fig:inflow} at $N=7$.
By considering a case where dynamics are influenced by two time scales associated with the respiratory and cardiac cycles, we are testing the proposed method in an extreme scenario that requires a larger $N$ than if, e.g., only dynamics were determined by the cardiac cycle only.

The blood density and viscosity are $\rho=1.06$ g/cm$^3$ and $\mu = 0.04$ g/(s$\cdot$cm), respectively. 
The mass diffusivity of the tracer for the convection-diffusion equation is assumed to be the same as that of the kinematic viscosity at $\kappa=0.04/1.06=0.0377$ cm$^2$/s.
Based on these parameters and the IVC hydraulic diameter ($D_{\rm IVC}$), the mean and peak flow Reynolds number ($Re = \rho U D_{\rm IVC}/\mu$) are 468 and 787, respectively. 
The flow Womersley number ($W = D_{\rm IVC}\sqrt{\rho n \omega /\mu}$), on the other hand, ranges from 9.7 at $n=1$ to 29 at $n=9$ (the highest simulated mode below).  
Thus, both the Eulerian and convective acceleration terms play an important role in the solution. 

\subsection{Simulation parameters}
Since the proposed time-spectral method will be compared against standard time formulation, it is important to keep the numerical parameters entering the two methods as similar as possible. 
For this purpose, both spectral and time formulations are implemented in our in-house finite element solver, sharing the same routines for the linear system assembly \cite{mupfes}. 
The tangent matrix optimization described under Section \ref{sec:Hstructure} is applied to both formulations. 
The same iterative linear solver is also used for both, which is based on the GMRES algorithm with restarts \cite{saad1986gmres, Esmaily2015BIPN, shakib1989multi}. 
The size of Krylov subspace and tolerance on the linear solve solution are set to 100 and 0.05, respectively, for both formulations. 
For the time simulations, a maximum is set on the number of GMRES iterations (defined as the total number of matrix-vector products) at 3,000 to reduce its overall cost. 
For the spectral simulations, the optimal cost is achieved when this value is set at a very large number (i.e., 10,000) so that that limit is never reached. 
For the time formulation, the maximum number of Newton-Raphson iterations per time step is set at 10, which was only reached for the first two time steps. 
That number is set to one for the spectral simulations when performing pseudo-time stepping and to a very large number otherwise (i.e., when $\Delta \tilde t = \infty$). 
All computations are performed on the same hardware that includes 16 nodes (Dell PowerEdge R815) that are connected using QDR InfiniBand interconnect (Mellanox Technologies MT26428).
Each compute node has four AMD Opteron(TM) 6380 processors (a total of 64 physical 2.4GHz cores), and 128 GB of PC3-12800 ECC Registered memory.

For the time simulation, the cardiac cycle is discretized to $2,860$ time steps, resulting in a time step size of $\Delta t = 10^{-3}$. 
To select an optimal pseudo-time step size $\Delta \tilde t$ for the spectral simulations, we performed multiple calculations at $N=7$, monitoring the rate of drop in the residual $\|\bl r\|$ at different values of $\Delta \tilde t$. 
As shown in Figure~\ref{fig:sudotime}, this exercise shows that $\Delta \tilde t = 2\times 10^{-2}$ minimizes the cost of the solution. 
A very small or very large $\Delta \tilde t$ drives up the cost as it leads to too many pseudo-time steps or linear solver iterations, respectively. 
This trend continues up to a degree at which at $\Delta \tilde t = \infty$ the residual fails to drop below $10^{-2}$, indicating the importance of pseudo-time stepping for convergence of this method. 

\begin{figure}
    \centering
    \includegraphics[width=0.6\textwidth]{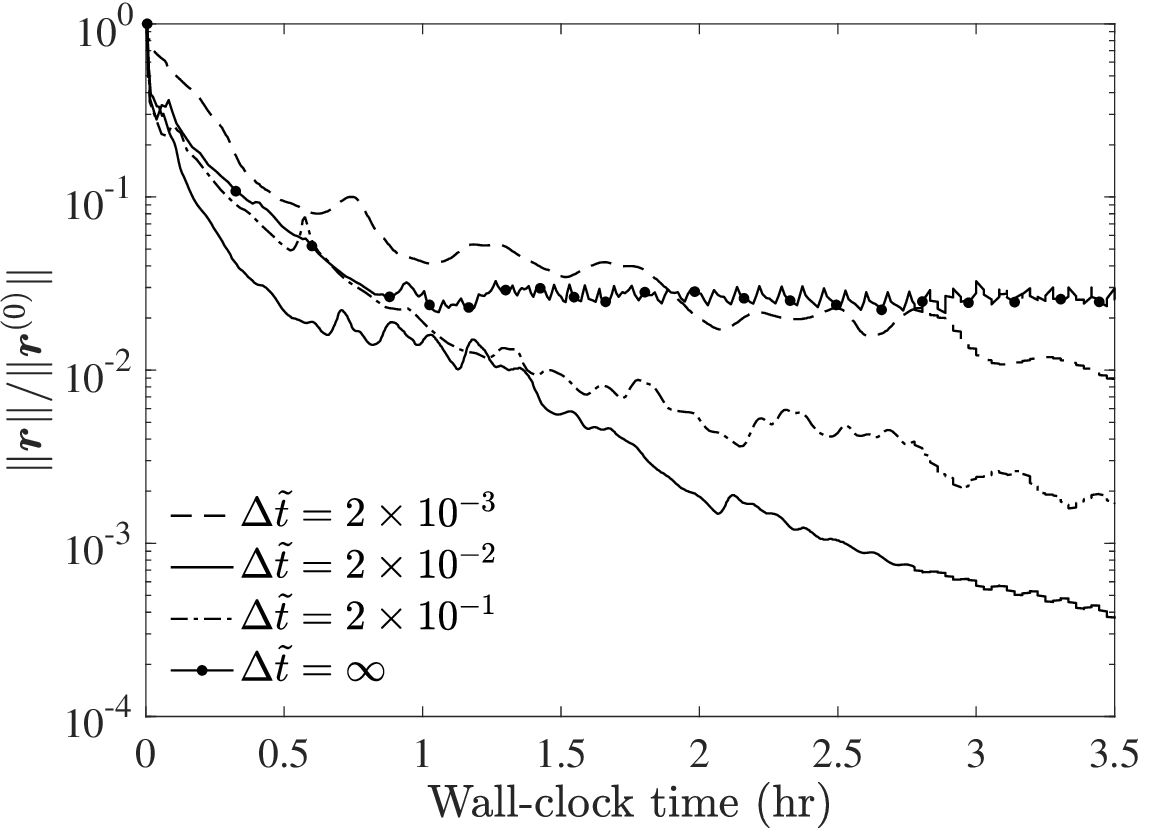}
    \caption{The norm of the residual as a function of simulation wall-clock time in performing time-spectral simulation with different pseudo-time step sizes.}
    \label{fig:sudotime}
\end{figure}

The last numerical parameter is tolerance on the residual $\epsilon_{\rm NR}$. 
This parameter determines the termination point of pseudo-time stepping (or the Newton-Raphson iterations when $\Delta \tilde t= \infty$) for the spectral formulation. 
To select this parameter, we rely on the error in predicted flow at $N=7$ through the LPA. 
This error is calculated as $\|Q_{\rm LPA,ref}-Q_{\rm LPA}\|/\|Q_{\rm LPA,ref}\|$, where $Q_{\rm LPA}$ is the total flow through the LPA branches computed at a given pseudo-time step and $Q_{\rm LPA,ref}$ is its reference quantity that is taken to be $Q_{\rm LPA}$ at the last pseudo-time step. 
As shown in Figure~\ref{fig:tolselect}, the error is proportional to the residual such that at normalized residual of $10^{-3}$, the error associated with the early termination of the Newton-Raphson iterations is approximately 0.1\%. 
Given that this error is below other sources of error (spatial discretization and spectral truncation), we set $\epsilon_{\rm NR}=10^{-3}$ for all simulations. 
For an apple-to-apple comparison, the same tolerance is also used in the time formulation when terminating the Newton-Raphson iterations within each time step. 

\begin{figure}
    \centering
    \includegraphics[width=0.6\textwidth]{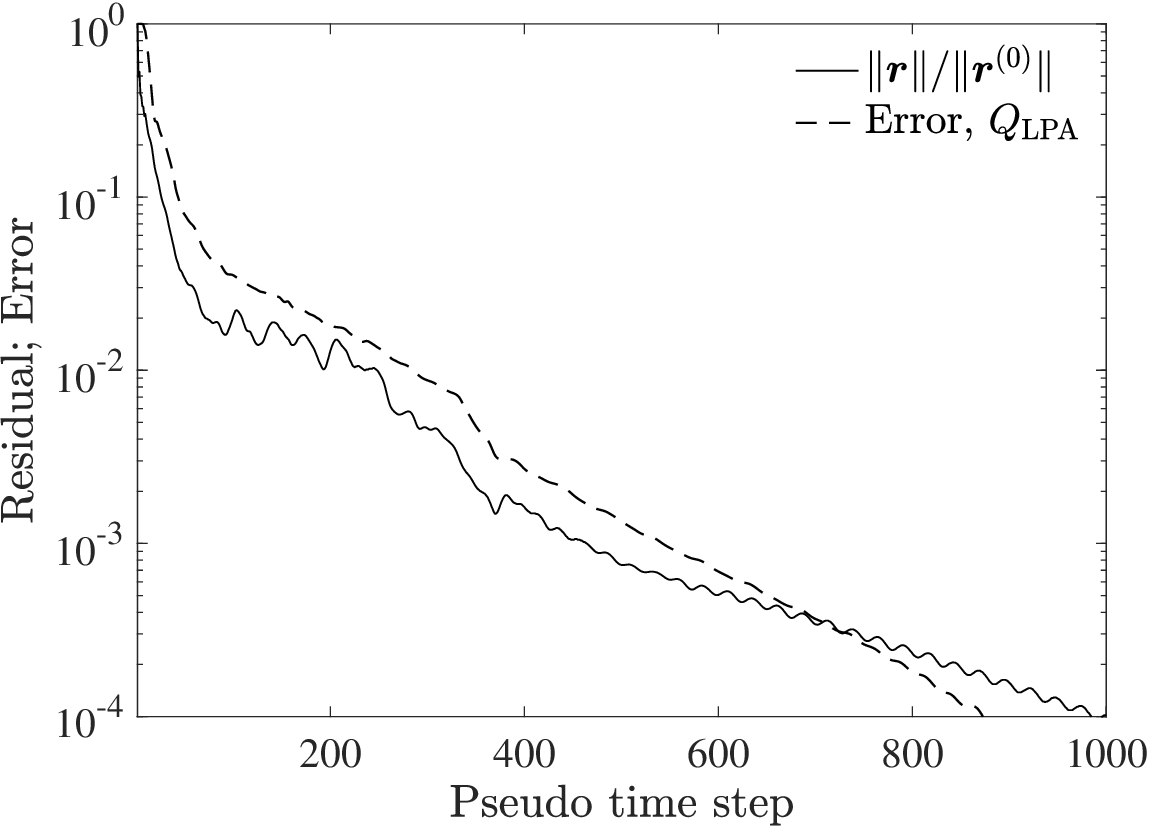}
    \caption{Normalized residual (solid line) and $L_2$-norm error in the predicted flow rate through the LPA (dashed line) versus pseudo-time step.}
    \label{fig:tolselect}
\end{figure}

\subsection{Solution behavior} \label{sec:solB}
For the simulations performed in the time-spectral domain, each computed mode represents a solution ``building block''. 
These modes are independent in the case of the convection of a neutral tracer in a steady flow. 
For the case under consideration, however, these modes interact owing to the unsteady flow that generates a nonzero solution at all modes even if the boundary conditions at all modes but one are homogeneous. 
Figure~\ref{fig:adfcont}, which shows the concentration of the hepatic factor at different modes $\phi_n$, demonstrates this point. 
All boundary conditions are zero for this case except for the steady mode on the IVC face where $\phi_0 = 1$. 
Despite homogeneous boundary conditions, all unsteady modes are nonzero, highlighting how mode coupling via convective term energizes solutions at all modes. 
Clearly, a cost-saving strategy that relies on neglecting the mode coupling (e.g., by simulating the steady mode only) will incur large errors, thus emphasizing the need for the coupled solution strategy described in this study. 

\begin{figure}
    \centering
    \includegraphics[width=0.8\textwidth]{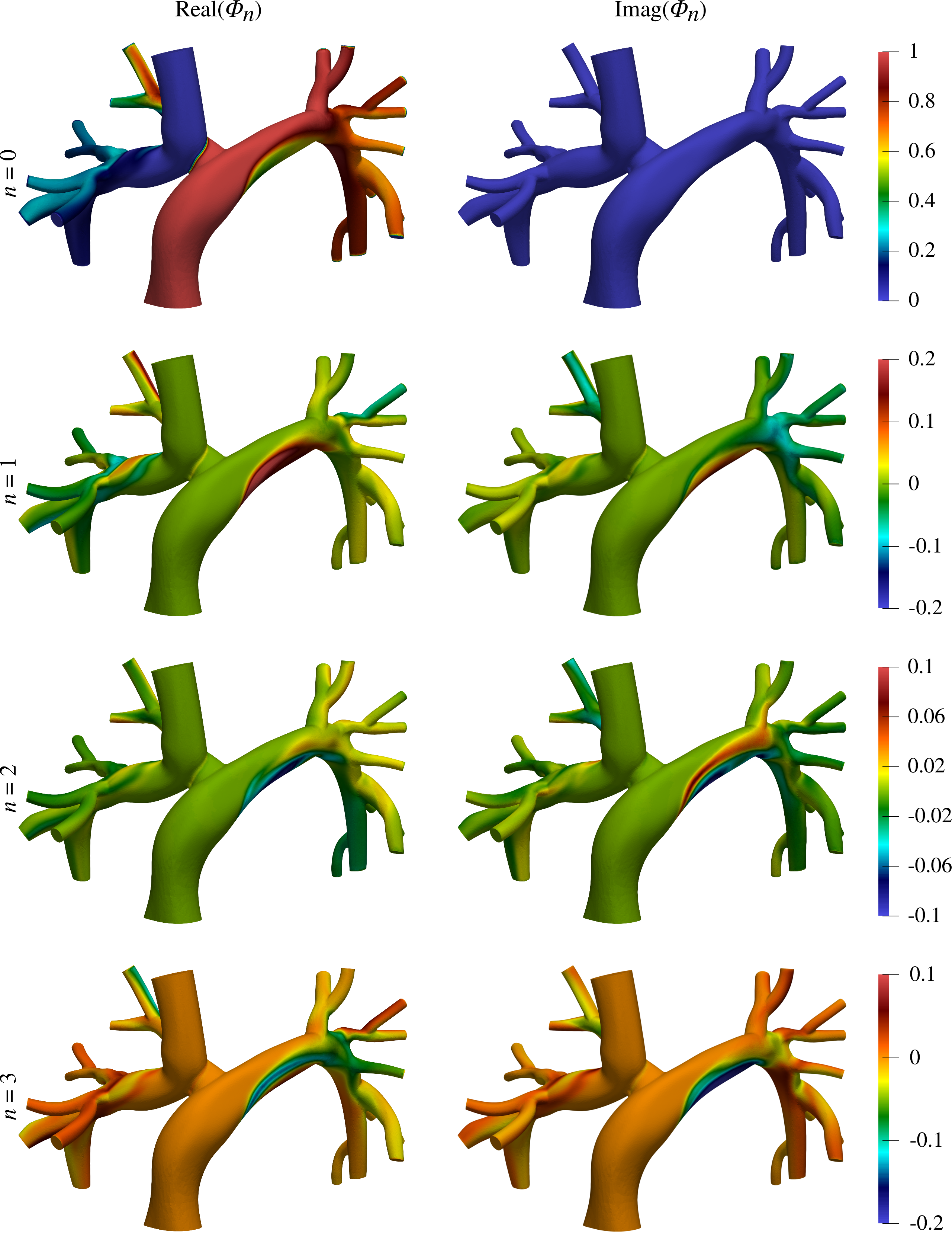}
    \caption{The solution obtained from the GLS method for the convection-diffusion problem subjected to a steady boundary condition at the IVC. The steady, the first, the second, and the third mode are shown in order from top to bottom rows. The left and right columns show the solution that is in-phase and out-of-phase, respectively, with the prescribed inflow boundary condition.  }
    \label{fig:adfcont}
\end{figure}

One can make two more observations based on the results shown in Figure \ref{fig:adfcont}. 
Firstly, while the in-phase Real($\phi_n$) and out-of-phase Imag($\phi_n$) solution are both nonzero for $n\ge1$, only the in-phase component is nonzero when $n=0$. 
That is expected from the symmetry that is built into the GLS method and permits one to use it as a verification test when implementing this method.
Physically, the out-of-phase mode for the steady solution $\phi_0$ has no meaning and must be zero. 
Secondly, the solution at various modes could be significantly different. 
This observation has implications on whether it would be possible to save on the cost by interpolating, rather than computing, a given mode from the neighboring modes. 
That said, however, there have been recent advances in predicting the truncated modes using the computed modes in a different context~\cite{rigas2022data}. 

The solution investigated above corresponded to the flow and tracer simulations performed with $N=4$ independent modes. 
The behavior of the solution to the Navier-Stokes equations as more modes are incorporated in the simulations is qualitatively studied in Figure~\ref{fig:contp}.
In presenting these results, we compare them against the solution obtained from the conventional time formulation to show the gradual convergence of the two as $N$ is increased. 
We selected $t = 1.28$ s within the respiratory cycle to compare the two formulations since the truncation error in the boundary condition is relatively large at this time point (Figure \ref{fig:inflow}). 
 
\begin{figure}
    \centering\includegraphics[width=0.8\textwidth]{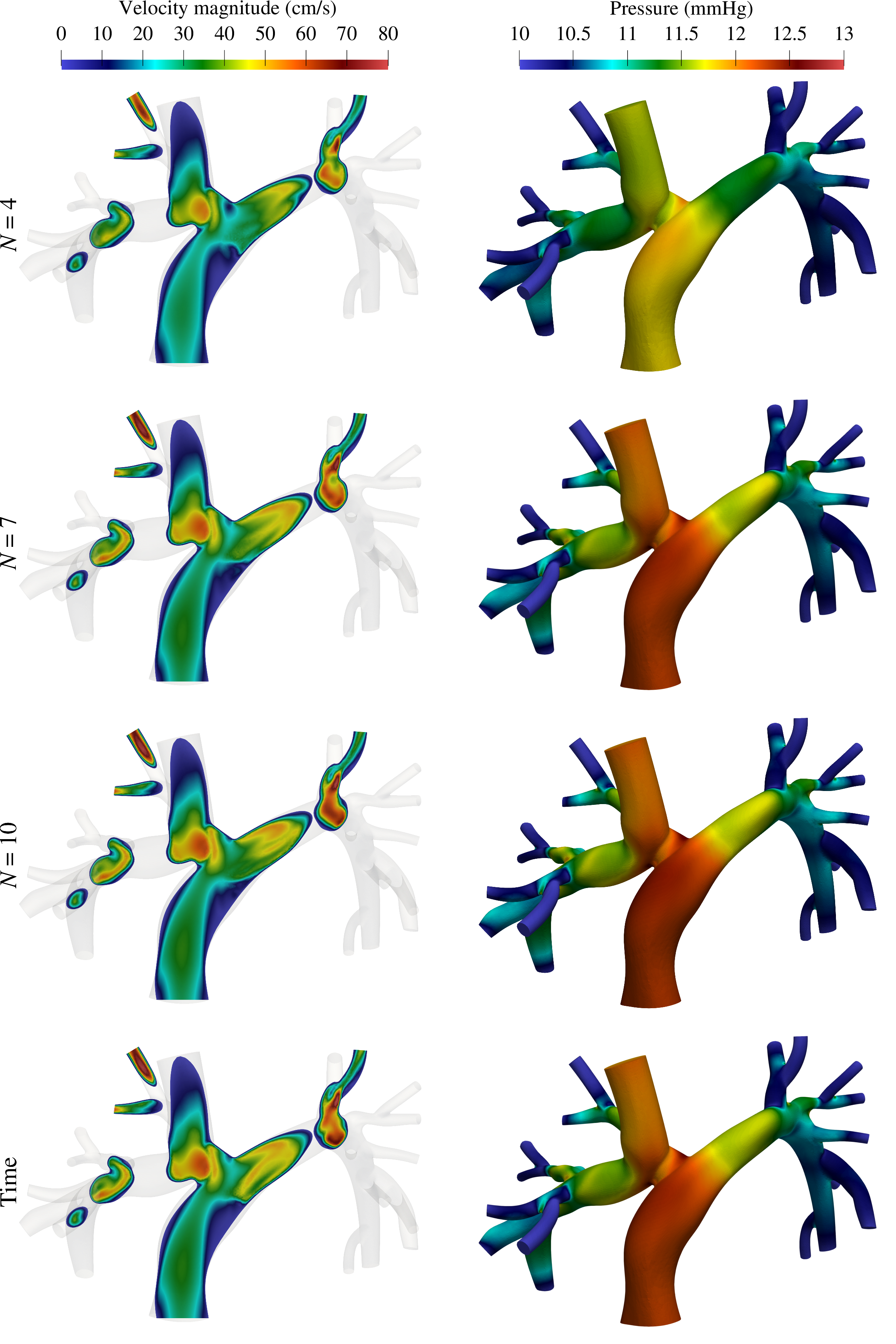}
    \caption{Velocity magnitude (left) and pressure (right) contour obtained from three time-spectral simulations performed using $N=4$ (top row), $N=7$ (second row), and $N=10$ (third row) modes. The results obtained from the conventional time formulation are also shown for comparison (bottom row). These results correspond to $t=1.28$ s. }
    \label{fig:contp}
\end{figure}

These two plots show that as we increase $N$, the spectral formulation's result converges toward that of the temporal formulation. 
While a close agreement between the two is observed for $N\ge 7$, slight differences remain even at larger $N$. 
Based on these results, the difference between the two formulations at small $N$ is dominated by the large truncation error in the boundary condition (Figure \ref{fig:inflow}).
As $N$ increases and truncation error becomes smaller, other sources of error, such as spatial discretization, become the main source of difference between the two formulations. 

\subsection{Importance of stabilization}
All spectral results presented so far were obtained from the Galerkin/least-squares (GLS) stabilization method.
To show the importance of the added penalty term to the formulation, we compare that stabilized method against the baseline Petrov-Galerkin's method (GAL) in this section. 
For this purpose, we consider the convection-diffusion equation as the Navier-Stokes can not be solved using the GAL method on the existing mesh with equal order elements for velocity and pressure. 
Thus, we use the same velocity field (which is computed using the GLS method) to model hepatic factor transport via the GAL method and compare it against the GLS solution that was presented earlier. 

The importance of the stabilization scheme becomes apparent when there is a sharp change in the solution. 
This point is demonstrated in Figure~\ref{fig:galvgls}, which shows the failure of the GAL method to produce a smooth solution near the outlet, where $\bl \phi$ should drop quickly from its upstream value to a zero value imposed at the outlet.
When the solution is converted to the time domain, it must remain bounded by $0 \le \hat \phi \le 1$. 
Nevertheless, a negative value is obtained for $\hat \phi$ in this region. 
Such a non-physical quantity may cause a host of issues in other contexts such as convection-diffusion-reaction-type modeling of thrombus formulation, where $\hat \phi <0$ may create rather than consume a species. 
Ensuring the boundedness of the solution in such scenarios is of paramount importance, thus underscoring the need for the proposed stabilization method.  

\begin{figure}
    \centering
    \includegraphics[width=0.6\textwidth]{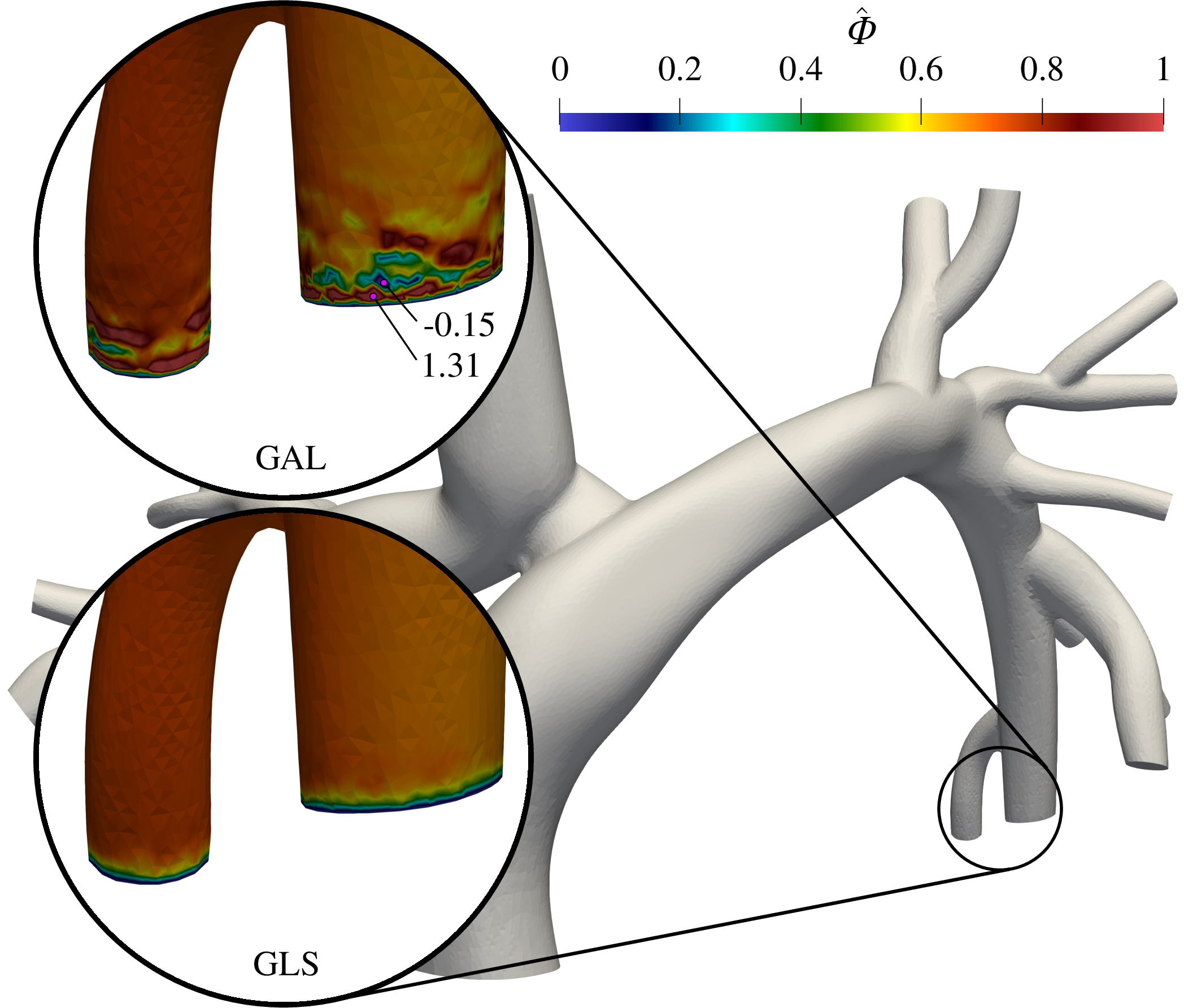}
    \caption{The transport of the hepatic factor modeled using Galerkin's method (top) and the Galerkin/least-squares method (bottom). Results correspond to $t = 1.28$ s. Note the overshoot and undershoot in $\hat \phi$ produced by Galerkin's method. }
    \label{fig:galvgls}
\end{figure}

Setting aside considerations regarding the stability of the solution, the least-squares term also affects the solution cost for the problem studied here.
As shown in Figure~\ref{fig:res_gal_v_lsq}, the number of linear solver iterations, which determines the bulk of the cost, is significantly lowered by the GLS method relative to the GAL method. 
This gap in performance can be attributed to the change in the condition number of the tangent matrix. 
Note that the smallest eigenvalue of the tangent matrix is proportional to $\inf_{\bl w^h}\dnorm{\bl w^h}^2/\norm{\bl w^h}^2$. 
Since the least-square term increases $\dnorm{\bl w^h}^2$ by $\norm{\res(\bl w^h)}^2_{\bl \tau,\Ot}$, one would expect the GLS method to produce a tangent matrix that is less stiff, and hence, an overall method that is more economical. 
That said, one should be careful not to generalize this argument to other situations where an improvement in stability may or may not coincide with an improvement in cost.

\begin{figure}
    \centering
    \includegraphics[width=0.6\textwidth]{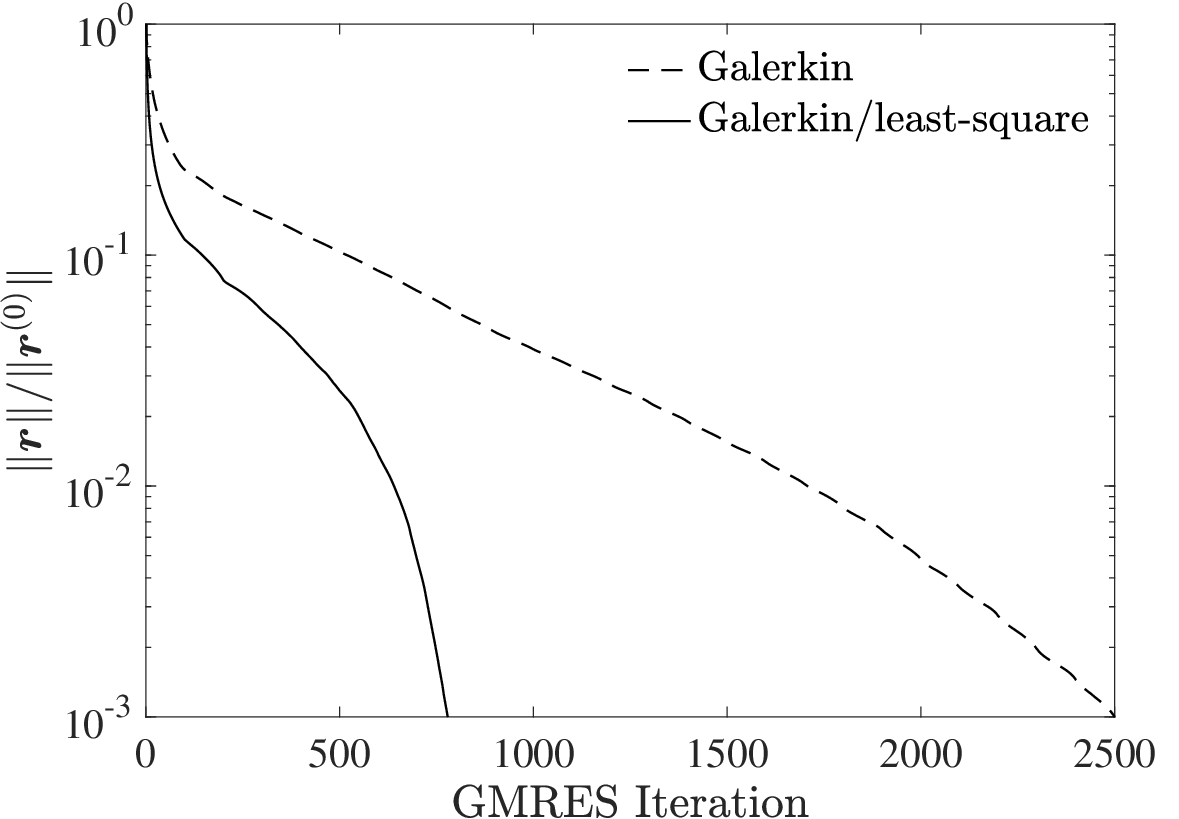}
    \caption{The convergence of the residual versus linear solver (GMRES) iterations for Galerkin's (dashed line) and the Galerkin/least-squares (solid line) methods.}
    \label{fig:res_gal_v_lsq}
\end{figure}

\subsection{Error behavior}
In section \ref{sec:solB}, the behavior of the GLS solution and its variation with the number of modes $N$ were studied qualitatively.
In this section, we investigate the GLS error, or more specifically the truncation error associated with selecting a small $N$, on a more quantitative basis. 
The motivation for doing so is that there is a big cost incentive in keeping $N$ as small as possible. 
Thus, having an idea of the expected error, which is produced by truncating the time-spectral discretization series at small $N$, will allow one to select the smallest possible $N$ that maintains the error below an acceptable tolerance. 

To investigate this point further, the total flow through the LPA branches $Q_{\rm LPA}$ and pressure at the IVC $P_{\rm IVC}$ are considered as the predicted parameters of interest. 
The time variation of these two parameters over the cardiac cycle is shown in Figure~\ref{fig:pqcompare}. 
The overall behavior of the predicted $Q_{\rm LPA}$ using the GLS method at different values of $N$ and how they compare against that of the time formulation closely resembles that of the imposed inflow boundary condition $Q_{\rm in}$ in Figure \ref{fig:inflow}. 

To investigate this close relationship more quantitatively, we have computed the error $\|Q_{\rm LPA, ref} - Q_{\rm LPA}\|/\|Q_{\rm LPA, ref}\|$ using the time formulation prediction $Q_{\rm LPA, ref}$ as the reference. 
The result of this calculation is shown in Figure~\ref{fig:solerr}. 
This figure shows that indeed the error in the predicted $Q_{\rm LPA}$ closely tracks the truncation error in the imposed boundary condition $Q_{\rm in}$.
This observation has a major practical implication. 
It suggests that one can simply rely on the boundary condition truncation error, a parameter that can be easily calculated before running the costly simulation, to assess the expected error in a parameter of interest. 

\begin{figure}
\centering
\begin{subfigure}[b]{0.49\textwidth}
    \centering
    \includegraphics[width=\textwidth]{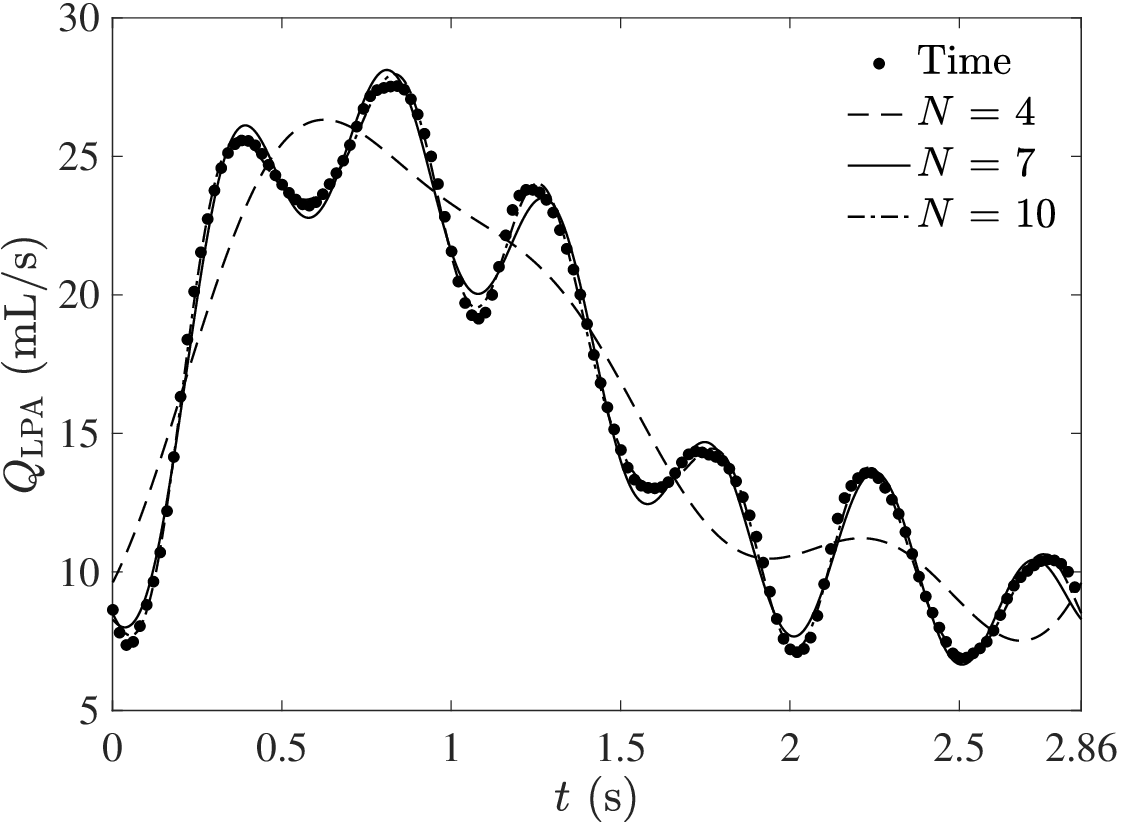}
    \caption{}
    \label{fig:qlpa}
\end{subfigure}
\hfill
\begin{subfigure}[b]{0.49\textwidth}
    \centering
    \includegraphics[width=\textwidth]{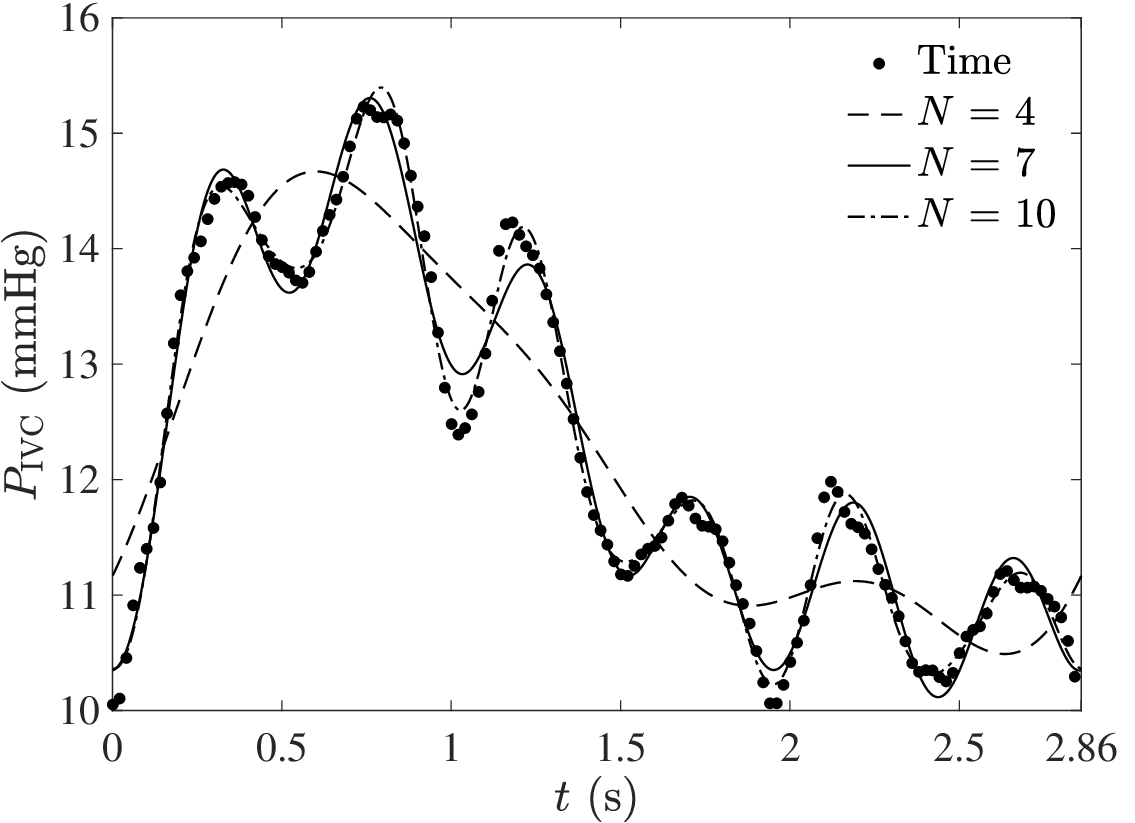}
    \caption{}
    \label{fig:pivc}
\end{subfigure}
\caption{Flow through the LPA branches $Q_\mathrm{LPA}$ (left; a) and pressure at the IVC (right; b) as a function of time $t$ predicted using the conventional time formulation (filled circles) and the proposed GLS method using $N=4$ (dashed line), $N=7$ (solid line), and $N=10$ (dash-dot line). }
\label{fig:pqcompare}
\end{figure}

The behavior of the predicted IVC pressure in Figure \ref{fig:pivc} and the error in its prediction in Figure \ref{fig:solerr} tells a similar story as that of the $Q_{\rm LPA}$. 
Although the relative error in $P_{\rm IVC}$ is shifted relative to that of the boundary condition $Q_{\rm in}$, they both fall at more or less the same rate.
The lower relative error in $P_{\rm IVC}$ at small $N$ is because the pressure at the outlet is 10 mmHg, thus changing the steady pressure of both formulations by the same value, which in turn reduces the relative difference between the two. 

Assuming that a relative error of 2 to 3\% is acceptable, these results suggest that the time-spectral simulations of the case under consideration can be performed with as few as $N=7$ independent modes. 
In the discrete setting, that produces far fewer unknowns than the time formulation, which requires an order of 10$^4$ time steps. 
In the next section, we investigate how this reduction in the dimension of the problem translates to a lower simulation cost. 

\begin{figure}
    \centering
    \includegraphics[width=0.6\textwidth]{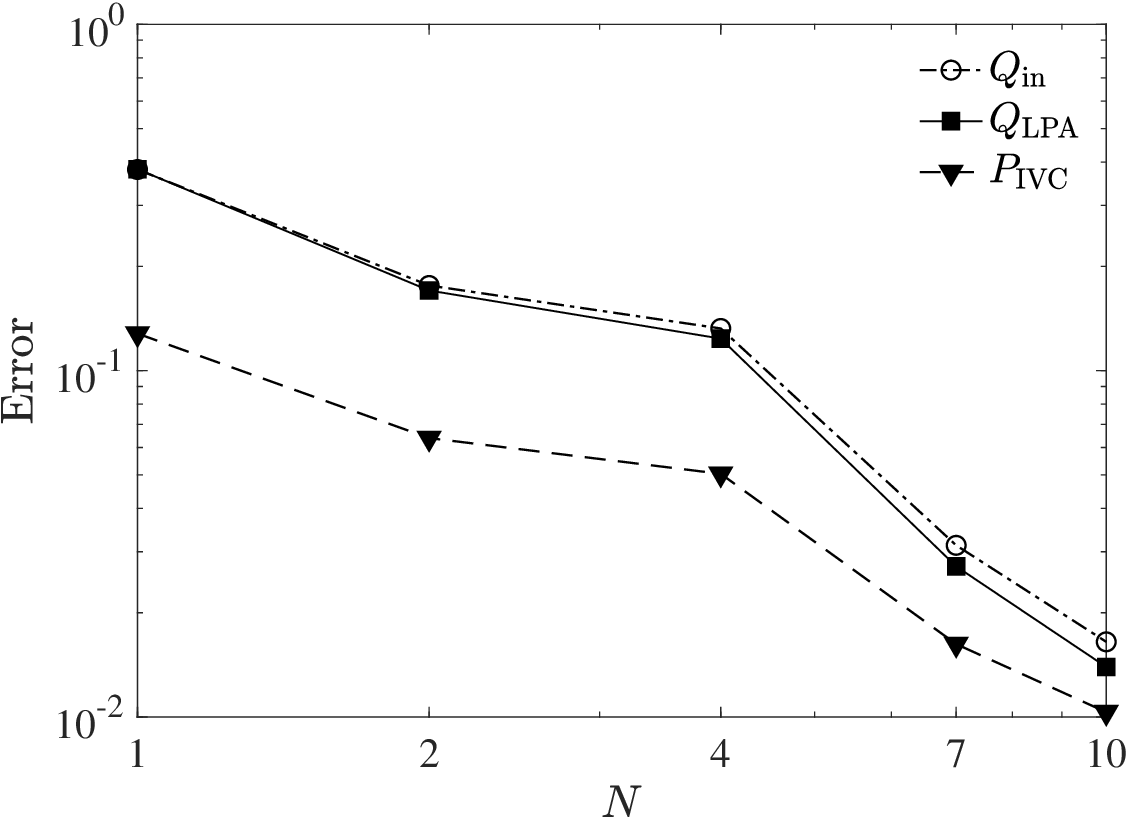}
    \caption{The relative $L_2$-norm error in the predicted IVC pressure $P_{\rm IVC}$ (dashed line with triangles) and LPA flow rate $Q_{\rm LPA}$ (solid line with squares) using the GLS method as a function of the number of simulated modes $N$. The boundary condition truncation error $Q_{\rm in}$ (dash-dot line with circles) is also shown for comparison. } 
    \label{fig:solerr}
\end{figure}

\subsection{Cost analysis}
Earlier in Section \ref{sec:opt}, we introduced several optimization strategies to lower the cost of the GLS method. 
One of those strategies was to take advantage of the structure of $\bl H$ in \eqref{tangent} by performing matrix-vector operations on individual blocks. 
As shown in Figure~\ref{fig:densensf}, this optimization has a twofold effect.
Firstly, it reduces the overall cost of simulation by approximately a factor of two. 
Secondly, it significantly reduces the memory requirement of the GLS method. 
The latter is apparent since the memory requirement for performing these computations at $N=10$ will exceed our hardware capacity if this optimization is not employed. 
With this optimization, however, the memory footprint drops by approximately a factor of 8, theoretically permitting these computations at $N>50$ (verified up to $N=30$). 

\begin{figure}
    \centering
    \includegraphics[width=0.6\textwidth]{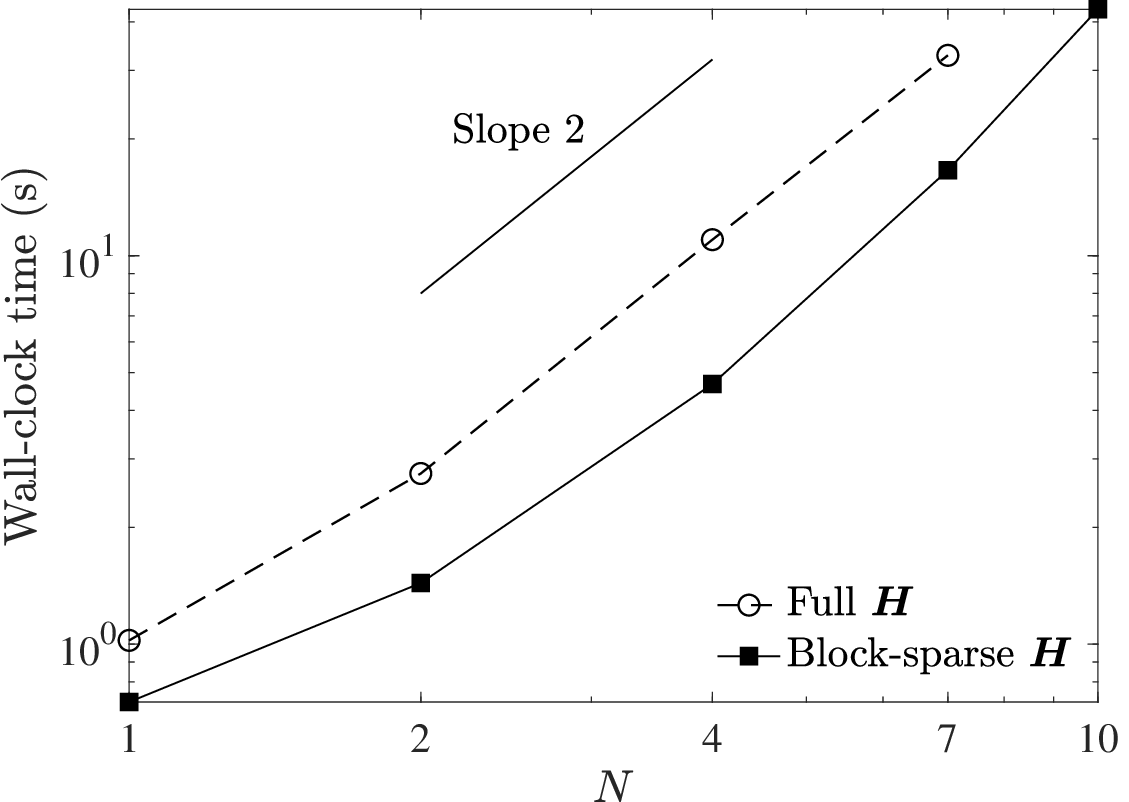}
    \caption{Wall-clock time for performing one pseudo-time step when the tangent matrix $\bl H$ is stored in full (dashed line with circles) or block-by-block (solid line with squares). Full $\bl H$ at $N=10$ is not shown as its memory requirement exceeds that of the hardware on which these computations are performed.}
    \label{fig:densensf}
\end{figure}

Figure~\ref{fig:speedup} shows perhaps the most important result of this test case study.
It reports the cost of the proposed time-spectral method and compares it against the conventional time method. 
First and foremost, a quadratic change in cost is observed as $N$ is changed. 
That strong dependence is an inherent property of the present time-spectral method that has mode-mode coupling incorporated in the tangent matrix.  

Despite such a rapid growth of cost with $N$, the spectral method still presents a significant cost advantage over the time method. 
The performance gap will surely depend on $N$, which must be selected according to the accuracy requirements. 
Nevertheless, at $N=7$, where reasonably accurate predictions are obtained, the present time-spectral method is about an order of magnitude faster than the traditional FEM method for fluid. 
For an application where $N=2$ or 3 is adequate (e.g., when dynamics are determined by the cardiac cycle only), the present method could reduce the simulation turnover time by two orders of magnitude. 

In terms of absolute cost, the simulation performed with the conventional time formulation takes 39 hours to complete on 288 cores. 
On the same number of cores, the time-spectral method at $N=7$ will produce a solution in 4.3 hours (Figure \ref{fig:speedup}). 
That solution turnover time falls below an hour at lower $N$ so that at $N=2$, the simulation is completed in 22 minutes. 

\begin{figure}
    \centering
    \includegraphics[width=0.6\textwidth]{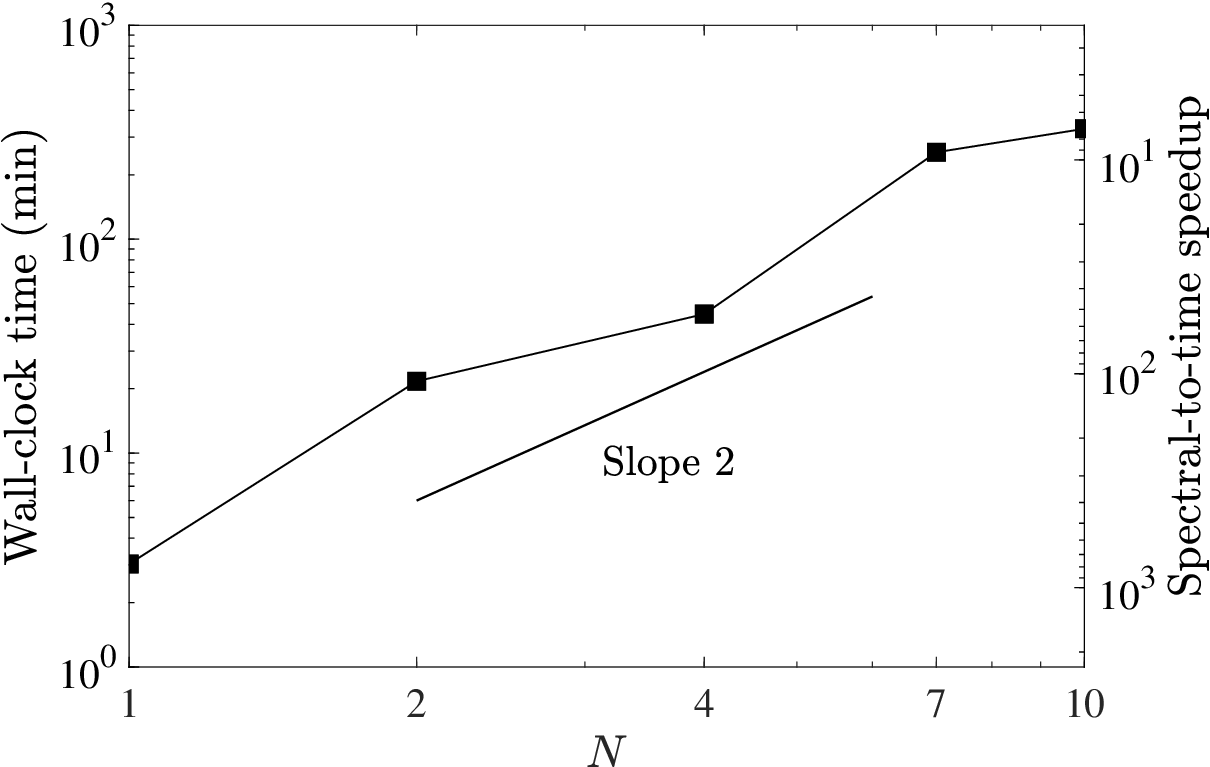}
    \caption{The simulation turnover time for the present time-spectral method as a function of $N$. The speedup relative to the conventional time formulation, which takes 39.2 hours to complete, is also measured on the right axis. These results correspond to simulations that are performed using 288 cores.}
    \label{fig:speedup}
\end{figure}

\subsection{Scalability analysis}
The cost advantage of the proposed method over the conventional method, which was shown in Figure \ref{fig:speedup}, was obtained at a fixed number of compute cores of 288. 
As this number increases, the performance gap between the two methods will widen (Figure \ref{fig:npscale}). 
That highlights another advantage of the proposed time-spectral method over its temporal counterpart: improved parallel scalability. 

To demonstrate this point, we considered $N=7$ and measured the cost of a single pseudo-time step while systematically varying the number of compute nodes from 1 to 16. 
The same exercise is repeated for the time formulation, where we measured the cost of a single time step instead. 
From those results, the parallel speedup is computed for each method and plotted in (Figure~\ref{fig:npscale}). 
This strong scaling analysis is performed on the original mesh, which contains roughly 1.7 million tetrahedral elements, and a coarser mesh with 280 thousand elements, which is adopted here to demonstrate the relative scalability of the two methods under more extreme circumstances. 

\begin{figure}
    \centering
    \includegraphics[width=0.6\textwidth]{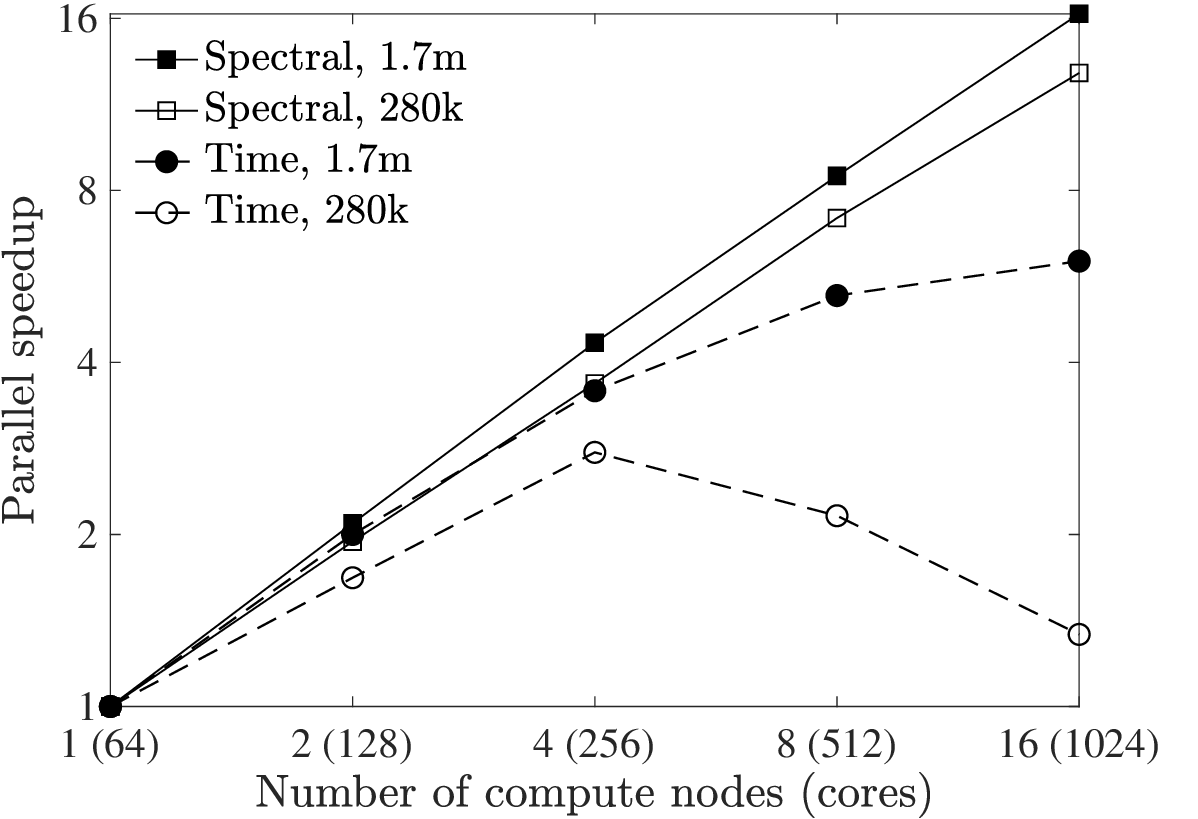}
    \caption{Strong scaling of the time-spectral (solid line with squares) and conventional (dashed-line with circles) formulations using the original mesh (filled symbols) and a smaller mesh with 282k elements (hollow symbols).}
    \label{fig:npscale}
\end{figure}

The time-spectral method exhibits significantly better parallel scalability. 
On the finer mesh, the spectral method demonstrates perfect linear scaling, whereas the time method hardly scales beyond 4 to 8 nodes. 
The improvement in parallel efficiency is even more dramatic on the coarser grid with fewer elements per partition. 
While the spectral method still scales well up to 16 nodes (13-fold speedup), the time method fails to scale beyond 3-fold, which is achieved at 4 nodes. 

We must emphasize the fact that this improved scalability comes for ``free'' as we put no effort into implementing anything special for the spectral method. 
The spectral method is implemented on top of the same set of routines (including the linear solver) as those of the time formulation. 
They both use the same mesh partitioning approach for MPI distributed memory parallelization \cite{Esmaily2015DS}. 

This improved scalability is a direct consequence of the fact that the computations at each mesh node are much more intensive in the spectral method than in the time method. 
Taking the matrix-vector product as an example, the time method performs 10 floating point product operations per connection in the mesh. 
The spectral method, in contrast, performs $4N(3+4N)$ of those operations.
That amounts to an 87-fold increase in the number of local-to-processor operations at $N=7$. 
In terms of processor-to-processor communication latency, the two methods are identical as the number of MPI messages only depends on the number of shared boundaries in the mesh partitions. 
In terms of overhead associated with the limited communication bandwidth, MPI messages for the spectral method are $2N$ longer than those of the time method. 
Thus, in the worst-case scenario, the communication overhead of the spectral will be 14 times that of the time method at $N=7$. 
Therefore, having $4N(3+4N)$ values to compute while $2N$ values to communicate and the same number of messages to send reduces communication overhead by at least $6+8N$ as a percentage of the overall cost. 
A similar argument applies to the inner product where the spectral has a $2N$ scaling advantage over the time formulation.  
As a result, a much more scalable method is obtained without the need for building a separate partitioning strategy that exploits the frequency domain for parallelization. 

\section{Conclusions} \label{sec:conclusion}

A time-spectral finite element method is proposed for simulating physiologically stable cardiorespiratory flows. 
This method, involving the addition of a least-squares penalty term to baseline Galerkin's method, is constructed by generalizing an earlier method designed for the convection-diffusion equation in steady flows. 
The least-squares penalty term in the GLS method is weighted by the tensorial stabilization parameter $\bl \tau$ that is calculated by solving an eigenvalue problem at each Gauss point. 
The proposed GLS method reduces to its temporal counterpart when modeling steady flows and retains its attractive properties, as formally analyzed via Theorem \ref{thm:error} to provide an error estimate for this method in various regimes.

Testing the proposed method using a patient-specific test case showed that a reasonably accurate solution, which closely aligns with the time formulation, can be obtained using as few as $N=7$ modes. 
For this test case, the proposed method provides a substantial cost advantage over the time formulation, reducing solution turnover time from 40 to approximately 4 hours. 
Additionally, the time-spectral method improves parallel scalability by a factor of $O(N)$, allowing for more efficient solutions on high-performance computing platforms.

\bibliographystyle{unsrt}
\bibliography{main}

\begin{thebibliography}{100}

\bibitem{taylor2009patient}
Charles~A Taylor and CA~Figueroa.
\newblock Patient-specific modeling of cardiovascular mechanics.
\newblock {\em Annual review of biomedical engineering}, 11:109--134, 2009.

\bibitem{Yang20102135}
Weiguang Yang, Jeffrey~A Feinstein, and Alison~L Marsden.
\newblock Constrained optimization of an idealized {Y}-shaped baffle for the
  {F}ontan surgery at rest and exercise.
\newblock {\em Computer Methods in Applied Mechanics and Engineering},
  199(33-36):2135--2149, 2010.

\bibitem{Marsden20081890}
Alison~L Marsden, Jeffrey~A Feinstein, and Charles~A Taylor.
\newblock A computational framework for derivative-free optimization of
  cardiovascular geometries.
\newblock {\em Computer Methods in Applied Mechanics and Engineering},
  197(21-24):1890--1905, 2008.

\bibitem{soerensen2007introduction}
Dennis~D Soerensen, Kerem Pekkan, Diane de~Z{\'e}licourt, Shiva Sharma, Kirk
  Kanter, Mark Fogel, and Ajit~P Yoganathan.
\newblock Introduction of a new optimized total cavopulmonary connection.
\newblock {\em The Annals of Thoracic Surgery}, 83(6):2182--2190, 2007.

\bibitem{blalock1945surgical}
Alfred Blalock and Helen~B Taussig.
\newblock The surgical treatment of malformations of the heart.
\newblock {\em Journal of the American Medical Association}, 128(3):189--202,
  1945.

\bibitem{Norwood1980hypoplastic}
William~I Norwood, James~K Kirklin, and Stephen~P Sanders.
\newblock Hypoplastic left heart syndrome: experience with palliative surgery.
\newblock {\em The American Journal of Cardiology}, 45(1):87--91, 1980.

\bibitem{Esmaily2015ABG}
Mahdi Esmaily, Tain-Yen Hsia, and Alison Marsden.
\newblock The assisted bidirectional {G}lenn: a novel surgical approach for
  first stage single ventricle heart palliation.
\newblock {\em The Journal of Thoracic and Cardiovascular Surgery},
  149(3):699--705, 2015.

\bibitem{Esmaily2012optimization}
Mahdi Esmaily, Francesco Migliavacca, Irene Vignon-Clementel, Tain-Yen Hsia,
  and Alison Marsden.
\newblock Optimization of shunt placement for the {N}orwood surgery using
  multi-domain modeling.
\newblock {\em Journal of Biomechanical Engineering}, 134(5):051002, 2012.

\bibitem{verma2018optimization}
Aekaansh Verma, Mahdi Esmaily, Jessica Shang, Richard Figliola, Jeffrey~A
  Feinstein, Tain-Yen Hsia, and Alison~L Marsden.
\newblock Optimization of the assisted bidirectional {G}lenn procedure for
  first stage single ventricle repair.
\newblock {\em World Journal for Pediatric and Congenital Heart Surgery},
  9(2):157--170, 2018.

\bibitem{driessen2019comparison}
Roel~S Driessen, Ibrahim Danad, Wijnand~J Stuijfzand, Pieter~G Raijmakers,
  Stefan~P Schumacher, Pepijn~A van Diemen, Jonathon~A Leipsic, Juhani Knuuti,
  S~Richard Underwood, Peter~M van~de Ven, et~al.
\newblock Comparison of coronary computed tomography angiography, fractional
  flow reserve, and perfusion imaging for ischemia diagnosis.
\newblock {\em Journal of the American College of Cardiology}, 73(2):161--173,
  2019.

\bibitem{modi2019predicting}
Bhavik~N Modi, Sethuraman Sankaran, Hyun~Jin Kim, Howard Ellis, Campbell
  Rogers, Charles~A Taylor, Ronak Rajani, and Divaka Perera.
\newblock Predicting the physiological effect of revascularization in serially
  diseased coronary arteries: clinical validation of a novel {CT} coronary
  angiography--based technique.
\newblock {\em Circulation: cardiovascular interventions}, 12(2):e007577, 2019.

\bibitem{hughes1987recent}
Thomas~JR Hughes.
\newblock Recent progress in the development and understanding of {SUPG}
  methods with special reference to the compressible {E}uler and
  {N}avier-{S}tokes equations.
\newblock {\em International journal for numerical methods in fluids},
  7(11):1261--1275, 1987.

\bibitem{franca1992stabilized}
Leopoldo~P Franca and S{\'e}rgio~L Frey.
\newblock Stabilized finite element methods: {II}. the incompressible
  {Navier-Stokes} equations.
\newblock {\em Computer Methods in Applied Mechanics and Engineering},
  99(2-3):209--233, 1992.

\bibitem{hughes1995multiscale}
Thomas~JR Hughes.
\newblock Multiscale phenomena: Green's functions, the {D}irichlet-to-{N}eumann
  formulation, subgrid scale models, bubbles and the origins of stabilized
  methods.
\newblock {\em Computer methods in applied mechanics and engineering},
  127(1-4):387--401, 1995.

\bibitem{hauke1994unified}
G~Hauke and TJR Hughes.
\newblock A unified approach to compressible and incompressible flows.
\newblock {\em Computer Methods in Applied Mechanics and Engineering},
  113(3-4):389--395, 1994.

\bibitem{codina2000stabilization}
Ramon Codina.
\newblock Stabilization of incompressibility and convection through orthogonal
  sub-scales in finite element methods.
\newblock {\em Computer methods in applied mechanics and engineering},
  190(13-14):1579--1599, 2000.

\bibitem{hughes2010stabilized}
Thomas~JR Hughes, Guglielmo Scovazzi, and Tayfun~E Tezduyar.
\newblock Stabilized methods for compressible flows.
\newblock {\em Journal of Scientific Computing}, 43:343--368, 2010.

\bibitem{hughes1979multidimentional}
Thomas~JR Hughes.
\newblock A multidimentional upwind scheme with no crosswind diffusion.
\newblock {\em Finite element methods for convection dominated flows, AMD 34},
  1979.

\bibitem{brooks1982streamline}
Alexander~N Brooks and Thomas~JR Hughes.
\newblock Streamline upwind/{Petrov-Galerkin formulations for convection
  dominated flows with particular emphasis on the incompressible Navier-Stokes
  equations}.
\newblock {\em Computer methods in applied mechanics and engineering},
  32(1-3):199--259, 1982.

\bibitem{ladyzhenskaya1969mathematical}
Olga~Alexandrovna Ladyzhenskaya.
\newblock {\em The mathematical theory of viscous incompressible flow},
  volume~2.
\newblock Gordon and Breach New York, 1969.

\bibitem{babuvska1971error}
Ivo Babuska.
\newblock Error-bounds for finite element method.
\newblock {\em Numerische Mathematik}, 16(4):322--333, 1971.

\bibitem{brezzi1974existence}
Franco Brezzi.
\newblock On the existence, uniqueness and approximation of saddle-point
  problems arising from {L}agrangian multipliers.
\newblock {\em ESAIM: Mathematical Modelling and Numerical
  Analysis-Mod{\'e}lisation Math{\'e}matique et Analyse Num{\'e}rique},
  8(R2):129--151, 1974.

\bibitem{hughes1986circumventing}
Thomas~JR Hughes, Leopoldo~P Franca, and Marc Balestra.
\newblock A new finite element formulation for computational fluid dynamics:
  {V}. circumventing the {B}abu{\v{s}}ka-{B}rezzi condition: {A} stable
  {P}etrov-{G}alerkin formulation of the {S}tokes problem accommodating
  equal-order interpolations.
\newblock {\em Computer Methods in Applied Mechanics and Engineering},
  59(1):85--99, 1986.

\bibitem{tezduyar1991stabilized}
Tayfun~E Tezduyar.
\newblock Stabilized finite element formulations for incompressible flow
  computations.
\newblock In {\em Advances in applied mechanics}, volume~28, pages 1--44.
  Elsevier, 1991.

\bibitem{hughes1986generalized}
Thomas~JR Hughes and Michel Mallet.
\newblock A new finite element formulation for computational fluid dynamics:
  {III}. the generalized streamline operator for multidimensional
  advective-diffusive systems.
\newblock {\em Computer methods in applied mechanics and engineering},
  58(3):305--328, 1986.

\bibitem{shakib1991new}
Farzin Shakib, Thomas~JR Hughes, and Zden{\v{e}}k Johan.
\newblock A new finite element formulation for computational fluid dynamics:
  {X}. the compressible {E}uler and {N}avier-{S}tokes equations.
\newblock {\em Computer Methods in Applied Mechanics and Engineering},
  89(1-3):141--219, 1991.

\bibitem{hauke1998comparative}
Guillermo Hauke and Thomas~JR Hughes.
\newblock A comparative study of different sets of variables for solving
  compressible and incompressible flows.
\newblock {\em Computer methods in applied mechanics and engineering},
  153(1-2):1--44, 1998.

\bibitem{hughes1989new}
Thomas~JR Hughes, Leopoldo~P Franca, and Gregory~M Hulbert.
\newblock A new finite element formulation for computational fluid dynamics:
  {VIII}. the {G}alerkin/least-squares method for advective-diffusive
  equations.
\newblock {\em Computer methods in applied mechanics and engineering},
  73(2):173--189, 1989.

\bibitem{shakib1989finite}
Farzin Shakib.
\newblock {\em Finite element analysis of the compressible {E}uler and
  {N}avier-{S}tokes equations}.
\newblock Stanford University, 1989.

\bibitem{meng2020time}
Chenwei Meng, Anirban Bhattacharjee, and Mahdi Esmaily.
\newblock A scalable spectral {S}tokes solver for simulation of time-periodic
  flows in complex geometries.
\newblock {\em Journal of Computational Physics}, 445:110601, 2021.

\bibitem{esmaily2023stabilized}
Mahdi Esmaily and Dongjie Jia.
\newblock A stabilized formulation for the solution of the incompressible
  unsteady stokes equations in the frequency domain.
\newblock {\em Journal of Computational Physics}, 473:111736, 2023.

\bibitem{esmaily2023stabilizedA}
Mahdi Esmaily and Dongjie Jia.
\newblock Stabilized finite element methods for the time-spectral
  convection-diffusion equation.
\newblock {\em arXiv preprint arXiv:2305.12038}, 2023.

\bibitem{jameson2002application}
Antony Jameson, J~Alonso, and M~McMullen.
\newblock Application of a non-linear frequency domain solver to the {E}uler
  and {N}avier-{S}tokes equations.
\newblock In {\em 40th AIAA aerospace sciences meeting \& exhibit}, page 120,
  2002.

\bibitem{mcmullen2003application}
Matthew~Scott McMullen.
\newblock {\em The application of non-linear frequency domain methods to the
  {Euler and Navier-Stokes} equations}.
\newblock Citeseer, 2003.

\bibitem{gopinath2005time}
Arathi Gopinath and Antony Jameson.
\newblock Time spectral method for periodic unsteady computations over two- and
  three-dimensional bodies.
\newblock In {\em 43rd AIAA aerospace sciences meeting and exhibit}, page 1220,
  2005.

\bibitem{mcmullen2001acceleration}
Matthew McMullen, Antony Jameson, and Juan Alonso.
\newblock Acceleration of convergence to a periodic steady state in
  turbomachinery flows.
\newblock In {\em 39th aerospace sciences meeting and exhibit}, page 152, 2001.

\bibitem{gopinath2007three}
Arathi Gopinath, Edwin van~der Weide, Juan Alonso, Antony Jameson, Kivanc
  Ekici, and Kenneth Hall.
\newblock Three-dimensional unsteady multi-stage turbomachinery simulations
  using the harmonic balance technique.
\newblock In {\em 45th AIAA Aerospace Sciences Meeting and Exhibit}, page 892,
  2007.

\bibitem{he2002analysis}
L~He, T~Chen, RG~Wells, YS~Li, and W~Ning.
\newblock Analysis of rotor-rotor and stator-stator interferences in
  multi-stage turbomachines.
\newblock In {\em Turbo Expo: Power for Land, Sea, and Air}, volume 3610, pages
  287--298, 2002.

\bibitem{sicot2012time}
Frederic Sicot, Guillaume Dufour, and Nicolas Gourdain.
\newblock A time-domain harmonic balance method for rotor/stator interactions.
\newblock {\em J. Turbomach.}, 134(1):011001, 2012.

\bibitem{arbenz2017comparison}
Peter Arbenz, Daniel Hupp, and Dominik Obrist.
\newblock Comparison of parallel time-periodic {N}avier-{S}tokes solvers.
\newblock In {\em International Conference on Parallel Processing and Applied
  Mathematics}, pages 57--67. Springer, 2017.

\bibitem{hupp2016parallel}
Daniel Hupp, Peter Arbenz, and Dominik Obrist.
\newblock A parallel {N}avier--{S}tokes solver using spectral discretisation in
  time.
\newblock {\em International journal of computational fluid dynamics},
  30(7-10):489--494, 2016.

\bibitem{hall2013harmonic}
Kenneth~C Hall, Kivanc Ekici, Jeffrey~P Thomas, and Earl~H Dowell.
\newblock Harmonic balance methods applied to computational fluid dynamics
  problems.
\newblock {\em International Journal of Computational Fluid Dynamics},
  27(2):52--67, 2013.

\bibitem{woiwode2020comparison}
Lukas Woiwode, Nidish~Narayanaa Balaji, Jonas Kappauf, Fabia Tubita, Louis
  Guillot, Christophe Vergez, Bruno Cochelin, Aur{\'e}lien Grolet, and Malte
  Krack.
\newblock Comparison of two algorithms for harmonic balance and path
  continuation.
\newblock {\em Mechanical Systems and Signal Processing}, 136:106503, 2020.

\bibitem{hall2002computation}
Kenneth~C Hall, Jeffrey~P Thomas, and William~S Clark.
\newblock Computation of unsteady nonlinear flows in cascades using a harmonic
  balance technique.
\newblock {\em AIAA Journal}, 40(5):879--886, 2002.

\bibitem{rigas2021nonlinear}
Georgios Rigas, Denis Sipp, and Tim Colonius.
\newblock Nonlinear input/output analysis: application to boundary layer
  transition.
\newblock {\em Journal of Fluid Mechanics}, 911:A15, 2021.

\bibitem{bermejo2015clinical}
Javier Bermejo, Pablo Mart{\'\i}nez-Legazpi, and Juan~C del {\'A}lamo.
\newblock The clinical assessment of intraventricular flows.
\newblock {\em Annual Review of Fluid Mechanics}, 47:315--342, 2015.

\bibitem{mittal2001application}
R~Mittal, SP~Simmons, and HS~Udaykumar.
\newblock Application of large-eddy simulation to the study of pulsatile flow
  in a modeled arterial stenosis.
\newblock {\em Journal of Biomechanical Engineering}, 123(4):325--332, 2001.

\bibitem{updegrove2017simvascular}
Adam Updegrove, Nathan~M Wilson, Jameson Merkow, Hongzhi Lan, Alison~L Marsden,
  and Shawn~C Shadden.
\newblock Simvascular: An open source pipeline for cardiovascular simulation.
\newblock {\em Annals of Biomedical Engineering}, 45(3):525--541, 2017.

\bibitem{jia2022characterization}
Dongjie Jia and Mahdi Esmaily.
\newblock Characterization of the ejector pump performance for the assisted
  bidirectional {G}lenn procedure.
\newblock {\em Fluids}, 7(1):31, 2022.

\bibitem{jia2021efficient}
Dongjie Jia, Matthew Peroni, Tigran Khalapyan, and Mahdi Esmaily.
\newblock An efficient assisted bidirectional glenn design with lowered
  superior vena cava pressure for stage-one single ventricle patients.
\newblock {\em Journal of Biomechanical Engineering}, 143(7):071008, 2021.

\bibitem{Lagana2002359}
Katia Lagana, G~Dubini, Francesco Migliavacca, Riccardo Pietrabissa, Giancarlo
  Pennati, Alessandro Veneziani, and A~Quarteroni.
\newblock Multiscale modelling as a tool to prescribe realistic boundary
  conditions for the study of surgical procedures.
\newblock {\em Biorheology}, 39:359--364, 2002.

\bibitem{Formaggia2002}
Luca Formaggia, Jean-Frederic Gerbeau, Fabio Nobile, and Alfio Quarteroni.
\newblock Numerical treatment of defective boundary conditions for the
  {N}avier-{S}tokes equations.
\newblock {\em {SIAM} Journal on Numerical Analysis}, 40:376--401, 2002.

\bibitem{Vignon20063776}
Irene~E Vignon-Clementel, C~Alberto Figueroa, Kenneth~E Jansen, and Charles~A
  Taylor.
\newblock Outflow boundary conditions for three-dimensional finite element
  modeling of blood flow and pressure in arteries.
\newblock {\em Computer Methods in Applied Mechanics and Engineering},
  195(29-32):3776--3796, 2006.

\bibitem{Vignon2010625}
I.E. Vignon-Clementel, C.A. Figueroa, K.E. Jansen, and C.A. Taylor.
\newblock Outflow boundary conditions for three-dimensional simulations of
  non-periodic blood flow and pressure fields in deformable arteries.
\newblock {\em Computer Methods in Biomechanics and Biomedical Engineering},
  13(5):625--640, 2010.

\bibitem{Esmaily2013coupling}
Mahdi Esmaily, Irene Vignon-Clementel, Richard Figliola, and Alison Marsden.
\newblock A modular numerical method for implicit {0D/3D} coupling in
  cardiovascular finite element simulations.
\newblock {\em Journal of Computational Physics}, 224:63--79, 2013.

\bibitem{Bazilevs20083}
Y.~Bazilevs, V.M. Calo, T.J.R. Hughes, and Y.~Zhang.
\newblock Isogeometric fluid-structure interaction: theory, algorithms, and
  computations.
\newblock {\em Computational Mechanics}, 43:3--37, 2008.

\bibitem{Bazilevs2009computational}
Y.~Bazilevs, M.C. Hsu, D.~Benson, S.~Sankaran, and A.L. Marsden.
\newblock Computational fluid-structure interaction: methods and application to
  a total cavopulmonary connection.
\newblock {\em Computational Mechanics}, 45:77--89, 2009.

\bibitem{long2013fluid}
CC~Long, AL~Marsden, and Y~Bazilevs.
\newblock Fluid--structure interaction simulation of pulsatile ventricular
  assist devices.
\newblock {\em Computational Mechanics}, 52(5):971--981, 2013.

\bibitem{kamensky2015immersogeometric}
David Kamensky, Ming-Chen Hsu, Dominik Schillinger, John~A Evans, Ankush
  Aggarwal, Yuri Bazilevs, Michael~S Sacks, and Thomas~JR Hughes.
\newblock An immersogeometric variational framework for fluid--structure
  interaction: Application to bioprosthetic heart valves.
\newblock {\em Computer methods in applied mechanics and engineering},
  284:1005--1053, 2015.

\bibitem{Esmaily2012multipleSPS}
Mahdi Esmaily, Bari Murtuza, T.Y. Hsia, and Alison Marsden.
\newblock Simulations reveal adverse hemodynamics in patients with multiple
  systemic to pulmonary shunts.
\newblock {\em Journal of Biomechanical Engineering},
  137(3):031001--031001--12, 2015.

\bibitem{arzani2014longitudinal}
Amirhossein Arzani, Ga-Young Suh, Ronald~L Dalman, and Shawn~C Shadden.
\newblock A longitudinal comparison of hemodynamics and intraluminal thrombus
  deposition in abdominal aortic aneurysms.
\newblock {\em American Journal of Physiology-Heart and Circulatory
  Physiology}, 307(12):H1786--H1795, 2014.

\bibitem{Esmaily2013RT}
Mahdi Esmaily, Tain-Yen Hsia, and Alison Marsden.
\newblock A non-discrete method for computation of residence time in fluid
  mechanics simulations.
\newblock {\em Physics of Fluids}, 25(11):110802--21, 2013.

\bibitem{shadden2013potential}
Shawn~C Shadden and Sahar Hendabadi.
\newblock Potential fluid mechanic pathways of platelet activation.
\newblock {\em Biomechanics and modeling in mechanobiology}, 12:467--474, 2013.

\bibitem{rydquist2022cell}
Grant Rydquist and Mahdi Esmaily.
\newblock A cell-resolved, {L}agrangian solver for modeling red blood cell
  dynamics in macroscale flows.
\newblock {\em Journal of Computational Physics}, 461:111204, 2022.

\bibitem{wu2015coupled}
Jiacheng Wu and Shawn~C Shadden.
\newblock Coupled simulation of hemodynamics and vascular growth and remodeling
  in a subject-specific geometry.
\newblock {\em Annals of biomedical engineering}, 43(7):1543--1554, 2015.

\bibitem{bangalore2017towards}
Abhay Bangalore~Ramachandra.
\newblock {\em Towards a Fluid-Structure-Growth and Remodeling Framework to
  Simulate Vein Graft Failure Post Coronary Artery Bypass Surgery}.
\newblock PhD thesis, UC San Diego, 2017.

\bibitem{coogan2013computational}
Jessica~S Coogan, Jay~D Humphrey, and C~Alberto Figueroa.
\newblock Computational simulations of hemodynamic changes within thoracic,
  coronary, and cerebral arteries following early wall remodeling in response
  to distal aortic coarctation.
\newblock {\em Biomechanics and modeling in mechanobiology}, 12(1):79--93,
  2013.

\bibitem{bazilevs2007variational}
Y~Bazilevs, VM~Calo, JA~Cottrell, TJR Hughes, A~Reali, and G~Scovazzi.
\newblock Variational multiscale residual-based turbulence modeling for large
  eddy simulation of incompressible flows.
\newblock {\em Computer Methods in Applied Mechanics and Engineering},
  197(1-4):173--201, 2007.

\bibitem{babuska1997pollution}
Ivo~M Babuska and Stefan~A Sauter.
\newblock Is the pollution effect of the {FEM} avoidable for the {H}elmholtz
  equation considering high wave numbers?
\newblock {\em SIAM Journal on numerical analysis}, 34(6):2392--2423, 1997.

\bibitem{harari1990design}
Isaac Harari and Thomas~JR Hughes.
\newblock Design and analysis of finite element methods for the {H}elmholtz
  equation in exterior domains.
\newblock {\em Appl. Mech. Rev.}, 43(5):366--373, 1990.

\bibitem{thompson1995galerkin}
Lonny~L Thompson and Peter~M Pinsky.
\newblock A {G}alerkin least-squares finite element method for the
  two-dimensional {H}elmholtz equation.
\newblock {\em International Journal for numerical methods in engineering},
  38(3):371--397, 1995.

\bibitem{noschese2013tridiagonal}
Silvia Noschese, Lionello Pasquini, and Lothar Reichel.
\newblock Tridiagonal {T}oeplitz matrices: properties and novel applications.
\newblock {\em Numerical linear algebra with applications}, 20(2):302--326,
  2013.

\bibitem{galeao2004finite}
AC~Galeao, RC~Almeida, SMC Malta, and AFD Loula.
\newblock Finite element analysis of convection dominated reaction--diffusion
  problems.
\newblock {\em Applied Numerical Mathematics}, 48(2):205--222, 2004.

\bibitem{brezzi1990discourse}
Franco Brezzi and Klaus-J{\"u}rgen Bathe.
\newblock A discourse on the stability conditions for mixed finite element
  formulations.
\newblock {\em Computer methods in applied mechanics and engineering},
  82(1-3):27--57, 1990.

\bibitem{franca1991error}
Leopoldo~P Franca and Rolf Stenberg.
\newblock Error analysis of {G}alerkin least squares methods for the elasticity
  equations.
\newblock {\em SIAM Journal on Numerical Analysis}, 28(6):1680--1697, 1991.

\bibitem{franca1993convergence}
Leopoldo~P Franca and Thomas~JR Hughes.
\newblock Convergence analyses of {G}alerkin least-squares methods for
  symmetric advective-diffusive forms of the {S}tokes and incompressible
  {N}avier-{S}tokes equations.
\newblock {\em Computer Methods in Applied Mechanics and Engineering},
  105(2):285--298, 1993.

\bibitem{knobloch2011stability}
Petr Knobloch and Lutz Tobiska.
\newblock On the stability of finite-element discretizations of
  convection--diffusion--reaction equations.
\newblock {\em IMA Journal of numerical analysis}, 31(1):147--164, 2011.

\bibitem{Hughes1986beyond}
T.J.R. Hughes, M.~Mallet, and M.~Akira.
\newblock A new finite element formulation for computational fluid dynamics:
  {II}. beyond {SUPG}.
\newblock {\em Computer Methods in Applied Mechanics and Engineering},
  54(3):341--355, 1986.

\bibitem{johnson1984finite}
Claes Johnson, Uno Navert, and Juhani Pitkaranta.
\newblock Finite element methods for linear hyperbolic problems.
\newblock {\em Computer methods in applied mechanics and engineering},
  45:285--312, 1984.

\bibitem{Bazilevs20093534}
Y.~Bazilevs, J.R. Gohean, T.J.R. Hughes, R.D. Moser, and Y.~Zhang.
\newblock Patient-specific isogeometric fluid-structure interaction analysis of
  thoracic aortic blood flow due to implantation of the {J}arvik 2000 left
  ventricular assist device.
\newblock {\em Computer Methods in Applied Mechanics and Engineering},
  198(45-46):3534--3550, 2009.

\bibitem{womersley1955method}
John~R Womersley.
\newblock Method for the calculation of velocity, rate of flow and viscous drag
  in arteries when the pressure gradient is known.
\newblock {\em The Journal of Physiology}, 127(3):553, 1955.

\bibitem{barbosa1991finite}
Helio~JC Barbosa and Thomas~JR Hughes.
\newblock The finite element method with {L}agrange multipliers on the
  boundary: circumventing the {Babu{\v{s}}ka-Brezzi} condition.
\newblock {\em Computer Methods in Applied Mechanics and Engineering},
  85(1):109--128, 1991.

\bibitem{hughes2007variational}
Thomas~JR Hughes and Giancarlo Sangalli.
\newblock Variational multiscale analysis: the fine-scale {G}reen’s function,
  projection, optimization, localization, and stabilized methods.
\newblock {\em SIAM Journal on Numerical Analysis}, 45(2):539--557, 2007.

\bibitem{saad1986gmres}
Youcef Saad and Martin~H Schultz.
\newblock {GMRES}: A generalized minimal residual algorithm for solving
  nonsymmetric linear systems.
\newblock {\em SIAM Journal on scientific and statistical computing},
  7(3):856--869, 1986.

\bibitem{fried1972bounds}
ISAAC Fried.
\newblock Bounds on the extremal eigenvalues of the finite element stiffness
  and mass matrices and their spectral condition number.
\newblock {\em Journal of Sound and Vibration}, 22(4):407--418, 1972.

\bibitem{jansen2000generalized}
Kenneth~E Jansen, Christian~H Whiting, and Gregory~M Hulbert.
\newblock A generalized-$\alpha$ method for integrating the filtered
  {Navier--Stokes} equations with a stabilized finite element method.
\newblock {\em Computer methods in applied mechanics and engineering},
  190(3-4):305--319, 2000.

\bibitem{Esmaily2011backflow}
Mahdi Esmaily, Yuri Bazilevs, Tain-Yen Hsia, Irene Vignon-Clementel, and Alison
  Marsden.
\newblock A comparison of outlet boundary treatments for prevention of backflow
  divergence with relevance to blood flow simulations.
\newblock {\em Computational Mechanics}, 48(3):277--291, 2011.

\bibitem{bertoglio2018benchmark}
Crist{\'o}bal Bertoglio, Alfonso Caiazzo, Yuri Bazilevs, Malte Braack, Mahdi
  Esmaily, Volker Gravemeier, Alison L.~Marsden, Olivier Pironneau, Irene
  E.~Vignon-Clementel, and Wolfgang A.~Wall.
\newblock Benchmark problems for numerical treatment of backflow at open
  boundaries.
\newblock {\em International Journal for Numerical Methods in Biomedical
  Engineering}, 34(2):e2918, 2018.

\bibitem{kim20093551}
H.J. Kim, C.A. Figueroa, T.J.R. Hughes, K.E. Jansen, and C.A. Taylor.
\newblock Augmented {L}agrangian method for constraining the shape of velocity
  profiles at outlet boundaries for three-dimensional finite element
  simulations of blood flow.
\newblock {\em Computer Methods in Applied Mechanics and Engineering},
  198(45-46):3551--3566, 2009.

\bibitem{wilson2013vascular}
Nathan~M Wilson, Ana~K Ortiz, and Allison~B Johnson.
\newblock The vascular model repository: A public resource of medical imaging
  data and blood flow simulation results.
\newblock {\em Journal of Medical Devices}, 7(4):040923, 2013.

\bibitem{marsden2010new}
Alison~L Marsden, V~Mohan Reddy, Shawn~C Shadden, Frandics~P Chan, Charles~A
  Taylor, and Jeffrey~A Feinstein.
\newblock A new multiparameter approach to computational simulation for fontan
  assessment and redesign.
\newblock {\em Congenital Heart Disease}, 5(2):104--117, 2010.

\bibitem{duncan2003pulmonary}
Brian~W Duncan and Shailesh Desai.
\newblock Pulmonary arteriovenous malformations after cavopulmonary
  anastomosis.
\newblock {\em The Annals of thoracic surgery}, 76(5):1759--1766, 2003.

\bibitem{yang2013optimization}
Weiguang Yang, Jeffrey~A Feinstein, Shawn~C Shadden, Irene~E Vignon-Clementel,
  and Alison~L Marsden.
\newblock Optimization of a {Y}-graft design for improved hepatic flow
  distribution in the fontan circulation.
\newblock {\em Journal of biomechanical engineering}, 135(1):011002, 2013.

\bibitem{javadi2022predicting}
Elahe Javadi, Sebastian Laudenschlager, Vitaly Kheyfets, Michael Di~Maria,
  Matthew Stone, Safa Jamali, Andrew~J Powell, and Mehdi~H Moghari.
\newblock Predicting hemodynamic performance of {F}ontan operation for {G}lenn
  physiology using computational fluid dynamics: Ten patient-specific cases.
\newblock {\em Journal of clinical images and medical case reports}, 3(6),
  2022.

\bibitem{mupfes}
M.~Esmaily.
\newblock Multi-physics finite element solver (mupfes), May 2014.

\bibitem{Esmaily2015BIPN}
Mahdi Esmaily, Yuri Bazilevs, and Alison Marsden.
\newblock A bi-partitioned iterative algorithm for solving linear systems
  arising from incompressible flow problems.
\newblock {\em Computer Methods in Applied Mechanics and Engineering},
  286(1):40--62, 2015.

\bibitem{shakib1989multi}
Farzin Shakib, Thomas~JR Hughes, and Zden{\v{e}}k Johan.
\newblock A multi-element group preconditioned {GMRES} algorithm for
  nonsymmetric systems arising in finite element analysis.
\newblock {\em Computer Methods in Applied Mechanics and Engineering},
  75(1-3):415--456, 1989.

\bibitem{rigas2022data}
G~Rigas and PJ~Schmid.
\newblock Data-driven closure of the harmonic-balanced navier-stokes equations
  in the frequency domain.
\newblock {\em Proceedings of the Summer Program}, pages 67--76, 2022.

\bibitem{Esmaily2015DS}
Mahdi Esmaily, Yuri Bazilevs, and Alison Marsden.
\newblock Impact of data distribution on the parallel performance of iterative
  linear solvers with emphasis on {CFD} of incompressible flows.
\newblock {\em Computational Mechanics}, 55(1):93--103, 2015.

\bibitem{bazilevs2012computational}
Yuri Bazilevs, Kenji Takizawa, and Tayfun~E Tezduyar.
\newblock {\em Computational fluid-structure interaction: methods and
  applications}.
\newblock Wiley, 2013.

\bibitem{hsu2010improving}
M-C Hsu, Yuri Bazilevs, Victor~M Calo, Tayfun~E Tezduyar, and Thomas~JR Hughes.
\newblock Improving stability of stabilized and multiscale formulations in flow
  simulations at small time steps.
\newblock {\em Computer Methods in Applied Mechanics and Engineering},
  199(13-16):828--840, 2010.

\bibitem{codina2007time}
Ramon Codina, Javier Principe, Oriol Guasch, and Santiago Badia.
\newblock Time dependent subscales in the stabilized finite element
  approximation of incompressible flow problems.
\newblock {\em Computer Methods in Applied Mechanics and Engineering},
  196(21-24):2413--2430, 2007.

\bibitem{jia2023time}
Dongjie Jia and Mahdi Esmaily.
\newblock A time-consistent stabilized finite element method for fluids with
  applications to hemodynamics.
\newblock {\em Scientific Reports}, 13(1):19120, 2023.

\bibitem{evans2018residual}
John~A. Evans, Christopher Coley, Ryan~M. Aronson, Corey~L. Wetterer-Nelson,
  and Yuri Bazilevs.
\newblock {\em Residual-Based Large Eddy Simulation with Isogeometric
  Divergence-Conforming Discretizations}, pages 91--130.
\newblock Springer International Publishing, 2018.

\end{thebibliography}


\appendix
\section{Properties of the GLS method for the Navier-Stokes equations} \label{sec:prop-ns}
In Section \ref{sec:analysis}, we analyzed the properties of the GLS method when it is applied to the time-spectral form of the convection-diffusion equation. 
As was shown in \cite{hughes1989new}, such an analysis can be extended to the Navier-Stokes equations if its governing equations were to be cast in the form of a system of convection-diffusion equations. 

The process is as follows. Assuming that the number of spatial dimensions is 3, we define
\begin{equation*}
\bl \psi  \coloneqq  \left[\begin{matrix}
       \rho \bl u_1  \\     
       \rho \bl u_2 \\     
       \rho \bl u_3 \\     
       \bl p \\     
       \end{matrix} \right], \;\; 
\bl M  \coloneqq  \left[\begin{matrix}
       \bl \Omega & \bl 0 & \bl 0 & \bl 0 \\     
       \bl 0 & \bl \Omega & \bl 0 & \bl 0 \\     
       \bl 0 & \bl 0 & \bl \Omega & \bl 0 \\     
       \bl 0 & \bl 0 & \bl 0 & \bl 0 \\     
       \end{matrix} \right], \;\;
\bl B_i \coloneqq  \left[\begin{matrix}
       \bl A_i & \bl 0 & \bl 0 & \delta_{1i}\bl I \\     
       \bl 0 & \bl A_i & \bl 0 & \delta_{2i}\bl I \\     
       \bl 0 & \bl 0 & \bl A_i & \delta_{3i}\bl I \\     
       \delta_{1i}\bl I & \delta_{2i}\bl I & \delta_{3i}\bl I & \bl 0 \\     
       \end{matrix} \right], \;\;
\bl \kappa \coloneqq  \mu \left[\begin{matrix}
       \bl I & \bl 0 & \bl 0 & \bl 0 \\     
       \bl 0 & \bl I & \bl 0 & \bl 0 \\     
       \bl 0 & \bl 0 & \bl I & \bl 0 \\     
       \bl 0 & \bl 0 & \bl 0 & \bl 0 \\     
       \end{matrix} \right], 
\end{equation*}
as the new state variable and acceleration, convection, and diffusion matrices, respectively. 
With these, the momentum and continuity equations in \eqref{NS-spec} can be formulated as 
\begin{equation*}
    \bl M \bl \psi + \bl B_i \frac{\partial \bl \psi}{\partial x_i} = \frac{\partial }{\partial x_i }\left(\bl \kappa \frac{\partial \bl \psi}{\partial x_i} \right),
\end{equation*}
which has an identical form to \eqref{3D-spec}. Also, note that $\bl M$ and $\bl B_i$ have similar properties to $\bl \Omega$ and $\bl A_i$, respectively, namely $\bl M^H = - \bl M$ and $\bl B_i^H=\bl B_i$. Using this transformation, the GLS method for the Navier-Stokes in \eqref{NS-weak} can also be cast in the form of the GLS method for the convection-diffusion equation in \eqref{3D-weak}. That involves writing the stabilization matrix as a block diagonal matrix that is zero everywhere except for the top three diagonal entries where $\bl \tau$ appears. With this one-to-one correspondence between the two problems, one may follow a process similar to that of Section \ref{sec:analysis} to study the properties of the GLS method for the Navier-Stokes equations.

\section{The time formulation of the Navier-Stokes equations} \label{sec:app}
To simulate incompressible flows in the time domain directly, we must solve \eqref{NS-time} to find $\hat u_i(\bl x,t)$ and $\hat p(\bl x,t)$. 
As we will show later, the finite element method employed for this purpose is equivalent to the GLS method for steady flows, thus permitting an apple-to-apple comparison between the two. 
This method can be derived using the variational multiscale (VMS) method by dropping the unresolved subgrid scale pressure $\hat p^\prime$ from the formulation as well as jump conditions at the element boundaries. 
A more direct approach, however, is to use the least-squares framework and add a penalty term to the baseline Galerkin's method. 
Either way, the Eulerian acceleration term is excluded from this standard process~\cite{bazilevs2012computational}, marking a key difference between the time and spectral formulations in unsteady flows.

The method constructed this way is stated as finding $\hat u_i^h$ and $\hat p^h$ so that for any $w_i^h$ and $q^h$ we have
\begin{equation}
\begin{split}
       & \overbrace{\left(w^h_i,\rho \frac{\partial \hat u^h_i}{\partial t} + \rho \hat u_j^h \frac{\partial \hat u^h_i}{\partial x_j} \right)_\Omega + \left(\frac{\partial w^h_i}{\partial x_j}, -\hat p^h \delta_{ij} + \mu \frac{\partial\hat u^h_i}{\partial x_j}\right)_\Omega  + \left(q^h, \frac{\partial \hat u^h_i}{\partial x_i} \right)_\Omega  }^\text{Galerkin's terms} \\
       + & \underbrace{ \Big(\hat u_j^h \frac{\partial w^h_i}{\partial x_j} , \tau \hat r_i \Big)_\Ot}_\text{The SUPG terms} + \underbrace{ \Big(\frac{1}{\rho}\frac{\partial q^h}{\partial x_i} ,\tau \hat r_i \Big)_\Ot}_\text{The PSPG terms} = \underbrace{ \left( w_i^h,  \hat h n_i \right)_\Gh}_\text{Neumann BCs. term},
\end{split}
\label{NSt-weak}
\end{equation}
where  
\begin{equation*}
    \hat \rres_i(\hat u^h_i, \hat p^h) = \rho \frac{\partial \hat u_i^h}{\partial t} + \rho \hat u_j^h \frac{\partial \hat u^h_i}{\partial x_j} + \frac{\partial \hat p^h}{\partial x_i} - \frac{\partial}{\partial x_j} \left(\mu \frac{\partial \hat u^h_i}{\partial x_j}\right). 
\end{equation*}
The solution procedure for \eqref{NSt-weak} is similar to what was discussed earlier except for the fact that this equation is integrated in time using the second-order implicit generalized-$\alpha$ method~\cite{jansen2000generalized} with $\rho_{\infty} = 0.2$.
. 

One can readily verify that \eqref{NS-weak} reduces to \eqref{NSt-weak} for linear interpolation functions when $\partial^2 \bl w^h_i/\partial x_j^2 = \bl 0$ and in steady flows when $\omega=0$ and $\partial \hat u^h_i/\partial t = 0$. 
The remaining two terms in $\res_i (\bl w_i^h,\bl q_i^h)$ in \eqref{NS-weak} produce the SUPG and PSPG terms in \eqref{NSt-weak}, which stabilize the solution in the presence of strong convection and permit the use of equal order interpolation functions for velocity and pressure, respectively. 

The stabilization parameter $\tau$ in \eqref{NSt-weak} is computed using a relationship that is also very similar to its time-spectral version in \eqref{3D-tau}. That is  
\begin{equation}
    \tau = \left[ \hat \omega^2 + \hat u_i^h G_{ij} \hat u_j^h + C_{\rm I} \kappa^2 G_{ij}G_{ij} \right]^{-1/2}. 
    \label{3D-taut}
\end{equation}
Note that $\hat \omega$, which we define below, must be ideally dropped from \eqref{3D-taut} as it does not appear in the 1D model problem that originates this definition of $\tau$. 
Dropping it, however, will produce instabilities at a small time step size $\Delta t$. 
Those instabilities can be traced back to the acceleration term in the PSPG terms that generates an asymmetric contribution to the tangent matrix that is $(\alpha_m \tau/\Delta t) (\partial N_A/\partial x_i) N_B$, where $\alpha_m$ is an order 1 constant coming from the generalized-$\alpha$ time integration method. 
Since this term is proportional to $\Delta t^{-1}$, it causes convergence issues at a very small time step size. 

To remedy this issue, it is common to take $\hat \omega \propto \Delta t^{-1}$ in \eqref{3D-taut} (c.f.~\cite{bazilevs2012computational}). 
This way, as $\Delta t \to 0$, $\tau \propto \Delta t$ thereby preventing the growth of that asymmetric term in the tangent matrix as $\Delta t \to 0$. 
This solution, however, will create an inconsistent method with a solution that strongly depends on $\Delta t$, particularly when $\Delta t$ is very small~\cite{hsu2010improving, codina2007time, jia2023time}. 
Therefore, the resulting method can not be used as a reference to establish the accuracy of the proposed time-spectral method. 

To overcome the instability issue while ensuring the resulting method is accurate at small $\Delta t$, we resort to a different definition of $\hat \omega$ that is 
\begin{equation}
    \hat \omega = \norm{\frac{\partial \hat {\bl u}^h}{\partial t}}_\Omega \norm{\hat {\bl u}^h}_\Omega^{-1},
    \label{omegah}
\end{equation}
where $\hat {\bl u}^h \coloneqq [\hat u^h_1, \hat u^h_2, \hat u^h_3]^{T}$ in 3D. 
A slightly different version of this definition of $\hat \omega$ was originally proposed in \cite{evans2018residual}. 
The definition shown in \eqref{omegah} is proposed in \cite{jia2023time}. 

As studied at length in \cite{jia2023time}, this method is time-consistent and produces predictions that agree closely with the existing closed-form solution for laminar pipe flow. 
The resulting method is also more stable than taking $\hat \omega = 0$ because of a negative feedback loop. 
As the flow becomes unstable, the solution at the next step deviates from the previous time step, hence producing a very large $\partial u_i^h/\partial t$. 
That leads to a larger $\hat \omega$ in \eqref{omegah} and smaller $\tau$ in \eqref{3D-taut}, which in turn controls the size of the asymmetric term in the tangent matrix to recover stability.

\end{document}